\documentclass[hidelinks, 11pt,reqno]{amsart}
\usepackage[
  margin=1in,           % Equal margins on all sides
  textwidth=6.2in,      % Custom text width
  textheight=8.5in      % Custom text height
]{geometry}

\usepackage{amsmath, bm}
\usepackage[psamsfonts]{amssymb}
\usepackage{amsfonts}
\usepackage{algorithm2e}
\usepackage{graphicx}
\usepackage{verbatim}
\usepackage{dsfont}
\usepackage{tikz-cd}
\usepackage{mathrsfs}
\usepackage{color}
\usepackage{multicol}
 \usepackage{verbatim}

\usepackage[english]{babel}
\usepackage{fontenc}
\usepackage{indentfirst}
\usepackage{tikz}
\usetikzlibrary{matrix,arrows,decorations.pathmorphing}
\usepackage{algorithm2e}
\usepackage[all]{xy}
\usepackage[T1]{fontenc}
\usepackage{hyperref}
\usepackage{mathtools}
\usepackage{multicol}
\usepackage{float}
\usepackage{physics}
\usepackage{dsfont}
\usepackage{multirow}
\usetikzlibrary{tikzmark}
\usepackage{caption}
\usepackage{subcaption}
\usepackage{float}

\theoremstyle{definition}
\newtheorem{theorem}{Theorem}[section]
\newtheorem{prop}[theorem]{Proposition}
\newtheorem{lemma}[theorem]{Lemma}
\newtheorem{cor}[theorem]{Corollary}

\newtheorem{conj}[theorem]{Conjecture}
\newtheorem{question}[theorem]{Question}
\newtheorem{ex}[theorem]{Example}
\newtheorem{dfn}[theorem]{Definition}
\newtheorem{remark}[theorem]{Remark}

\newtheorem{thm}{Theorem}

\def\ep{\epsilon}

\def\R{\mathbb{R}}
\def\Z{\mathbb{Z}}
\def\N{\mathbb{N}}

\def\C{\mathbb{C}}

\def\F{\mathcal{F}}

\def\T{\mathbb T}
\def\Q{\mathbb Q}
\def\cal R{\mathcal R}
 \def\vol{\mathrm{vol}}
 \def\rhoT{\rho_{{\rm sys}, T^*\T^2}}
 \def\sys{\mathrm{sys}}

\begin{document}
% ---------------- Title block (manual; no preamble changes) 

\author{Jun Zhang}
\email{jzhang4518@ustc.edu.cn}
\address{The Institute of Geometry and Physics, University of Science and Technology of China, 96 Jinzhai Road, Hefei Anhui, 230026, China}

\author{Antong Zhu}
\email{antong@mail.ustc.edu.cn}
\address{School of Mathematical Sciences, University of Science and Technology of China, 96
Jinzhai Road, Hefei Anhui, 230026, China}

\title{Geometry and dynamics on Liouville domains in $T^*\T^2$}
   
\maketitle

\vspace{-5mm}
% ---------------- Abstract (inline style) ----------------
\begin{abstract}
Parallel to the study of toric domains, symplectically convex, and dynamically convex domains in $(\R^4, \omega_{\rm std})$, we build an analogous framework and corresponding subclasses for Liouville domains in $(T^*\T^2,\omega_{\rm can})$. A key feature of this framework is the introduction of a new notion of convexity, based on systolic ratios. Via various machinery in quantitative symplectic geometry, including ECH capacities, shape invariant, dynamical zeta function, etc., we investigate the relations between subclasses of Liouville domains in $T^*\T^2$, obtain large-scale geometry of Liouville domains in $T^*\T^2$ with respect to coarse
Banach-Mazur distance, provide a non-flat codisc bundle of  torus even under the action of exact symplectomorphisms, and verify the agreement of normalized capacities for a wide class of Liouville domains in $T^*\T^2$.
\end{abstract}
 
\tableofcontents

 \section{Introduction}\label{sec-intro}
Star-shaped domains (or called Liouville domains) in  $(\R^{2n}, \omega_{\rm std})$ have long served as a key object in symplectic geometry. Due to the complexity of a general star-shaped domain, usually people focus on subclasses with specific properties. For instance, given a domain $\Omega \subset \R_{\ge 0}^n$ as the closure of a non-empty open subset in $\R_{\geq 0}^n$, the associated toric domain is defined as follows:
\begin{equation}\label{eq-toric-domain}
    X_\Omega := \{z \in \C^n \, |\,  \mu(z)=\pi(|z_1|^2, \dots, |z_n|^2) \in \Omega\},
\end{equation}
which is a symplectic manifold with boundary equipped with the standard symplectic structure $\omega_{\rm std}.$  Moreover, a toric domain $X_\Omega$ is called \emph{monotone} if the outward normals at every point   in its profile hypersurface $\partial\Omega\cap \R_{>0}^n$ have nonnegative entries. 

For another instance, since the geometric convexity fails to be symplectically invariant, different replacements have been explored. One is symplectic convexity, that is, being geometrically convex up to symplectomorphism. The other is dynamical convexity, which requires an index-lower bound of closed orbits of the Reed dynamics on domain's (contact) boundary. In \cite{HWZ98}, Hofer-Wysocki-Zehnder shows that  symplectic convexity implies dynamical convexity, while the converse does not hold due to recent works \cite{CE22,CE25,DGZ24,DGRZ25}. In particular, the first such example in 4-dimension by Chaidez-Edtmair in \cite{CE22} is based on a dynamical criterion on convexity in \cite{CE22}.

A comparison for different subclasses of domains in $(\R^4, \omega_{\rm std})$ is established by Figure \ref{fig:comparison} (cf. Figure 1 in \cite{DGRZ25}).
In terms of notations, $\mathcal{T}_{\R^{2n}}$ denotes the set of toric domains, $\mathcal{M}$ the set of monotone toric domains, $\mathcal{C}$ the set of symplectically convex domains, and $\mathcal{D}_{\R^4}$ denotes  the set of dynamically convex domains. 

\medskip

\begin{figure}[h]
    \centering
    \includegraphics[width=0.8\linewidth]{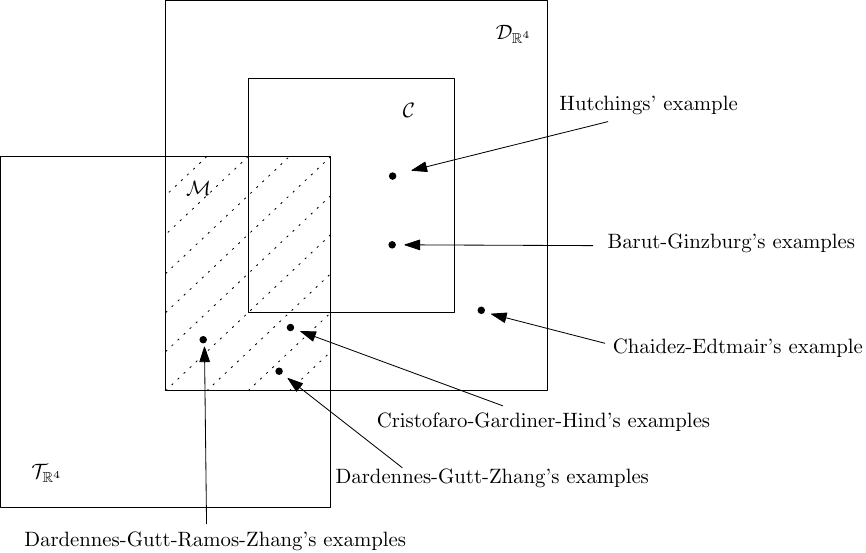}
    \caption{Relations between subclasses  of star-shaped domains in $\R^4$.}
    \label{fig:comparison}
\end{figure}
Examples that distinguish these subclasses have been provided by various references \cite{CE22,DGZ24,DGRZ25,DR23,Hut24,BG25}, which supports those labels in Figure \ref{fig:comparison}. In particular, discovery of these examples relies on various symplectic machinery including embedded contact homology (ECH), symplectic capacities, filtered symplectic homology, etc.

\medskip

In this paper, we aim to set up a parallel framework for fiberwise star-shaped domains (also called by Liouville domains) in $T^*\T^n$, mostly on $T^*\T^2$, where the results involving higher-dimensional $\T^n$ appear in Section \ref{ssec-large} and \ref{ssec-nor-cap}. 
Here, we always identify $\T^2$ with $\R/\Z \times \R/\Z$ (so the standard symplectic volume or area of $\T^2$ is $1$). Adapting Definition 1.1 in \cite{HN16} (see the beginning of Section \ref{sec-back-conx}), one can define dynamical convexity for domains in $T^*\T^2$. Similarly to $\mathcal{D}_{\R^4}$ in Figure \ref{fig:comparison}, denote by $\mathcal{D}_{T^*\T^2}$ the set of dynamically convex domains in $T^*\T^2$. 

Recall that  an (open) toric domain $X_\Omega\subset \R^4$, defined at the beginning of this section, is a star-shaped domain   equipped with a toric fibration whose regular fibers lie over the interior $\mathrm{int}(\Omega)$, while singular fibers occur over the intersections of $\Omega$ with coordinate axes. After removing these singular fibers, the resulting domain $X_{\mathrm{int}(\Omega)}$ becomes topologically non-trivial and can be identified (via symplectic embedding in  (\ref{eq-emb-one})) with   domain  $\T^2\times\mathrm{int}(\Omega)\subset (T^*\T^2,\omega_{\rm can})$. Motivated by this correspondence, we call domains of the form    $\T^2 \times A \subset T^*\T^2$ where $A\subset \R^2$ is star-shaped,   {\it  product domains.}  Similarly to $\mathcal{T}_{\R^4}$ in Figure \ref{fig:comparison}, we denote 
\begin{equation} \label{equ-set-prod-domain}
\mathcal{T}_{T^*\T^2} := \left\{\T^2 \times A \subset T^*\T^2 \,\big| \, \mbox{star-shaped domain $A \subset \R^2$}\right\}.
\end{equation}
It is natural to investigate the relationship between toric domains $X_{\Omega}$ and their corresponding product domains $\T^2 \times {\rm int}(\Omega)$ obtained by deleting all the singular fibers. As will be shown in Propositions \ref{prop-convex} and \ref{prop-zeta-func},   certain dynamical properties will change, while in contrast several quantitative symplectic invariants remain  unchanged under suitable conditions, as demonstrated in Corollary \ref{cor-capacity-BM} and Theorem 1.2.1 of \cite{Mcd25}.

%Meanwhile, by Proposition \ref{prop-toric} based on a trick in \cite{Tra95}, a natural analogue of a domain in $T^*\T^2$ being toric is so-called the {\it product domains}, explicitly in the form of $\T^2 \times A$ for some star-shaped domain $A \subset \R^2$. 

Note that there exists a quite large family of star-shaped domains in $T^*\T^2$ that arises as the (unit) codisc bundles of $(\T^2, F)$ for Finsler metrics $F$ on $\T^2$, usually denoted by $D_F^*\T^2$. 
Recall the definition of  codisc  bundles 
$D_F^*\T^2 \coloneqq \{(q,p)\in T^*\T^2 \mid F^*(q,p) \leq 1\},$
where   
$F^*\colon T^*\T^2 \to \R$ denotes the dual norm
\begin{equation}\label{dfn-F*}
    F^*(q,p)\coloneqq \sup_{v\in T_q\T^2,\; F(q,v)\leq 1} p\cdot v, 
\qquad \forall (q,p)\in T^*\T^2.
\end{equation} 
Denote by $\mathcal{F}$ the set of all codisc bundles of $(\T^2, F)$ for Finsler metrics $F$ on $\T^2$, and by definition their fibers are not only star-shaped but also convex. However, $\mathcal F$ serves as an ill analogue of $\mathcal C$ (the set of convex domains in $\R^4$) in Figure \ref{fig:comparison} since, by (9) of Theorem \ref{thm-class} below, $\mathcal F$ does {\it not} sit inside $\mathcal{D}_{T^*\T^2}$. In fact, elements in $\mathcal F$ may go beyond those convexity constrains as in the Euclidean spaces, either from index reason or from other geometric reasons. In order to set up a refined subclass in $\mathcal F$, let us recall the weak Viterbo conjecture in \cite{Vit00} that reflects the rigidity of being convex in terms of dynamics: for any convex domain $X \subset (\R^4, d\lambda_{\rm std})$ (with smooth boundary), 
\begin{equation}\label{eq-sys-rat-1}
    \rho_{\rm sys, \R^4}(X)\coloneqq \frac{\mathrm{sys}(\partial X)^2}{\mathrm{vol}(\partial X, \lambda_{\rm std}|_{\partial X})} \leq 1
\end{equation}
where $\mathrm{sys}(\partial X)$ denotes the minimal action of  closed Reeb orbit on $(\partial X, \lambda_{\rm std}|_{\partial X})$ and $\mathrm{vol}(\partial X, \lambda_{\rm std}|_{\partial X})$ denotes the volume of $\partial X$ with respect to the contact 1-form $\lambda_{\rm std}|_{\partial X}$ on $\partial X$. However, it was recently proved to be false in \cite{PY25} via a non-centrally symmetric  example. In the setting of $T^*\T^2$, analogous questions concerning   the systolic ratio of   codisc bundles of the torus  have been studied (see, e.g., \cite{Sab10, ABT16, BG24}). This motivates the following definition.

%In other words, there exists a relatively small (uniform) upper bound of $\rho_{\rm sys}$ for elements in $\mathcal C$. This motives the following definition in the setting of $(T^*\T^2, d\lambda_{\rm can})$. 

 \begin{dfn}\label{dfn-sys-conx}
A fiberwise star-shaped domain $X\subset  (T^*\T^2, d\lambda_{\rm can})$ is called {\bf systolically convex} if it is dynamically convex and \footnote{We use different notations for systolic ratio since in the setting $(\R^4, d \lambda_{\rm std})$, the ratio $\rho_{\rm sys, \R^4}$ in (\ref{eq-sys-rat-1}) is a symplectic invariant under symplectomorphism acting on input domain $X$ (see Lemma 3.5 \cite{CE22} or Lemma 2.10 in \cite{Ush22}), while in $(T^*\T^2, d\lambda_{\rm can})$ since $H^1(\T^2; \R)$ is non-trivial, the ratio $\rhoT(X)$ is only a contact invariant of the boundary $\partial X$. Note that any {\it exact} symplectomorphism in $(T^*\T^2, d\lambda_{\rm can})$ preserves the ratio $\rhoT$.}
\begin{equation} \label{eq-codisc-sym-sys}
\rho_{{\rm sys}, T^*\T^2}(X)\coloneqq\frac{\mathrm{sys}(\partial X)^2}{\mathrm{vol}(\partial X, \lambda_{\rm can}|_{\partial X})} \leq \frac{1}{4}. 
\end{equation}
 \end{dfn}

Here, we give two examples to support Definition \ref{dfn-sys-conx}. 

\begin{ex} \label{ex-sys-conv} Consider a product domain $X = \T^2 \times A \subset T^*\T^2$ (as an element in $\mathcal{T}_{T^*\T^2}$ in (\ref{equ-set-prod-domain})), where $A \subset (\R^2)^*$ is a star-shaped and centrally symmetric convex domain in the dual space of $\R^2$. Then via Minkowski gauge functional on $A$, one obtains a map $\kappa^*_{A}: (\R^2)^* \to \R_{\geq 0}$. It induces a reversible flat Finsler norm $F: T\T^2 \simeq \T^2 \times \R^2 \to \R$ defined by 
\begin{equation} \label{cons-Finsler}
F(q, x) =\kappa_A(x) \,\,\,\,\mbox{for $x \in \R^2$}
\end{equation}
where $\kappa_A$ is the dual of $\kappa^*_A$ (and then $D_F^*\T^2 = \T^2 \times A$). In fact, by an equivalence definition of a Finsler norm being flat \cite{Ben24,BG24}, the codisc bundle of any flat 2-torus $(\T^2, F)$ is a product domain lying in $\mathcal{T}_{T^*\T^2}$, and being reversible enhances the fiber to be centrally symmetric. By Proposition \ref{prop-convex} below, $X$ is dynamically convex. Moreover, by Theorem 12.1 in \cite{Sab10} (see also the paragraph under Theorem IV in \cite{ABT16}) applied to the reversible Finsler metric $F$ on $\T^2$ constructed in (\ref{cons-Finsler}) above, $\rho_{{\rm sys}, T^*\T^2}(X)\leq\frac{1}{4}$. Explicitly, in terms of the notation in \cite{ABT16}, 

\[ \rho_{{\rm sys}, T^*\T^2}(X) \leq \frac{(\sqrt{V/3})^2}{4V/3} = \frac{1}{4} \]
where $V = \mathrm{vol}(\partial X, \lambda_{\rm can}|_{\partial X})$. Therefore, any such product domain $X$ in $\mathcal{T}_{T^*\T^2}$ is systolically convex. The upper bound $\frac{1}{4}$ is sharp in the sense that it can be obtained by $A$ as a polygon in $\R^2$ with vertices $(1,0), (0, 1), (-1, 0), (0, -1)$ by \cite{Sab10} and the paragraph under Theorem 4.6 in \cite{ABT16}.
\end{ex}

\begin{remark} \noindent (i) \label{rmk-1-3} Theorem IV in \cite{ABT16} in fact implies that for any $X = \T^2 \times A \subset T^*\T^2$ where $A$ is star-shaped and convex (but not necessarily centrally symmetric), and then the ratio satisfies $\rho_{{\rm sys}, T^*\T^2}(X) \leq \frac{1}{3}$. Again,  by the paragraph under Theorem 4.6 in \cite{ABT16}, this upper bound $\frac{1}{3}$ can be obtained by a (non-centrally symmetric but convex) polygon $A \subset \R^2$ as shown in Figure \ref{fig:example_(5)} below (borrowed from Figure 1 in \cite{ABT16}). 

\noindent (ii) The criterion formulated in Proposition 1.9 of \cite{CE22} involves two key ingredients: the systolic ratio $\rho_{\rm sys, \R^4}$ and the Ruelle invariant (see Definition 2.17 of \cite{CE22}). However, we will see in Proposition \ref{prop-convex} that Ruelle invariant does  not distinguish product domains $\mathcal{T}_{T^*\T^2}$ in $(T^*\T^2,\omega_{\rm can}).$ 
\end{remark} 

\begin{remark} For the product domain $\T^2\times A$, we mainly work with those where the fiber $A$ admits a smooth boundary $\partial A$. When $\partial A$ is non-smooth (for example a polygon),  we approximate it (in a $C^0$-sense) by a sequence of  star-shaped domains $\{A_\ep\}_{\ep >0}$ with smooth boundaries and define the systolic ratio $\rhoT(\T^2\times A)$  as the limit  of the systolic ratio of the sequence $\{\T^2\times A_\ep\}_{\ep >0}$.  An explicit smoothing construction and computation is presented in Example \ref{ex-perturb-cal}.\end{remark}

Besides product domains with convex fibers as discussed in Example \ref{ex-sys-conv}, here is another family of systolically convex domains, supporting Definition \ref{dfn-sys-conx}.  For  a (smooth) star-shaped domain $A\subset \R^2$, we call it {\it generalized monotone} if for any $p=(p_1,p_2)\in \partial A$ with outward normal vector $n(p)=(n_1,n_2)$, the ``product'' vector $(p_1n_1,p_2n_2)$ lies in $ \R_{\geq0}^2.$ As a special case, a domain $\Omega\subset \R^2_{\geq 0}$ is generalized monotone if $X_\Omega$ is a monotone toric domain, after performing appropriate smoothings   of $\Omega$ around the corners.
 
\begin{ex}\label{ex-mono-sys-convex}
Any product domain $ \T^2\times A$, where $A\subset \R^2$ is generalized monotone, is  systolically convex. 
In each quadrant, we consider the maximal triangle $\Delta_i \subset A$ and $\Delta_i\cap \partial A\neq \emptyset$  with vertices
\[
 \{(0,0),(a_i,0),(0,a_i)\}\subset A ,\quad 1\leq i\leq 4.
\] 
   Let $p\in \Delta_i\cap \partial A$. If $p$ lies on a coordinate axis, then by the action formula~(\ref{eq-action}), there exists a closed Reeb orbit in $\T^2\times\{p\}$ corresponding to the normal direction $(\pm1,0)$ or $(0,\pm1)$ with action $a_i$. 
If $p\in \R_{>0}^2$, then $\Delta_i$ is tangent to $\partial A$ at $p$, and the outward normal direction at $p$ is $(\pm1,\pm1)$. In this case there exists a closed Reeb orbit in $\T^2\times\{p\}$ with action $a_i$ corresponding to this direction. Moreover, in either case there exists a polygon with vertices $(\min_i |a_i|,0),(0,\min_i |a_i|), (-\min_i |a_i|,0),(0,-\min_i |a_i|),$ contained inside $A$. Then  the upper bound for the systolic ratio of $\T^2\times A$ follows Example \ref{ex-convex-hull} in Appendix \ref{sec-app}. Together with Proposition \ref{prop-convex}, we conclude that any such domain is systolically convex. 
\end{ex}

%\begin{remark}\label{rmk-upper-bound}
%Here are some remarks around   Definition \ref{dfn-sys-conx}.

%\noindent (ii) The equation  (\ref{eq-codisc-sym-sys}) can be viewed as the inequality of weak Viterbo’s conjecture in $T^*\T^2$ setting. Moreover, by \cite{Ben24}, for codsic bundle of flat Riemannian torus $D_g^*\T^2$, any  ball-normalized  capacity $c$ on  $D_g^*\T^2$ is equal to $2\mathrm{sys}(D_g^*\T^2)$. Therefore, inequality (\ref{eq-codisc-sym-sys}) in this case can be translated as 
%$c(D_g^*\T^2)^2\leq 2 \mathrm{Vol}(D_g^*\T^2),$ 
%which  agrees with the inequality of Viterbo’s volume-capacity conjecture in \cite{Vit00}.

%\end{remark}

Denote by $\mathcal{S}$ the set of all systolically convex domains in $(T^*\T^2, d\lambda_{\rm can})$. In this paper, we regard it as an analogue of $\mathcal C$ in Figure \ref{fig:comparison}. So far, we have encountered the following four subclasses of domains in $T^*\T^2$ from the discussions above: $\mathcal{D}_{T^*\T^2}, \mathcal{T}_{T^*\T^2}, \mathcal F$ and $\mathcal S$, with inclusion relations (by Proposition \ref{prop-convex} for the first inclusion and by definition directly for the second inclusion): 
\begin{equation} \label{inclusion-subclass}
\mathcal{T}_{T^*\T^2} \subset \mathcal{D}_{T^*\T^2},\,\,\,\mbox{and}\,\,\,\,\mathcal S \subset \mathcal{D}_{T^*\T^2}. 
\end{equation}
 Different from Figure \ref{fig:comparison}, in the setting of $T^*\T^2$ the existence of subcalss $\mathcal F$ (which does not admit any analogue in the Euclidean space setting) provides richer examples that spread around other subclasses. To facilitate our discussion, let us introduce a few more notations that refine $\mathcal F$: 
\begin{equation} \label{subclass-F}
\mathcal{F}_{\rm rev} \subset \mathcal{F}_{\rm flat} =: \mathcal F \cap \mathcal{T}_{T^*\T^2} 
\end{equation}
where $\mathcal{F}_{\rm rev}$ consists of those product domains with fibers star-shaped and centrally symmetric convex (see Example \ref{ex-sys-conv}), and $\mathcal{F}_{\rm flat}$ consists of those product domains with fibers star-shaped and also convex. Later along the proof of Theorem \ref{thm-normalized} we will see that subclass $\mathcal{F}_{\rm flat}$ is closely related to 
\emph{generalized convex toric domains} defined in \cite{Mcd25} (see Section \ref{section-normalized}). Theorem \ref{thm-class} fully distinguishes all these subclasses, with schematic picture Figure \ref{fig:diagram} below, serving as an analogue of Figure \ref{fig:comparison}.
\begin{figure}[h]
    \centering
    \includegraphics[width=0.62\linewidth]{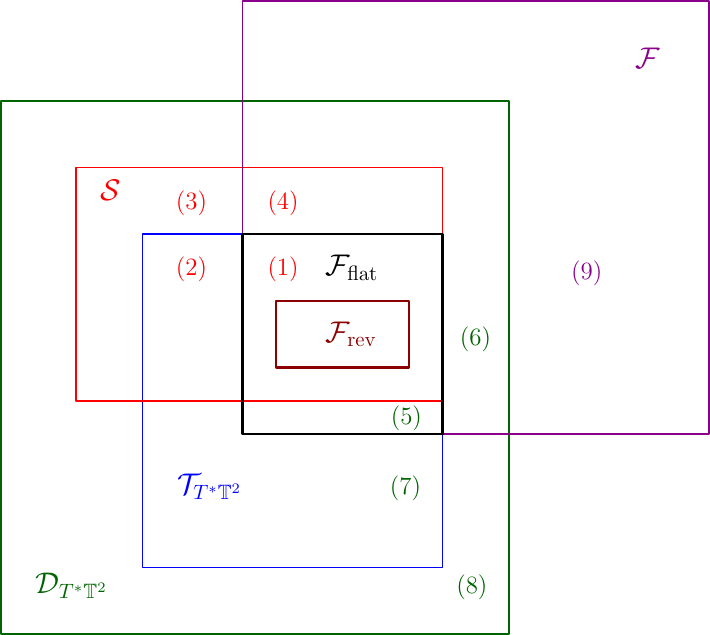}
    \caption{Relation between subclasses of fiberwise star-shaped domains in $T^*\T^2$. Distinguishing examples (1) - (9) will be provided in the proof of Theorem \ref{thm-class} in Section \ref{sec-thm-A}. }
    \label{fig:diagram}
\end{figure} 
Moreover, recall that toric domains $\mathcal{T}_{\R^4}$ can be distinguished from convex domains $\mathcal{C}$ in Figure \ref{fig:comparison}  under symplectomorphism equivalence by computing dynamical zeta function introduced in \cite{Hut24}. Here, in a similar manner, we can distinguish some element in $\mathcal F$ from those in $\mathcal{F}_{\rm flat} $ even under exact symplectomorphisms.  This will be confirmed by Theorem \ref{thm-non-flat}.

\begin{remark} Let us mention another curious point about the subclasses constructed above. Comparing Figure \ref{fig:comparison} and Figure \ref{fig:diagram}, there is no direct analogue of ``monotone toric domain'' in the setting of $T^*\T^2$ partially due to the first relation in (\ref{inclusion-subclass}). One possible way to specify such a family is to strengthen the definition of dynamical convexity, for instance, with extra index-condition on {\it non-contractible} closed Reeb orbits. Then the size of subclass $\mathcal{D}_{T^*\T^2}$ is decreased and the first inclusion in (\ref{inclusion-subclass}) does not always hold. Another possible way is to specify a family of star-shaped domains in $T^*\T^2$ such that proofs in \cite{GHR22,GMV22,Ben24}, which confirm the coincidence of capacities, go through still (cf.~Corollary \ref{cor-capacity-BM} below). This will be explored somewhere else. \end{remark}

Parallel to the ``systolic ratio conjecture'' on dynamically convex domains in $\R^{4}$ (see the paragraph below Theorem 1.1 in \cite{ABHS18}), here we propose a similar question but in the setting of $T^*\T^2$:

\begin{question}\label{que-sys-ratio} Is $\sup\left\{\rho_{{\rm sys}, T^*\T^2}(X) \,| \, X \in \mathcal{D}_{T^*\T^2}\right\}$ finite or infinite? \end{question}

A few words about Question \ref{que-sys-ratio}: when restricted to product domains in the set (\ref{equ-set-prod-domain}) above, this question was mentioned in the second paragraph in Section 5.1 of \cite{Sim25}; in Appendix \ref{sec-app} of the current paper, we can verify the boundedness for a certain family of product domains (see Corollary~\ref{cor-bound-sys} and Example~\ref{ex-convex-hull}). In fact, we have been informed by A.~Abbondandolo and more in detail by S.~Vialaret that in an upcoming work \cite{BVV}, the systolic ratios of product domains in (\ref{equ-set-prod-domain}) is uniformly bounded by $\frac{4}{9}$. In a parallel way, one may ask  the same question on the systolic ratio $\rho_{\rm sys, \R^4}$ for toric domains $X_\Omega \subset \R^4$. In this case the finiteness of $\rho_{\rm sys, \R^4}$ is almost immediate. Indeed, one can consider either closed Reeb orbits from the intersection of $\Omega$ with the coordinate axes or the closed Reeb orbit from the largest isosceles right triangle embedded inside $\Omega$. 
 
\medskip

Apart from the comparison and distinguishing results of star-shaped domains in $(\R^4, \omega_{\rm std})$ illustrated in Figure \ref{fig:comparison}, several further studies have been explored. For instance, from the perspective of metric geometry, for the set of star-shaped domains in $(\R^{2n},\omega_{\rm std})$, its large-scale geometry with respect to \emph{(fine) symplectic Banach-Mazur distance} (and here we denote it by $d_{\rm SBM}$) was obtained by Theorem  1.5 in \cite{Ush22}. Restricted to $(\R^4, \omega_{\rm std})$, Theorem 1.2 in \cite{DR23} confirms the large-scale geometry of the set of monotone toric domains $\mathcal{M}$ with respect to \emph{coarse Banach-Mazur distance} (denoted by $d_{\rm BM}$ and see (\ref{eq-CBM}) in Section \ref{ssec-large}). Note that \cite{DR23} considerably improves \cite{Ush22} (in the case of $\R^4$), since $d_{\rm BM}$ could be much smaller than $d_{\rm SBM}$ in general (see Theorem 1.4 in \cite{Ush22}). 

For the same study on the set of fiberwise star-shaped domains in $T^*\T^2$ (in fact, for more general $T^*\Sigma_{g}$ where $\Sigma_{g}$ is a closed surface with genus $g$), large-scale geometry has been confirmed with respect to $d_{\rm SBM}$ as the main results in \cite{SZ21}. In this paper, Theorem \ref{thm-large-scale-one} below confirms the large-scale geometry of the set of product domains $\mathcal{T}_{T^*\T^2}$ in $(T^*\T^2,\omega_{\rm can})$ with respect to $d_{\rm BM}$. Meanwhile, incorporating with non-trivial topology in the setting of $T^*\T^n$ (for a general $n \geq 2$), the existence of large-scale geometry on $\mathcal{T}_{T^*\T^n}$ with respect to newly-introduced Banach-Mazur type distance can be obtained by Theorem~\ref{thm-large-scale-two}. 

\medskip

For another instance, the strong Viterbo conjecture approaches the study of star-shaped domains in $(\R^{2n}, \omega_{\rm std})$ (specifically for convex domains) from another quantitative perspective based on symplectic capacities. It claims that all ball-normalized symplectic capacities coincide on convex domains, which was  recently  disproved  in \cite{PY25}. In $(T^*\T^2,\omega_{\rm can})$, the coincidence of ball-normalized capacities was confirmed for a certain family of fiberwise star-shaped domains (in fact, of $T^*\T^n$), including all codisc bundles of $(\T^2, g)$ for flat Riemannian metrics $g$ (see Theorem 1.1 and Remark 1.2 in \cite{Ben24}). Theorem \ref{thm-normalized} and Theorem \ref{thm-ball-normalized} below confirm this coincidence for some other cases in $T^*\T^n$. 

\begin{remark}\label{rmk-intro-cap}
Note that all the cases, where the normalized capacities coincide in \cite{Ben24}, Theorem \ref{thm-normalized}, and Theorem \ref{thm-ball-normalized} in the current paper, sit inside $\mathcal S$ (i.e., systolically convex as in Definition \ref{dfn-sys-conx}). Meanwhile, Example \ref{ex-cap-non-coin} below confirms that there exists a domain $\T^2\times A \in {\mathcal T}_{T^*\T^2} \backslash \mathcal{S}$ where two ball-normalized capacities do {\it not} coincide. View the defining property of $\mathcal S$, being systolically convex, as our counterpart of being geometrically convex in $\R^{4}$, then we propose the following analogue of the strong Viterbo conjecture in the setting of $(T^*\T^2, \omega_{\rm can})$:

\begin{conj} \label{conj-T-Vit}
All ball-normalized capacities coincide for  elements in $\mathcal S$. 
\end{conj}

Denote by $\F_{\rm rev}^{2n}$ the general $2n$-dimensional version of $\F_{\rm rev}$. Similarly, one can formulate a conjecture, analogue to the weak Viterbo conjecture, in this cotangent bundle setting (cf.(\ref{eq-sys-rat-1})) when the fiberwise star-shaped domain possesses some symmetry (for instance, elements in $\mathcal{F}_{\rm rev}^{2n}$).  The same conclusion as in Conjecture \ref{conj-T-Vit} can also be posted for $\F_{\rm rev}^{2n}$. Interestingly, there is a close connection between Conjecture \ref{conj-T-Vit} on $\F_{\rm rev}^{2n}$ and an ``enhanced'' version of the strong Viterbo conjecture (claiming the coincidence of all ball-normalized capacities on centrally symmetric convex domains in $\R^{2n}$): Proposition \ref{prop-counterexample} below implies that if there exists a product domain $\T^n \times A \in\F_{\rm rev}^{2n}$ for which ball-normalized capacities do {\it not} coincide, then the Lagrangian product $(A\times A^*,\omega_{\rm can})\subset (T^*\R^n,\omega_{\rm can})\simeq(\R^{2n},\omega_{\rm std})$ provides a counterexample of this enhanced strong Viterbo conjecture in $(\R^{2n}, \omega_{\rm std})$. This provides a way to disprove the enhanced strong Viterbo conjecture.
\end{remark}

\section{Main results}
\subsection{Distinguish subclasses}\label{sec-back-conx}
For our setting $(T^*\T^2, \omega_{\rm can} = d\lambda_{\rm can})$, its second homotopy group $\pi_2$ vanishes. Meanwhile, for any fiberwise star-shaped domain $U \subset T^*\T^2$, its boundary $\partial U$ can be viewed as a contact $3$-manifold $(\partial U,\xi_{\rm can} = \ker \lambda_{\rm can}|_{\partial U})$. Then, following Definition 1.1 in \cite{HN16}, the \emph{dynamical convexity} condition on the contact 1-form $\lambda_{\rm can}|_{\partial U}$  requires  that
\begin{itemize}
    \item [(i)] either $(\partial U,\lambda_{\rm can}|_{\partial U})$ has no contractible Reeb orbits;
    \item [(ii)] or any contractible Reeb orbit $\gamma$ has CZ-index  greater or equal to $3.$  
\end{itemize}
Here, CZ-index represents the Conley-Zehnder index of $\gamma$ with respect to any trivialization $\tau$ of $\xi|_\gamma$ extended to a bounding disk of $\gamma$. 

\medskip

In some special cases, such as product domains $U = \T^2 \times A$ in (\ref{equ-set-prod-domain}), the Reeb dynamics on the contact boundary $\partial U = \T^2 \times \partial A$ is particularly easy to analyze (see the beginning of Section \ref{sec-proofs}). Proposition \ref{prop-convex} then follows, and it completes the set-containing relations in (\ref{inclusion-subclass}). Moreover, the following is our first main result, which refines the set-containing relations in Figure \ref{fig:diagram} by constructing a series of examples that distinguish different subclasses in Figure \ref{fig:diagram}.

\begin{thm}\label{thm-class}
All containment relations in Figure~\ref{fig:diagram} are strict. 
\end{thm}

Note that neither Figure \ref{fig:comparison} nor Figure \ref{fig:diagram} encodes the action of symplectomorphisms. In our setting of $T^*\T^2$, due to the non-triviality of $\pi_1(T^*\T^2) = \pi_1(\T^2) ( = H_1(\T^2; \Z))$, we will mainly consider {\it exact} symplectomorphisms $\varphi$ in the sense that \begin{equation}\label{eq-exact-H1}
  \varphi^*\lambda_{\rm can}-\lambda_{\rm can} =df  \text{ for some $f\in C^\infty (T^*\T^2) $.   } 
\end{equation} 
The following result serves a parallel conclusion to Example 1.4 in \cite{Hut24}, which distinguishes (the interiors of) a codisc bundle of tori  $\mathcal{F}$ from those in $\mathcal{F}_{\rm flat} $ even under exact symplectomorphisms. To this end, we compute the dynamical zeta function $\zeta$ for product domains $\mathcal{T}_{T^*\T^2}$ introduced in Definition 1.5 and Definition 1.11 in \cite{Hut24}. We emphasize that by Theorem 1.16 in \cite{Hut24}, in general dynamical zeta function is invariant only under exact symplectomorphisms (which also emphasizes the necessity of (\ref{eq-exact-H1}) above). 
 
\begin{thm}\label{thm-non-flat}
    There exists a fiberwise star-shaped domain $U \in \mathcal F$ such that
the  interior ${\rm int}(U)$ is {\it not} exact symplectomorphic to the interior of any domains in $\mathcal{F}_{\rm flat}$.
\end{thm}
\subsection{Large-scale geometry on $T^*\T^n$} \label{ssec-large} 
Recall that given compact  star-shaped domains $U,V\subset\R^{2n}$,  the \emph{coarse Banach-Mazur distance} (cf.~Definition 1.3 in \cite{Ush22})  is defined as
\begin{equation}\label{eq-CBM}
    d_{\rm BM}(U,V)\coloneqq
\inf\left\{\ln C>0 \mid 
U\hookrightarrow C\cdot V \,\,\mbox{and}\,\,
V\hookrightarrow C\cdot U\,\right\},
\end{equation} 
where $C\cdot U$ denotes the scaling of $U$ by $C>0$ via flowing along the radial vector field $X_{\rm rad} = \frac{1}{2}\sum_{i=1}^n \left(x_i \partial_{x_i} + y_i \partial_{y_i}\right)$ in $\R^{2n} = \R_{2n}(x_1, \ldots, y_n)$, and ``$\hookrightarrow$'' denotes the existence of a symplectic embedding. For $T^*\T^n$, the same definition applies to star-shaped domains in $T^*\T^n$, simply changing $X_{\rm rad}$ to the vector field $\sum_{i=1}^n p_i \partial_{p_i}$ where $p_i$'s denote the coordinates of co-vectors. Again, ``$\hookrightarrow$'' denotes the existence of a symplectic embedding (and {\it not} necessarily exact). 

Meanwhile, given two (pseudo-)metric spaces $(X,d)$ and $(Y,d')$, a map $\varphi: X \to Y$ is called a \emph{quasi-isometric embedding} if there are constants $A\ge1 ,B\geq 0$ such that
\[\frac{1}{A}\cdot d(x_1,x_2)-B\leq d'(\varphi(x_1),\varphi(x_2))\leq A\cdot d(x_1,x_2)+B\]
for any $x_1,x_2\in X.$ In particular, $\varphi: X \to Y$ is called an {\it isometric embedding} if $A = 1$ and $B =0$. 

Based on the large-scale geometry of the set of monotone toric domains $\mathcal{M}$ with respect to the coarse Banach-Mazur distance $d_{\rm BM}$ in Theorem 1.2 of \cite{DR23}, our next main result confirms the conclusion in the setting of $(T^*\T^2, \omega_{\rm can})$: 

\begin{thm}\label{thm-large-scale-one}
    For any positive integer $N\in \N$, there exists a quasi-isometric embedding from the metric space $(\R^N,|\cdot|_{\infty})$ to the pseudo-metric space that consists of all fiberwise star-shaped domains in $(T^*\T^2,\omega_{\rm can})$ equipped with coarse Banach-Mazur distance $d_{\rm BM}.$
\end{thm} 

Theorem \ref{thm-large-scale-one} above considerably improves the conclusion of Theorem 1.11 of \cite{SZ21} for the case $\T^2 = \Sigma_{g = 1}$, in the sense that it removes all the extra ``decorations'' in the definition of a weaker pseudo-distance, (fine) symplectic Banach-Mazur distance denoted by $d_{\rm SBM}$, widely used in \cite{SZ21}. These decorations include the following three topological conditions: (i) symplectic embeddings from ``$\hookrightarrow$'' induce trivial actions on the set of homotopy classes of free loops in $\T^2$; (ii) symplectic embeddings need to be exact as in (\ref{eq-exact-H1}); (iii) compositions of certain embeddings need to be ``unknotted'' (see Definition 1.2 in \cite{SZ21}, Definition 1.3 (ii) in \cite{Ush22}). 

 In the setting of cotangent bundle of higher-dimensional torus $(T^*\T^n, \omega_{\rm can})$, we define a pseudo-distance that is analogue to $d_{\rm SBM}$, but only with topological conditions (i) and (ii) above. 
 
 \begin{dfn}\label{dfn-HBM}
  Let  $U,V\subset (T^*\T^n,\omega_{\rm can}=d\lambda_{\rm can})$ be fiberwise star-shaped domains. The {\bf homological Banach-Mazur  distance} between $U$ and $V$ is defined by
\[d_{\rm HBM}(U,V)\coloneqq \inf \left\{\ln C>0 \,\Big | \, U\xhookrightarrow{H^1-\text{trivial}} C\cdot V \,\,\mbox{and}\,\, V\xhookrightarrow{H^1-\text{trivial}} C\cdot U\right\} . \]
Here $\xhookrightarrow{H^1-\text{trivial}}$ denotes an  exact $H^1$-trivial symplectic embedding,  meaning that the embedding  $\varphi$ satisfies (\ref{eq-exact-H1}) and 
\begin{equation}\label{eq-H1}
  \varphi^*|_{H^1(\T^n;\R)}=\mathrm{id.}   
\end{equation} 
\end{dfn}

Curiously, Definition \ref{dfn-HBM} applies to star-shaped toric domains in $X_\Omega\subset (\R^{2n},\omega_{\rm std})$ in the following way. Denote by $\mathrm{int}(\Omega)\subset \R^n_{>0}$ the interior of $\Omega$ that is disjoint from axes, then 
\[X_{\mathrm{int}(\Omega)}\coloneqq\mu^{-1}(\mathrm{int}(\Omega))\subset X_\Omega\]
where $\mu$ is the moment map given in (\ref{eq-toric-domain}). By  Proposition \ref{prop-toric} below,   $(X_{\mathrm{int}(\Omega)},\omega_{\rm std})$  is symplectomorphic to $(\T^n\times\mathrm{int}(\Omega)  ,\omega_{\rm can})$. In particular, the topology of $X_{\mathrm{int}(\Omega)}$ is non-trivial (while $X_{\Omega}$ contracts to a point), and compared with $X_{\Omega}$ the domain $X_{\mathrm{int}(\Omega)}$ loses the part $X_{\Omega} \cap \left(\cup_{i=1}^n \mbox{\{$z_i$-plane\}}\right)$ where $\R^{2n}$ is identified with $\C^n$ via $z_i = x_i + \sqrt{-1} y_i$. Then we have the following definition, parallel to Definition \ref{dfn-HBM}: 
\begin{dfn}\label{dfn-HBM-toric}
    Given toric domains $X_\Omega, X_{\Omega'}\subset (\R^{2n},\omega_{\rm std})$, their {\bf homological Banach-Mazur distance}  is defined by
\begin{equation}\label{eq-dfn-HBM-toric}
    \begin{aligned}
d_{\rm HBM}(X_\Omega, X_{\Omega'})
\coloneqq \inf \left\{\ln C>0 \ \middle|\ 
\begin{array}{l}
X_{\mathrm{int}(\Omega)} \xhookrightarrow{H^1\text{-trivial}} C\cdot X_{\mathrm{int}(\Omega')},\\
X_{\mathrm{int}(\Omega')} \xhookrightarrow{H^1\text{-trivial}} C\cdot X_{\mathrm{int}(\Omega)}
\end{array}
\right\}
\end{aligned}
\end{equation} 
where $\xhookrightarrow{H^1-\text{trivial}}$ denotes an exact and $H^1$-trivial  symplectic embedding  satisfying (\ref{eq-exact-H1}) and (\ref{eq-H1}) between $(\T^n\times\mathrm{int}(\Omega) ,\omega_{\rm can})$ and $ (\T^n\times\mathrm{int}(\Omega') ,\omega_{\rm can})$. 
\end{dfn}
\begin{remark}\label{remark-HBM}  
Definition \ref{dfn-HBM} and Definition \ref{dfn-HBM-toric} indeed define pseudo-metrics, where the sub-additivity comes from the fact that the composition of exact $H^1$-trivial  symplectic embeddings is still exact and $H^1$-trivial. We leave details of the verification to interested readers. 

Note that in Definition \ref{dfn-HBM-toric}, the domain $X_{\mathrm{int}(\Omega)}\subset\R^{2n}$ is {\it not}  star-shaped. Instead,   the scaling in (\ref{eq-dfn-HBM-toric}) is performed on the fiberwise star-shaped domain  $\T^n\times \mathrm{int}(\Omega)\subset T^*\T^n$, and it acts in  the same manner as the rescaling described below equation~\eqref{eq-CBM}.
\end{remark}

 \begin{remark}\label{rmk-Ham} One can obtain exact $H^1$-trivial symplectic embeddings appearing in Definition \ref{dfn-HBM} and Definition \ref{dfn-HBM-toric} from Hamiltonian dynamics. In fact, any Hamiltonian diffeomorphism $\varphi\in \mathrm{Ham}(T^*\T^n,\omega_{\rm can})$ is an exact and $H^1$-trivial symplectomorphism.
Let $\{\varphi_t\}_{t\in[0,1]}$ be a Hamiltonian isotopy with  
Hamiltonian function $H_t\colon [0,1]\times T^*\T^n\to \R$ and $\varphi_0 = \mathrm{id}$, $\varphi_1 = \varphi$.  
Define $\alpha_t := \varphi_t^*\lambda_{\rm can} - \lambda_{\rm can}.$
We compute the time derivative:
\[
\frac{d}{dt}\alpha_t   
= \varphi_t^*(\mathcal{L}_{X_{H_t}}\lambda_{\rm can})=\varphi_t^*(d(\iota_{X_{H_t}}\lambda_{\rm can}) + \iota_{X_{H_t}}d\lambda_{\rm can})=d\bigl(\varphi_t^*(\iota_{X_{H_t}}\lambda_{\rm can} - H_t)\bigr),
\]
where $\mathcal{L}$ denotes the Lie derivative.   
Thus $\frac{d}{dt}\alpha_t$ is exact for all $t$.  Integrating from $0$ to $1$ and using
$\alpha_0=0$, we obtain
\[
\alpha_1
= \int_0^1 \frac{d}{dt}\alpha_t\,dt
= d\left(\int_0^1 \varphi_t^*(\iota_{X_{H_t}}\lambda_{\rm can} - H_t)\,dt\right).
\]
Hence
$\varphi^*\lambda_{\rm can} - \lambda_{\rm can} = \alpha_1 = d\left(\int_0^1 \varphi_t^*(\iota_{X_{H_t}}\lambda_{\rm can} - H_t)\,dt\right)$
so that $\varphi$ is exact, satisfying (\ref{eq-exact-H1}). It is also $H^1$-trivial since  $\varphi$ is smoothly isotopic to the identity so that it induces the identical map on cohomology, thus satisfying (\ref{eq-H1}).  \end{remark}

%Large-scale geometric results for (fiberwise) star-shaped domains are often derived from various types of symplectic invariants that capture the rigidity of symplectic embeddings.
%Among them, \emph{shape invariants} (see Definition \ref{dfn-shape} and relevant studies in \cite{Yak91,HZ21,HZ23,RZ21}) arise from obstruction to Lagrangian embeddings and encode higher dimensional    information compared to symplectic capacities. 
 %Using  shape invariant techniques for product domains in $(T^*\T^n,\omega_{\rm can})$ (see Theorem 2.4.1 in  \cite{Yak91} as a reformulation of Theorem 3 in \cite{Sik89}), we have the following results on distance $d_{\rm HBM}$:

 With respect to $d_{\rm HBM}$, we have the following large-scale geometry results. We emphasize that conclusion (i) in Theorem \ref{thm-large-scale-two} serves as a generalization of Theorem \ref{thm-large-scale-one} to higher-dimensional cases; (ii) in Theorem \ref{thm-large-scale-two}, serves as a generalization of the main result, Theorem 1.2 of \cite{DR23}, to higher-dimensional cases (but within toric domains). 
 
  \begin{thm}\label{thm-large-scale-two}
  For any $N \in \N$ and $n\geq 2$:
  \begin{itemize}
      \item [(i)] there exists an isometric embedding from  $(\R^{N} ,|\cdot|_{\infty})$ to  fiberwise star-shaped domains in $(T^*\T^n,\omega_{\rm can})$   equipped with $d_{\rm HBM};$  
      \item [(ii)]there exists an isometric embedding from   $(\R ^N_{\geq 0},|\cdot|_{\infty})  $  to toric domains in $(\R^{2n}, \omega_{\rm std})$  equipped with $d_{\rm HBM}.$  
  \end{itemize}
  \end{thm}
  \begin{remark}
      By  Lemma 8.13 in \cite{SZ21}, for any $N\in \N$ there exists a  quasi-isometric embedding from $(\R ^N ,|\cdot|_{\infty})  $ to $(\R ^{2N}_{\geq 0},|\cdot|_{\infty})  $. Consequently, (ii) in Theorem \ref{thm-large-scale-two}  implies the existence of a quasi-isometric embedding from $(\R ^N ,|\cdot|_{\infty})  $ to the set of
toric domains $\mathcal{T}_{\R^{2n}}$ in $(\R^{2n}, \omega_{\rm std})$  equipped with $d_{\rm HBM}.$  
  \end{remark}
  
\subsection{Normalized capacity in $T^*\T^n$}\label{ssec-nor-cap}
Recall a \emph{ball-normalized symplectic capacity} is a function $c(\cdot)$ on the set of $2n$-dimensional symplectic manifolds
satisfying the following conditions: 

\noindent (i) $c(M_1 ,\omega_1)\leq c(M_2,\omega_2)$ if $(M_1,\omega_1)$ symplectically embeds into $(M_2,\omega_2)$;

\noindent (ii) $c(M,a\omega)=a\cdot c(M,\omega)$ for any $a>0$;  

\noindent (iii) $c(B^{2n}(r),\omega_{\rm std})=\pi r^2=c(Z^{2n}(r),\omega_{\rm std}),$
where $B^{2n}(r)$ is the symplectic ball of radius $r$, and $Z^{2n}(r) := B^2(r)\times \R^{2n-2}\subset \R^{2n}$ is the symplectic cylinder. 

It is easily verified that the existence of a ball-normalized symplectic capacity is equivalent to the celebrated Gromov non-squeezing theorem in \cite{Gro85}. As a consequence, the first ball-normalized symplectic capacity is the Gromov width denoted by $c_{\rm Gr}$ defined as follows:
 \[c_{\rm Gr}((M, \omega)):= \sup\left\{\pi r^2 \mid \exists \text{ symplectic embedding $(B^{2n}(r),\omega_{\rm std})\hookrightarrow (M, \omega)$} \right\}.\] 
In general, computing the exact value of $c_{\rm Gr}$ for a given symplectic manifold $(M, \omega)$ is challenging (for related results, see \cite{KT05,LMS13,Sch17,FRV26}). One way to overcome this difficulty is to equate $c_{\rm Gr}$ with some other capacities, which leads to the following discussion. 

\medskip
 
The next result confirms the coincidence of all ball-normalized capacities for domains (viewed as symplectic manifolds with decent symplectic structure) in a certain family of fiberwise star-shaped domains in $T^*\T^2$. 

\begin{thm}\label{thm-normalized}
Let $c$ be a ball-normalized symplectic capacity  and     $D_F^*\T^2 \in \mathcal F_{\rm rev}$   be the codisc bundle induced by a flat reversible Finsler metric $F$ on $\T^2$.  Assume that the metric $F$ satisfies either one of the following conditions:
\begin{itemize}
    \item [(i)] $F(q_1,q_2,v_1,v_2)=F(q_1,q_2,v_1,-v_2)$ for any $(q_1,q_2,v_1,v_2)\in T\T^2 $;
    \item [(ii)] $\rhoT(D_F^*\T^2)\leq \frac{1}{8}$.
\end{itemize} 
Then we have $c(D_F^*\T^2)=2\mathrm{sys}(\partial D_F^*\T^2).$ 
 \end{thm}
 
 \begin{remark}\label{rmk-thm-E} 
The coincidence conclusion from condition (i) of Theorem \ref{thm-normalized} applies to those flat Finsler metric $F$ induced by $\ell^p$-norms (defined in (\ref{eq-l-p-norm})), namely   $F(q,v)=\|v \|_p$  on $\T^2$ with $p\in [1,\infty].$ Earlier results in this direction include  \cite{Jia93} for codsic bundle $D_F^*\T^n$ where $F$  is the $\ell^p$-norm with $p=1$ and $ n\geq 2$, as well as \cite{Ben24} for  $p\in [1,2]$ and $n\geq 2$. \end{remark}

As an application of Theorem \ref{thm-normalized}, the coincidence conclusion from its condition (ii) can be used to calculate  the Gromov width  $c_{\rm Gr}$ of the tilted cylinder $Y^4(r,v)$ defined above Theorem 1.18 of \cite{GX20}: for $r \in (0, \infty)$, 
\begin{equation} \label{dfn-tilted-cyl} 
Y^4(r,v)\coloneqq \T^2\times \left((-r,r)v\times v^{\perp}\right)
\end{equation}
where $v$ is a unit vector (with respect to the standard Euclidean norm) in $\R^2$, and $(-r, r)v$ is a rescaling of $v$ by scalar $r$, and $v^{\perp}$ denote the direction orthogonal to $v$. Here is the result. We call vector $\alpha=(m,n)\in \Z^2\backslash\{0\}$ {\it prime} if one of the following conditions holds:
\begin{itemize}
    \item [(i)] $m,n\neq 0$ and $m,n$ coprime;
    \item [(ii)] one of $m,n$ is $0$ and the other is $\pm1.$
\end{itemize}

    \begin{cor}\label{cor-cap-cylinder}
   Let $Y^4(r,v)\subset (T^*\T^2,\omega_{\rm can})$ be a  tilted cylinder and let $c $ be any ball-normalized symplectic capacity. If $v$ (as a unit vector) is the scalar multiple of a prime integer vector $\alpha\in \Z^2\backslash\{0\}$, then we have 
    \[c (Y^4(r,v))=2r\|\alpha\| \]
    where $ \|\alpha\|$ denotes the standard Euclidean norm of $\alpha.$ If $v$ is {\it not} a scalar multiple of any integer vector, then 
    \[c (Y^4(r,v))=+\infty.\]
    In particular, the conclusions above holds for the Gromov width $c_{\rm Gr}$. 
\end{cor}

\begin{remark} As a comparison, Theorem 1.18 in \cite{GX20} computes another capacity (called the BPS capacity, invented in \cite{BPS03}) of $Y^4(r, v)$ which depends on scalar $r$, direction $v$, and also a homotopy class in $[S^1, \T^n]$. \end{remark} 

 Recall that a symplectic capacity $c$ is called \emph{cube-normalized }if the usual ball-normalization condition (iii) above is replaced by   normalization with respect to cube  $P^{2n}(r)$
 and $N^{2n}(r)$, as in Definition 4 in \cite{GMV22}:
 \begin{equation}\label{eq-cube}
     P^{2n}(r)\coloneqq \mu^{-1}(\Omega_{P^{2n}(r)}), \text{ where }    \Omega_{P^{2n}(r)}=\left\{x\in \R^n_{\geq0}\mid \forall i=1,\ldots,n , \, x_i\leq  \pi r^2\right\}
 \end{equation}
 where $\mu$ is the moment map in (\ref{eq-toric-domain}) and 
 \[N^{2n}(r)\coloneqq  \mu^{-1}(\Omega_{N^{2n}(r)}), \text{ where }    \Omega_{N^{2n}(r)}=\{x\in \R^n_{\geq0}\mid  \exists   i=1,\ldots,n , \, x_i\leq   \pi r^2\}.\]
 
The following result provides another family of domains that one can confirm the agreement of ball-normalized capacities and cube-normalized capacities. It follows immediately from Theorem 1.2  in \cite{DR232} and Theorem 2 in \cite{GMV22}. 

\begin{thm}\label{thm-ball-normalized}
 If $\Omega \subset \R^n_{\geq 0}$ is the moment image of a  monotone toric domain $X_\Omega\subset (\R^{2n},\omega_{\rm std})$, then all ball-normalized capacities coincide on  $(\T^n\times \Omega,\omega_{\rm can})$. Also, all cube-normalized capacities coincide on  $(\T^n\times \Omega,\omega_{\rm can})$. \end{thm}

Note that, as a special case of  the Theorem \ref{thm-ball-normalized}, one can take 
 $X_\Omega$ as the cube $P^{2n}(\sqrt{2/\pi})$  defined in (\ref{eq-cube}). Then the associated product domain 
 \[ \T^n\times \Omega_{P^{2n}\left(\sqrt{2/\pi}\right)}\simeq \T^n\times B_\infty(1) \subset T^*\T^n\]
  can be identified with $D_F^*\T^n$ where Finsler metric $F$ is the $\ell^1$-norm defined in (\ref{eq-l-p-norm}). Then Theorem \ref{thm-ball-normalized}  reproves   Theorem 1.4 in \cite{Jia93}  and also generalize it to cube-normalized capacities.  

\section{Proofs} \label{sec-proofs} 
\subsection{Preparations}
Given a product domain $\T^2\times A\subset (T^*\T^2,\omega_{\rm can} = d\lambda_{\rm can})$ where $A \subset \R^2$ is starshaped (as an element in $\mathcal{T}_{T^*\T^2}$ defined in (\ref{equ-set-prod-domain})), first let us investigate the Reeb dynamics on its boundary $(\T^2\times \partial A,  \lambda_{\rm can}|_{\T^2\times \partial A})$, following closely the approach of Section 2.1 in \cite{DGZ24}. By definition, the Reeb vector field simply denoted by $R$ is uniquely determined by the following equations:
\begin{equation}\label{eq-Reeb-eqs}
  \lambda_{\rm can}(R)=1,\quad d\lambda_{\rm can}(R,\cdot)=0.  
\end{equation} 
In coordinates $(q,p)=(q_1,q_2,p_1,p_2) \in T^*\T^2\simeq \T^2\times \R^2$ where $q = (q_1, q_2) \in \T^2=\R^2/\Z^2$ and $p = (p_1, p_2) \in \mathbb{R}^2$, we can express $\lambda_{\rm can} = \sum_{i=1}^2 p_i dq_i$  and $R(q,p):=(V_1(q,p),V_2(q,p),W_1(q,p),W_2(q,p)) \in T_{(q,p)}(T^*\T^2)\simeq T_{(q,p)}(\T^2\times \R^2)\simeq  \R_q^2\oplus  \R_p^2$.  Then the relations in (\ref{eq-Reeb-eqs}) translate to 
\begin{equation}\label{eq-Reeb-eqs-2}
\left\{
\begin{aligned}
& p_1 V_1(q,p)+p_2 V_2(q,p)=1,\\
& -V_1(q,p)\,dp_1 - V_2(q,p)\,dp_2
  + W_1(q,p)\,dq_1 + W_2(q,p)\,dq_2 = 0 .
\end{aligned}
\right.
\end{equation}
Note that along the boundary $\T^2\times \partial A$, the tangent space $T_{(q,p)}(\T^2\times \partial A)\subset T_{(q,p)}(\T^2\times \R^2)$ is spanned by the basis $\{ \partial_{q_1},\partial_{q_2}, v(p)\}$ where $\{\partial_{q_1},\partial_{q_2}\}$ forms the standard  basis of $T_q\T^2\simeq \R_q^2$ and $v(p)=(v_1(p),v_2(p))$ is the unit tangent vector  to $p\in \partial A \subset \R^2$, oriented clockwise along $\partial A$. Evaluate the second equation in (\ref{eq-Reeb-eqs-2}) on these basis elements $\partial_{q_1},\partial_{q_2},v(p)$, we obtain
\[W_1(q,p)=0, \quad W_2(q,p)=0,\quad V_1(q,p)v_1(p)+V_2(q,p)v_2(p)=0. \]
In other words, the dot product of $(V_1(q,p),V_2(q,p))$ and $v(p)$ is zero, therefore $(V_1(q,p),V_2(q,p))$ is colinear to the outer {\it unit} normal vector $n(p):=(n_1(p),n_2(p))$ at $p\in \partial A\subset \R^2$. Together with the first equation in  (\ref{eq-Reeb-eqs-2}), we conclude that
\begin{equation}\label{eq-Reeb-eqs-3}
    V_i(q,p)=  \frac{n_i(p)}{p_1 n_1(p)+p_2 n_2(p)},   \quad  W_i(q,p)= 0 , \quad i=1,2.
\end{equation} 
 For simplicity, denote $V(q,p)\in \R_q^2$ by $V (p)$ as it is $q$-independent, and then \begin{equation}\label{eq-Reeb}
   R(q,p) = \big(V(p),\,0\big),
\qquad
\phi^t_R(q,p) = \big(q + t\,V(p),\,p\big) 
\end{equation}
for any $(q,p)\in T^*\T^2 $ and $t\in \R.$ Moreover, since $A$ is assumed to be starshaped, $p_1 n_1(p)+p_2 n_2(p)>0$ for each $p \in \partial A$. 

Suppose $\gamma$ is a closed Reeb orbit of period  $T>0$ where $\gamma(0)=(q,p)\in \T^2\times \partial A$.  The formula of the Reeb flow in (\ref{eq-Reeb}) implies that $\gamma(t) = (q+ tV(p ),p ) \in \T^2\times \partial A$ for any $t\in[0,T]$. Hence, $\gamma$ is contained in the ``level set'' $\T^2\times \{p \}$.  Moreover, $\gamma$ being closed is equivalent to $q= q+T V(p) \in \T^2\simeq \R^2/\mathbb{Z}^2$, in other words, $T V(p) \in \Z^2.$ Then, by (\ref{eq-Reeb-eqs-3}), this is equivalent to the following condition, 
\begin{equation}\label{eq-cond-action}
    T V(p)=\left(\frac{Tn_1(p)}{p_1 n_1(p)+p_2 n_2(p)},\frac{Tn_2(p)}{p_1 n_1(p)+p_2 n_2(p)}\right)\in \Z^2.
\end{equation} 
Assume furthermore that $\gamma$ is simple in the sense that it is not a multiple cover of another closed Reeb orbit, then (\ref{eq-cond-action}) implies that $\frac{T}{p_1 n_1(p)+p_2 n_2(p)}>0$ is the smallest (positive) rescaling denoted by $c_p$ such that $c_p(n_1(p), n_2(p)) \in \Z^2.$
Note that this $c_p$ is uniquely determined by $(n_1(p), n_2(p))$. In this case, we have 
\begin{equation}\label{eq-action}
T = c_p(p_1 n_1(p)+p_2 n_2(p))
  \end{equation}
which by definition equals to the action of $\gamma$ under the contact 1-form $\lambda_{\rm can}|_{\T^2 \times \partial A}$. 

\medskip

 Next, given  $X$ a general fiberwise star-shaped domain in  $(T^*\T^2,\omega_{\rm can}=d\lambda_{\rm can})$, we  will study the Ruelle invariant $\mathrm{Ru}_{T^*\T^2}(\partial X)$ of $(\partial X,\lambda_{\rm can}|_{\partial X})$ with respect to a fixed trivialization $\tau$ of the contact structure $\xi=\ker \lambda_{\rm can}$. Such a quantity was first introduced into symplectic geometry by Hutchings in \cite{Hut22}. 
For brevity, denote  $Y :=\partial X \subset T^*\T^2 $. Let us elaborate our trivialization $\tau$ and recall the definition of Ruelle invariant in this case. 

Viewing $Y$ as an $S^1$-bundle over $\T^2$, consider coordinates $(q, p) = (q_1,q_2,p_1,p_2)\in Y$ globally defined. Then the tangent space at point $y = (q,p) \in Y$ splits as  $T_{(q,p)}Y=T_q\T^2\oplus T_p(Y_q)$ where $Y_q$ denotes the fiber of $Y$ over $q$. For any vector $v=a_1\partial_{q_1}+a_2\partial_{q_2} + b\partial_p \in \xi$, by the defining property $\lambda_{\rm can}(v)=0$, we have 
\[0=a_1p_1+a_2p_2 \,\,\,\,\mbox{and}\,\,\,\, \mbox{no constraint on $b$}.\]
So, (by taking $a_1 = p_2$ and $a_2 = -p_1$) we can take a basis of $\xi$ by $e_1(y)=p_2\partial_{q_1}-p_1\partial_{q_2}$   and $e_2(y)\in T_p(Y_q)\subset \R_p^2$, the unique unit vector (with respect to the Euclidean norm on  $\R_p^2$) such that the orientation condition 
\[ (\omega_{\rm can})_y (e_1(y),e_2(y))>0\]
 holds. Then $(e_1(y),e_2(y))$ forms an ordered basis of $\xi_y$, which also defines a trivialization of  the  contact 2-plane field $\xi$. Explicitly, we have a bundle isomorphism $\tau: \xi \to \T^2 \times \R^2$ by $\tau(v) = \tau(ae_1(y)+be_2(y)) := (y, (a,b))$. 
Let $d\phi^t_R|_{\xi_y} : \xi_y \to \xi_{\phi^t_R(y)}$
be the linearized Reeb flow restricted to $\xi_y$.  For any  $y\in Y$ and $T\geq 0$, via $\tau$ (in particular, under the basis $(e_1(y),e_2(y))$), the linearized flow $\{d\phi^t_R|_{\xi_y} \}_{t\in [0,T]}$ can be realized as a path of matrices
\[ \Psi_t(y) \coloneqq \tau_{\phi^t_R(y)}\circ  d\phi^t_R|_{\xi_y}\circ \tau^{-1}_{ y }\]
  in $ \mathrm{Sp} (2)$ starting from the identity matrix. 
Together with the rotation number function $\rho \colon \widetilde{\mathrm{Sp}}(2) \to \R$ (cf.~Theorem 1 in \cite{BS10}), 
this yields a real number $ \rho\!\left(\{\Psi_t(y)\}_{t\in[0,T]}\right)$. By \cite{Rue85} (see also Section 3.2 in \cite{GE97}),
 it  determines  a well-defined  integrable function $\mathrm{rot}_\tau\colon Y \to \R$ as follows, 
\begin{equation}\label{eq-rot-rate}
   \mathrm{rot}_\tau(y)  \coloneqq \lim_{T\to+\infty} \frac{\rho\!\left(\{\Psi_t(y)\}_{t\in[0,T]}\right)}{T}.
\end{equation} 
Then for the  fixed trivialization $\tau$ as above, the \emph{Ruelle invariant} of $(Y=\partial X,\lambda_{\rm can}|_{Y})$ is defined by
\begin{equation}\label{eq-dfn-Ruelle}
    \mathrm{Ru}_{T^*\T^2}(Y) \coloneqq \int_{Y } \mathrm{rot}_\tau(y)\,\lambda_{\rm can}\wedge (d\lambda_{\rm can})
\end{equation} 
where $d\lambda_{\rm can}$ is understood by restriction on $Y$. 

\medskip

Now, based on the computations and definition above, for product domains $\mathcal{T}_{T^*\T^2}$ defined in (\ref{equ-set-prod-domain}), we have the following proposition:  
  \begin{prop}\label{prop-convex}
  For any product domain $\T^2\times A \subset (T^*\T^2,\omega_{\rm can}=d\lambda_{\rm can})$ where $A\subset\R^2$ is star-shaped,  $\lambda_{\rm can}|_{\T^2\times \partial A}$ is dynamically convex and $\mathrm{Ru}_{T^*\T^2}(\T^2\times \partial A)=0$.
\end{prop}

\begin{proof}[Proof of Proposition \ref{prop-convex}]
First we prove that for any star-shaped domain $A\subset \R^2$, $ \lambda_{\rm can}|_{\T^2\times \partial A}$ is always dynamically convex.
   
Let $\gamma$ be a closed  Reeb orbit of action $T>0$, starting at $(q ,p  )\in \T^2\times \partial A $.  
  Consider the natural projection map $\pi\colon \T^2\times \partial A\to \T^2$ to the first component. Then by (\ref{eq-Reeb}) and (\ref{eq-cond-action}) the closed loop
  \[\pi\circ \gamma=\left\{(\pi\circ \gamma)(t)=q+t V(p)\right\}_{t\in [0,T]}\subset \T^2\]
is non-constant, which implies that $\pi\circ \gamma$ is non-contractible in $\T^2$. Since
$\pi$ induces a natural projection $\pi_*\colon \pi_1(\T^2) \times \pi_1(\partial A) (= \pi_1(\T^2\times \partial A))\to \pi_1(\T^2)$, which is an isomorphism when restricted on $\pi_1(\T^2)$, we know that $[\pi \circ \gamma] = \pi_*([\gamma]) \neq 0$ implies $[\gamma] \neq 0$. In particular, $\gamma$ is non-contractible in $\T^2\times \partial A$.
In this way, we have shown that any non-constant closed Reeb orbit of $(\T^2\times \partial A, \lambda_{\rm can}|_{\T^2\times \partial A})$ is non-contractible. Then by definition of dynamical convexity at the beginning of Section \ref{sec-back-conx}, $ \lambda_{\rm can}|_{\T^2\times \partial A}$ is dynamically convex.

\medskip

Next we prove the Ruelle invariant of $(\T^2\times  \partial A,\lambda_{\rm can}|_{\T^2\times  \partial A})$ vanishes. In this case,  for any  point $y=(q,p)\in \T^2\times \partial A$, by the discussion right above this proposition, consider the following basis 
\begin{equation}\label{eq-basis}
   e_1(y)= p_2\partial_{q_1}  -p_1\partial_{q_2}\quad \text{and} \quad e_2(y)= -n_2(p)\partial_{p_1}+n_1(p)\partial_{p_2} 
\end{equation} 
  where $n(p)=(n_1(p),n_2(p))$ is the outer unit normal vector at $p\in \partial A$ and $\{\partial_{p_1},\partial_{p_2}\}$ is the standard  basis of $T_p\R^2\simeq \R_p^2$. For any $v\in T_{(q,p)}(T^*\T^2)$, under the standard  basis  we can write $v=v_{q_1}\partial_{q_1}+v_{q_2}\partial_{q_2}+v_{p_1}\partial_{p_1}+v_{p_2}\partial_{p_2}$. 
To compute the differential $(d\phi^t_R)_y(v)$, we locally take a smooth path $r(s)\colon (-\ep,\ep)\to \T^2\times \partial A$ for $\ep>0$ sufficiently small  such that $r(0)=y=(q,p)\in \T^2\times \partial A$ and $r'(0)=v$. For brevity we denote 
\[ r(s)= (q(s),p(s))=(q_1(s),q_2(s),p_1(s),p_2(s))\in \T^2\times \partial A\subset T^*\T^2.\]
   For any $s\in (-\ep,\ep)$ we have
   \begin{equation*}
       \begin{split}
           r(s)&=r(0)+sv+o(s)\\&=(q_1+s v_{q_1},q_2+sv_{q_2},p_1+sv_{p_1},p_2+sv_{p_2})+o(s)
       \end{split}
   \end{equation*}
where the approximation term $o(s) $ exists to guarantee that $r(s)\in \T^2\times \partial A.$
Then by (\ref{eq-Reeb-eqs-3}) and (\ref{eq-Reeb}) we have
\[\phi^t_R(r(s))=(q_1(s)+t\cdot\theta_1(s),q_2(s)+t\cdot\theta_2(s),p_1(s),p_2(s)),\]
where 
\[\theta_1(s)=\frac{ n_1(p(s))}{p_1(s)n_1(p(s))+p_2(s)n_2(p(s))}, \quad\theta_2(s)=\frac{ n_2(p(s))}{p_1(s)n_1(p(s))+p_2(s)n_2(p(s))}.\]
Note that the denominator of  $\theta_1(s)$ and  $\theta_2(s)$ can be simplified as
\[p_1(s)n_1(p(s))+p_2(s)n_2(p(s))=p_1n_1(p(s))+p_2n_2(p(s))+o(s)\]
 and converges to $p_1n_1(p)+p_2n_2(p)$ as $s\to0.$ Then by definition and the computations above we calculate the differential as 
\begin{equation}\label{eq-differential}
    \begin{split}
        (d\phi^t_R)(v)&=\lim_{s\to 0}\frac{\phi^t_R(r(s))-\phi^t_R(r(0))}{s}\\&= (  v_{q_1},  v_{q_2}, v_{p_1}, v_{p_2})\\&+\left( t\cdot\lim_{s\to 0}\frac{\theta_1(s)-\frac{n_1(p)}{p_1n_1(p)+p_2n_2(p)}}{s},t\cdot\lim_{s\to 0}\frac{\theta_2(s)-\frac{n_2(p)}{p_1n_1(p)+p_2n_2(p)}}{s} ,0,0\right).
    \end{split}
\end{equation}
By simplifying terms in the above equation, we have
\begin{equation}\label{eq-diff-1}
    \lim_{s\to 0}\frac{\theta_1(s)-\frac{n_1(p)}{p_1n_1(p)+p_2n_2(p)}}{s}=\frac{(n_1(p)n_2(p(s))'|_{s=0}-n_1(p(s))'|_{s=0}n_2(p))(-p_2)}{(p_1n_1(p)+p_2n_2(p))^2}
\end{equation}
where $(\cdot)'$ denotes the derivative with respect to variable $s$. Similarly we have
\begin{equation}\label{eq-diff-2}
    \lim_{s\to 0}\frac{\theta_2(s)-\frac{n_2(p)}{p_1n_1(p)+p_2n_2(p)}}{s}=\frac{(n_1(p)n_2(p(s))'|_{s=0}-n_1(p(s))'|_{s=0}n_2(p))( p_1)}{(p_1n_1(p)+p_2n_2(p))^2}.
\end{equation}
We denote the term 
\begin{equation}\label{eq-function-f}
    f_v(p)=\frac{ n_1(p)n_2(p(s))'|_{s=0}-n_1(p(s))'|_{s=0}n_2(p)  }{(p_1n_1(p)+p_2n_2(p))^2}
\end{equation}
  appearing in both  of the equations (\ref{eq-diff-1}) and (\ref{eq-diff-2}) above and it depends on the choice of  $v$. Then for any $v\in \xi_y$, under the basis $(e_1(y),e_2(y))$ in (\ref{eq-basis}), $v$ can be expressed as
  \[ v=v_qe_1(y)+v_p e_2(y)=v_qp_2\partial_{q_1}-v_qp_1\partial_{q_2}-v_pn_2(p)\partial_{p_1}+v_pn_1(p)\partial_{p_2}.\]
  Then based on the computation in (\ref{eq-differential}), if $v=e_1(y)$, $(d\phi^t_R)|_{\xi_y}(v)=e_1(y)$. Meanwhile,   if $v=e_2(y)$, 
  \[(d\phi^t_R)|_{\xi_y}(v)=e_2(y)-tf_{e_2(y)}(p)e_1(y).\]
    Representing this by a matrix
with respect to the oriented basis $(e_1(y), e_2(y))$,  we have
\begin{equation}\label{eq-trivialization}
   \left.d \phi_{t}\right|_{\xi_{y}}=\left(\begin{array}{cc}
1 & -tf_{e_2(y)}(p)   \\
0 & 1
\end{array}\right)\in \mathrm{Sp}(2). 
\end{equation} 
Therefore for any  $y\in \T^2\times \partial A$ and $T\geq 0$, the linearized differential $\{\left.d \phi_{t}\right|_{\xi_{y}}\}_{t\in [0,T]}$ is realized as a continuous path of 
shear matrices.  
In particular, none of these matrices contributes for rotations,  hence the rotation number
 $ \rho\!\left(\{\Psi_t(y)\}_{t\in[0,T]}\right)=0$  for all  $y\in \T^2\times \partial A$ and $T>0.$
Therefore, by (\ref{eq-rot-rate}), $\mathrm{rot}_\tau(y) \equiv 0$
 for all  $y\in \T^2\times \partial A.$
Finally, by definition of the Ruelle invariant in (\ref{eq-dfn-Ruelle}),
\[
\mathrm{Ru}_{T^*\T^2}(\T^2\times \partial A)
= \int_{\T^2\times \partial A} \mathrm{rot}_\tau(y)\,\lambda_{\rm can}\wedge (d\lambda_{\rm can})
= \int_{\T^2\times \partial A} 0\cdot\lambda_{\rm can}\wedge (d\lambda_{\rm can}) = 0.
\]
This proves the second part of Proposition \ref{prop-convex}. 
\end{proof}
\begin{remark} \label{rmk-ru-smooth}
   By Proposition 2.2 in \cite{DGZ24}  (see also  Proposition 1.11 of \cite{Hut22}) we have, for any toric domain $X_\Omega\subset (\R^4,\omega_{\rm std})$,  its
Ruelle invariant is given by
\begin{equation}\label{eq-ru-ab}
    \mathrm{Ru}(X_\Omega)=a+b.
\end{equation}
Here, $a$ and  $b$ are intercepts of the moment image $\Omega=\mu(X_\Omega)\subset \R^2_{\geq 0}$ with coordinate axes, where $\mu$ is defined in (\ref{eq-toric-domain}). The  difference between (\ref{eq-ru-ab}) and the conclusion in Proposition \ref{prop-convex} comes from the fact that the trivialization $\tau$ in the discussion above Proposition \ref{prop-convex} is globally defined  on $\T^2\times \partial A$, while for toric domain $X_\Omega\subset (\R^4,\omega_{\rm std})$ extra counting of rotation along Reeb orbits on the $z_1$-plane and the $z_2$-plane is needed (see Section 2.2 of \cite{DGZ24}) and contributes $a$ and $b$ respectively. 
\end{remark}
 Note that all the computations on product domain $\T^2\times A$ above is based on the assumption that $\partial A$ is of at least $C^1$-regularity (so that the normal vector is always well-defined on $\partial A$). In many situations of interest, however, the domain $A$ does not have a smooth boundary. Then we consider a sequence of smooth approximations $\{(\T^2\times \partial A_\ep,\lambda_{\rm can}|_{\T^2\times \partial A_\ep})\}_{\ep>0}$ and define the systolic ratio   $\rhoT(\T^2\times A)$ to be the limit  $\lim_{\ep\to0}\rhoT(\T^2\times A_\ep)$. In the following, we provide an example of such approximation and apply it to compute the systolic ratio.
    \begin{figure}[h]
        \centering
\includegraphics[width=0.75\linewidth]{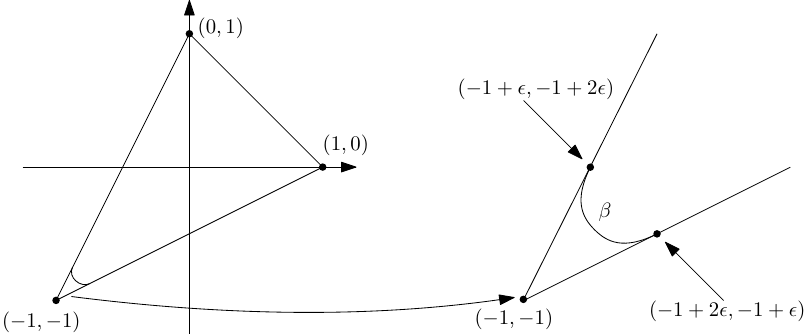}
        \caption{Smoothing triangle $A$ at vertex $(-1, -1)$.}
 \label{fig:example_(5)}
    \end{figure}
\begin{ex}\label{ex-perturb-cal} Consider triangle $A\subset \R^2$ with vertices $(-1,-1),(1,0),(0,1)$ as illustrated in the left part of Figure \ref{fig:example_(5)}. Denote by $\gamma$ a simple closed  Reeb orbit with $\gamma(0)=(q,p)\in \T^2\times \partial A$.   First consider the case where  $p\in \partial A$ which is not a vertex. There are three possibilities:
\begin{itemize}
\item[(i)] Point $p \in \partial A$ locates on the line $y=x-1$, where $n(p)=\left(\frac{\sqrt{2}}{2},\frac{\sqrt{2}}{2}\right)$.  Then $c_p=\sqrt{2}$ is the smallest positive rescaling   such that $c_pn(p)\in \Z^2$. By (\ref{eq-action}), the corresponding action 
\begin{align*}
T=c_p (p_1n_1(p)+p_2n_2(p))= \sqrt{2} \cdot (p_1, p_2) \cdot \left(\frac{\sqrt{2}}{2},\frac{\sqrt{2}}{2} \right) = 1. 
\end{align*}
Note that this holds for any point $p$ on this side.
\item[(ii)] Point $p \in \partial A$ locates on the line $y=\frac{1}{2}(x-1)$, where $n(p)=\left(\frac{\sqrt{5}}{5},-\frac{2\sqrt{5}}{5}\right)$. Then $c_p=\sqrt{5}$ is the smallest positive rescaling such that $c_pn(p)\in \Z^2$. By (\ref{eq-action}), the corresponding action 
\[ T=c_p (p_1n_1(p)+p_2n_2(p))= \sqrt{5} \cdot (p_1, p_2) \cdot \left(\frac{\sqrt{5}}{5},-\frac{2\sqrt{5}}{5}\right) = 1. \]
Note that this holds for any point $p$ on this side.
\item[(iii)] Point $p \in \partial A$ locates on the line $y=2x+1$, where $n(p)=\left(-\frac{2\sqrt{5}}{5}, \frac{\sqrt{5}}{5}\right)$. Then $c_p=\sqrt{5}$ is the smallest positive rescaling such that $c_pn(p)\in \Z^2$. By (\ref{eq-action}), the corresponding action 
\[ T=c_p (p_1n_1(p)+p_2n_2(p))= \sqrt{5} \cdot (p_1, p_2) \cdot \left(-\frac{2\sqrt{5}}{5}, \frac{\sqrt{5}}{5}\right) = 1. \]
 Note that this holds for any point $p$ on this side.
\end{itemize}
Consequently for any  simple closed Reeb orbit $\gamma$, starting from $(q,p)$ with $p\in \partial A$ not a vertex admits action $T=1$. 

 On each vertex  where the normal direction is not well-defined, we will always adopt the following smooth approximation (see the right picture in Figure \ref{fig:example_(5)}). For a sufficiently small $\ep>0$, consider a smooth embedded curve $\beta(t)\colon [0,1]\to \R^2$ (depending on $\ep$) satisfying the following properties:
\begin{itemize}
\item[(i)]  the image $\{\beta(t)\}_{t\in [0,1]}\subset \R^2$ is symmetric with respect to the line $y=x$ where $\beta(0)=(-1+\ep,-1+2\ep)$, $\beta(1/2) =(-1+\ep,-1+\ep)$, and $\beta(1)=(-1+2\ep,-1+\ep)$; moreover, the image of $\beta$ is contained inside the triangle $\Delta_\ep$ with vertices $(-1+\ep,-1+2\ep),(-1+2\ep,-1+\ep),(-1,-1)$;
\item[(ii)] the new boundary $(\partial A\backslash \partial\Delta_\ep)\cup \beta \subset \R^2$ is also smooth, in particular, near points $(-1+\ep,-1+2\ep)$ and $(-1+2\ep,-1+\ep)$;
\item[(iii)] the new boundary $(\partial A\backslash \partial\Delta_\ep)\cup \beta$  bounds a convex star-shaped domain in $\R^2$, denoted by $A_\ep$ (which is still symmetric to the line $y=x$); 
\item[(iv)] the planar area gap between $A$ and $A_{\ep}$ in $\R_{\geq 0}^2$ is at most $\frac{3\ep^2}{2}$. 
\end{itemize}
By symmetry, let us focus on the ``lower'' part of $\beta$, i.e., $\{\beta(t)\}_{t\in [\frac{1}{2},1]}$. The outer unit normal vector smoothly changes from $(-\frac{\sqrt{2}}{2},-\frac{\sqrt{2}}{2})$ to $( \frac{\sqrt{5}}{5},-\frac{2\sqrt{5}}{5})$. Along the way, these outer unit normal vectors $n$ fall into the following region:
\begin{equation}\label{eq-normal}
\left\{(x,y)\in \R^2 \,\middle|\, y \leq x \,\,\mbox{and}\,\, y \leq -2x \right\}.
\end{equation}

For this part of the image of $\beta$, express $\beta(t) =(-1+\delta_1,-1+\delta_2)$ where $\ep\leq \delta_1\leq 2\ep$ and $0\leq \delta_2\leq \ep$.  If a closed Reeb orbit $\gamma \subset \T^2\times \{\beta(t)\}$ for some $t \in [\frac{1}{2}, 1]$, then with respect to the smallest positive  rescaling $c$ of normal vector $n$ such that $c \cdot n=:(m_1,m_2) \in \Z^2 \backslash \{(0, 0)\}$. Moreover, by  (\ref{eq-normal}), the rescaled normal vector $(m_1, m_2)$ satisfies $m_2 \leq m_1$ and $m_2 \leq - 2m_1$.  Then, for the estimation of the action $T$ of $\gamma$, we divide it into different cases as follows. 

\begin{itemize}
\item[(i)] If $m_1 \geq 1$, then $m_2 \leq -2$ and 
\begin{align*}
T =m_1(-1+\delta_1)+m_2(-1+\delta_2) & \geq (-1 + \delta_1) + (-m_2) (1- \delta_2) \\
& \geq (-1+ \delta_1) + 2(1-\delta_2) \geq 1 + \ep - 2\ep = 1-\ep. 
\end{align*} 
\item[(ii)] If $m_1=0$, then $m_2 \leq -1$ (since $(m_1, m_2) \neq (0, 0)$) and 
\[ T = m_1(-1+\delta_1)+m_2(-1+\delta_2) = (-m_2)(1 - \delta_2)  \geq 1-\delta_2 \geq 1- \ep. \]
\item[(iii)] If $m_1 \leq -1$, then $m_2 \leq -1$ and 
\begin{align*}
T =m_1(-1+\delta_1)+m_2(-1+\delta_2) & \geq (-m_1) (-1 + \delta_1) + (-m_2) (1- \delta_2) \\
& \geq (1- \delta_1) + (1-\delta_2) \geq 2 -2 \ep - \ep = 2-3\ep. 
\end{align*} 
\end{itemize}
By assuming $\ep$ is sufficiently small (say, $\ep \leq \frac{1}{2})$, we have $2-3\ep  \geq 1-\ep$. Therefore, we conclude that the action of any closed Reeb orbit starting at point $(q,\beta(t))\in \T^2\times \partial A$ for $q \in \T^2$ and $t \in [\frac{1}{2},1]$ is always no less than  $1-\ep$. By the symmetry of $\beta$ with respect to the line $y=x$,  this action lower bound also holds for $t\in[0,\frac{1}{2}]$.  

We carry out a similar smoothing process at vertex $(0,1)$, where the resulting smooth family of normal vectors $(m_1, m_2)$ satisfy $m_2 \geq m_1$ and $m_2 \geq \frac{-m_1}{2}$. Then the same argument above implies that the action of any close Reeb orbit near $(0,1)$ is no less than $1-2\ep$. Similarly, carry out a similar smoothing process at vertex $(1,0)$, then the same conclusion for $(0,1)$ holds for $(1,0)$ by the symmetry with respect to the line $y =x$. Again, denote by $A_{\ep}$ the resulting domain after all these three smoothing processes at $(-1,-1)$, $(0,1)$ and $(1,0)$. We know that $\partial A_{\ep}$ is smooth, and  
\[1-2\ep\leq \sys(\T^2\times \partial A_\ep)\leq 1.\]
  Moreover, near each vertex, the smoothing removes at most an area in order $\ep^2$, so the volume of $\T^2\times \partial A_\ep$ satisfies:
   \[3-\mathcal O (\ep^2) \leq \mathrm{vol}(\T^2\times \partial A_\ep,\lambda_{\rm can}|_{\T^2\times \partial A_\ep})=2 \mathrm{area}(A_\ep)\leq 3\]
   where $\mathcal O(\ep^2)$ represents a number in order $\ep^2$. 
   Therefore, by definition in (\ref{eq-sys-rat-1}), the systolic ratio of $\T^2\times A_\ep$ satisfies  
\[\frac{(1-2\ep)^2}{3 }\leq \rhoT(\T^2\times  A_\ep)\leq\frac{1}{3-\mathcal O(\ep^2)}  . \]
Letting $\ep\to 0$, we have $ \rhoT(\T^2\times  A_\ep)\to \frac{1}{3}$, which is defined as the value of $ \rhoT(\T^2\times  A )$ under this smoothing and approximation process.
\end{ex}

\subsection{Proof of Theorem \ref{thm-class}}\label{sec-thm-A}
Based on the   Proposition \ref{prop-convex} above, we  now  prove Theorem \ref{thm-class} by fully distinguishing  the subclasses in Figure \ref{fig:diagram}.    For simplicity,  we introduce the following lemma that comes from the cotangent bundle lift theorem (see for example Exercise 3.1.21 (i) in \cite{MS17}.)
\begin{lemma}\label{lemma-contact}
    For any product domain $X=\T^2\times A\subset T^*\T^2$, there exists an exact symplectic embedding  
    \[\Phi \colon (X,\omega_{\rm can})\hookrightarrow (T^*\T^2,\omega_{\rm can}=d\lambda_{\rm can})\]
    such that $\Phi$ restricts to boundary $(\partial X,\lambda_{\rm can}|_{\partial X})$ a strict contactomorphism to $(\Phi(\partial X),\lambda_{\rm can}|_{\Phi(\partial X)})$. Moreover, $\Phi(X)$ is not a product domain in $T^*\T^2$\end{lemma}
\begin{proof}
    For any diffeomorphism $f\colon \T^2\to \T^2$, we  define the cotangent lift of $f$ as follows, 
    \begin{equation}\label{eq-contangent-lift}
        \Phi_f\colon T^*\T^2\to T^*\T^2,\quad \Phi_f(q,p)\coloneqq(f(q),(df(q)^{-1} )^*p).
    \end{equation}
    For any point $(q,p)\in T^*\T^2$ and $v\in T_{(q,p)}(T^*\T^2)$, by definition $\lambda_{\rm can}|_{(q,p)}(v)=p(d\pi(v))$ where $\pi\colon T^*\T^2\to \T^2$ is the natural projection to the base. Then we have
    \begin{equation*}
    \begin{split}
        (\Phi_f^*\lambda_{\rm can})_{(q,p)}(v)
&= \lambda_{\rm can}\big|_{\Phi_f(q,p)}\!\left(d\Phi_f(v)\right)\\
& =\bigl((df_q^{-1})^{*}p\bigr)\!\left(d\pi\circ d\Phi_f(v))\right) \\
&= \bigl((df_q^{-1})^{*}p\bigr)\!\left(df_q\circ d\pi(v)\right)= p(d\pi(v))
 = \lambda_{\rm can}\big|_{(q,p)}(v)
    \end{split}
\end{equation*}
where we use  the fact that $\pi\circ \Phi_f=f\circ \pi$ in the third equality.
  Especially, since $\Phi_f^*\lambda_{\rm can}=\lambda_{\rm can}$, when restricting to $(\partial X,\lambda_{\rm can}|_{\partial X})$, the cotangent lift $\Phi_f$ gives a strict contactomorphism. 
  
Now, for a non-zero given $\ep>0$, let us specify the diffeomorphism $f: \T^2=\R^2/\Z^2\to \T^2=\R^2/\Z^2$ above as follows, 
    \[f_\ep(q_1,q_2)\coloneqq \left(q_1+\frac{\ep}{2\pi}\sin (2\pi q_2),q_2\right).\] 
Note that it is well-defined since when $q_2 = \widetilde{q}_2$ in $\R/\Z$, we have $\sin (2\pi q_2) = \sin (2\pi \widetilde{q}_2)$. Under the standard basis $\{\partial_{q_1},\partial_{q_2}\}$ of $T_q\T^2$, differentials $df_{\ep}(q)$ and $\bigl(df_\ep(q)^{-1}\bigr)^{\!*}$  are given by  
\[
df_\ep(q)
=
\begin{pmatrix}
1 & \ep \cos(2\pi q_2) \\[0.3em]
0 & 1
\end{pmatrix},
\quad
\bigl(df_\ep(q)^{-1}\bigr)^{\!*}
=
\begin{pmatrix}
1 & 0\\[0.3em]
-\ep\cos(2\pi q_2) & 1
\end{pmatrix}.
\]
Then the cotangent lift of $f_\ep$ in (\ref{eq-contangent-lift}) is given by
\begin{equation}\label{eq-exact-symp}
    \Phi_{f_\ep}(q_1,q_2,p_1,p_2)
=
\Bigl(
q_1 + \frac{\ep}{2\pi}\sin(2\pi q_2),\;
q_2,\;
p_1,\;
-\ep\cos(2\pi q_2)\,p_1 + p_2
\Bigr).
\end{equation}
Denote by $X_\ep := \Phi_{f_\ep}(X) = \Phi_{f_\ep}(\T^2 \times A)$. For any $q=(q_1,q_2)\in \T^2$, the corresponding fiber $ (X_\ep)_{q} := X_\ep\cap (\{q\}\times T_{q}^*\T^2)$ is
\begin{equation}\label{eq-fiber-shear}
 (X_\ep)_{q} = \{q\} \times  \begin{pmatrix}
1 & 0\\[0.3em]
-\ep\cos(2\pi q_2) & 1
\end{pmatrix}A.
\end{equation}
Since $\ep \neq 0$ and the fiber depends on $q_2$,
the fiber $ (X_\ep)_{q} $ from (\ref{eq-fiber-shear}) varies with respect to $q$. Therefore, $X_\ep$ is a not a product domain. 
\end{proof}
 \begin{remark}\label{rmk-lift-lem}
 \noindent (i)
Since $\Phi_{f_\ep}$ restricts to $(\partial X,\lambda_{\rm can}|_{\partial X})$  a strict contactomorphism,  the minimal action of Reeb orbits of $(X_\ep=\Phi_{f_\ep}(\partial X),\lambda_{\rm can}|_{X_\ep})$ is the same as the minimal action of $(\partial X,\lambda_{\rm can}|_{\partial X})$. As a strict contactomorphism,  $\Phi_{f_\ep}$ also preserves the volume. Therefore by definition $\Phi_{f_\ep}$ preserves the systolic ratio $\rhoT$.   Moreover,  all the closed Reeb orbits of $(X_\ep,\lambda_{\rm can}|_{X_\ep})$ are non-contractible since by Proposition \ref{prop-convex} the closed Reeb orbits of $(\partial X,\lambda_{\rm can}|_{\partial X})$  are non-contractible.

  \noindent (ii) Given a convex domain $A\subset \R^2$, for any linear map on $\R^2$ represented by a matrix $M\in \mathrm{GL}(2,\R)$, the image $M(A)$ is convex since
  \[tM(p_1)+(1-t)M(p_2)\in M(A)  \Longleftrightarrow tp_1+(1-t)p_2 \in A\]
  for any points $p_1,p_2\in A$ and $t\in [0,1].$ Therefore, if $A$ is convex,
      for any $q\in \T^2 $, the fiber $(X_{\ep})_{q}$ in (\ref{eq-fiber-shear}) is also convex. In particular, $\Phi_{f_\ep}$ maps a  codisc bundle of torus (with respect to a metric $g$ on $\T^2$) to another codisc bundle of torus (with respect to another metric $g' \in \T^2$). Here, two metrics $g$ and $g'$ are in fact related by the diffeomorphism $f_{\ep}$ in the proof of Proposition \ref{lemma-contact}. 
 % \noindent (ii) Another way to obtain a perturbation of flat metric $F$ on  codsic bundle $X=D_F^*\T^2$ so that all Reeb orbits on boundary remain non-contractible is by applying classical KAM theory. By a comment in  \cite{Ant20} due to Dmitri Burago, using  KAM theory (see for example  Theorem 3.1 in \cite{Way96}), any such a small enough perturbation preserves an invariant torus therefore preserves a geodesic foliation (given by  the perturbed Finsler metric) on $\T^2$. Consequently we obtain a perturbed metric on torus without contractible closed geodesics, which lifts to a  cosphere bundle of torus without   contractible  Reeb orbits.
 \end{remark}
Based on the lemma above we  prove Theorem \ref{thm-class} as follows.

\begin{proof}[Proof of Theorem \ref{thm-class}]

\noindent (1) Example in $\mathcal F_{\rm flat} \backslash \mathcal F_{\rm rev}$. Consider the product domain $X=\T^2\times A$  where $A\subset \R^2$ is a triangle with vertices $(0,1),(1,-1),(-1,-1) $ as shown in Figure \ref{fig:example_(1)}. The outer unit normal vector at $(\frac{1}{2},0)\in \partial A$ is $\left(\frac{2\sqrt{5}}{5},\frac{\sqrt{5}}{5}\right)$, therefore $\sqrt{5} \cdot \left(\frac{2\sqrt{5}}{5},\frac{\sqrt{5}}{5}\right)=(2,1)\in \Z^2.$   By (\ref{eq-action}),  there exists a closed Reeb orbit in $\T^2\times\{(\frac{1}{2},0)\}$ with action  $\frac{1}{2}\cdot
2+0\cdot 1=1$.  Thus $\mathrm{sys}(\partial X)\leq 1$ while $\mathrm{vol}(\partial X,\lambda_{\rm can}|_{\partial X})=4$, so $X$ satisfies (\ref{eq-codisc-sym-sys}). Moreover, since $X$ is a product domain, $\lambda_{\rm can}|_{\T^2\times \partial A}$ is dynamically convex by Proposition \ref{prop-convex}. Then by Definition \ref{dfn-sys-conx} it is systolically convex. Note that $A$ is convex but non-symmetric to origin, which implies that if we write   $X$   as a codisc bundle  of torus $D_F^*\T^2$ under some flat Finsler metric $F$, then $F$ is non-reversible.
\begin{figure}[h]
    \centering
    \includegraphics[width=0.35\linewidth]{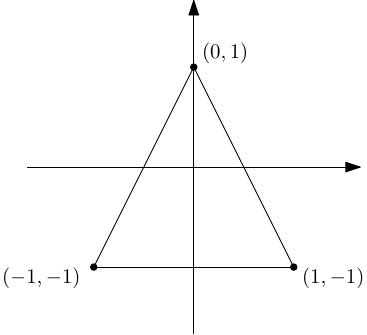}
    \caption{Example of (1).}
    \label{fig:example_(1)}
\end{figure}

(2) Example in $(\mathcal T_{T^*\T^2} \cap \mathcal S) \backslash \mathcal F_{\rm flat}$.  Consider the product domain $X=\T^2\times \Omega_p$ where $\Omega_p\subset \R^2 $ is a   domain that is symmetric with respect to $x$-axes and $y$-axes and is bounded by the curve $x^{1/p}+y^{1/p}=1$, for $p>1$, in the first quadrant (see Figure \ref{fig:example_(2)}). Obviously, it can not be expressed as a codisc bundle of torus $\T^2$ with respect to any Finsler metric since the fiber is not convex.
\begin{figure}[h]
    \centering
    \includegraphics[width=0.3\linewidth]{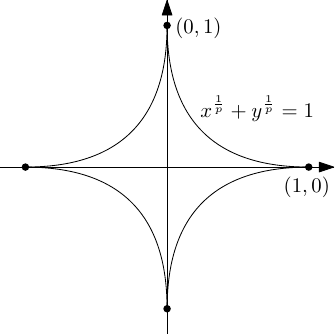}
    \caption{Example of (2), and also the example where (3) is built from.}
    \label{fig:example_(2)}
\end{figure}
Moreover, one can also estimate the systolic ratio as follows. The outer unit normal vector at $(2^{-p},2^{-p})\in \partial \Omega_p$ is $\left(\frac{\sqrt{2}}{2},\frac{\sqrt{2}}{2}\right)$ and $\sqrt{2}\cdot \left(\frac{\sqrt{2}}{2},\frac{\sqrt{2}}{2}\right)=(1,1)\in \Z^2$. By
    (\ref{eq-action}), there exists a closed Reeb orbit inside $ \T^2\times \{(2^{-p},2^{-p})\}$ with action  $2^{-p}+2^{-p}=2^{1-p}$. Therefore,
$   \mathrm{sys}(\partial X)\leq2^{1-p} $ while by considering the cube with vertices $(2^{1-p},0)$, $(0,2^{1-p})$, $(-2^{1-p},0)$, $(0,-2^{1-p})$ contained in $\Omega_p$, we have
\[\mathrm{vol}(\partial X,\lambda _{\rm can}|_{\partial X})= 2\mathrm{area}(\Omega_p)> {2^{4-2p}}.\]
Then $\rhoT(\T^2\times  \Omega_p) \leq \frac{(2^{1-p})^2}{2^{4-2p}} = \frac{1}{4}$. Therefore, by (\ref{eq-codisc-sym-sys}), we know that $X \in \mathcal S$.

\medskip

(3) Example in $\mathcal S \backslash (\mathcal F \cup \mathcal T_{T^*\T^2})$.  Consider the product domain $X=\T^2\times \Omega_p$, where $\Omega_p\subset \R^2$ is the one consider in (2) above (see Figure \ref{fig:example_(2)}). For any given $\ep \neq 0$, apply the exact symplectomorphism $\Phi_{f_\ep}$  defined in (\ref{eq-exact-symp})  on $\T^2\times \Omega_p$. By (i) in Remark  \ref{rmk-lift-lem}, the systolic ratio $\rhoT$ and the non-contractibility of all closed Reeb orbits are preserved under $\Phi_{f_\ep}$, so $\Phi_{f_\ep}(\T^2\times \Omega_p) \in \mathcal S$.  But it is not a product domain by the second conclusion of Lemma \ref{lemma-contact} so that $\Phi_{f_\ep}(\T^2\times \Omega_p)\notin \mathcal{T}_{T^*\T^2}$. By (ii) in  Remark \ref{rmk-lift-lem}, $\Phi_{f_\ep}(\T^2\times \Omega_p)$  is not fiberwise convex, therefore, $\Phi_{f_\ep}(\T^2\times \Omega_p) \notin \mathcal F$.

\medskip

(4) Example in $(\mathcal S \cap \mathcal F) \backslash \mathcal F_{\rm flat}$. Here we will provide two such examples. (i) Consider the product domain $X=\T^2\times B^2(1)$, where the fiber is the standard disk  $B^2(1)\subset \R^2$.   By (\ref{eq-action}), for the point $(0,1)\in \partial B^2(1)$ with normal vector $(0,1)$, there exists a closed  Reeb orbit in $\T^2\times\{(0,1)\}$ with action $1$. Therefore $\sys(\partial X)\leq 1$, $\mathrm{vol}(\partial X,\lambda_{\rm can}|_{\partial X})=2\pi$, so $\rhoT(X)\leq \frac{1}{2\pi}<\frac{1}{4}$. Now, applying $\Phi_{f_\ep}$   in (\ref{eq-exact-symp}) on $X$ for any non-zero $\ep$,  the image $\Phi_{f_\ep}(X)$ is still fiberwise convex by (ii) in Remark \ref{rmk-lift-lem}, so $\Phi_{f_\ep}(X) \in \mathcal F$. Moreover, by (i) in Remark \ref{rmk-lift-lem}, $\Phi_{f_\ep}$ preserves the  systolic ratio $\rhoT$ and the non-contractibility of all closed Reeb orbits, so $\Phi_{f_\ep}(X) \in \mathcal S$. However, by the second conclusion of Lemma \ref{lemma-contact}, $\Phi_{f_\ep}(X)$  is not product, so it is not in $\mathcal T_{T^*\T^2}$, hence not in $\mathcal F_{\rm flat}$. 

(ii) Consider any torus of revolution $\T^2$ embedded  inside $(\R^3,g_{\rm std})$, with induced metric $g_{\rm std}|_{\T^2}$ on it. By Theorem 3.5 in Chapter 6 in \cite{Bar66}, the metric $g_{\rm std}|_{\T^2}$ can not be flat. Moreover, by Chapter 7.9 (at the end of part A in Chapter 7)  in \cite{Gro99}, all closed geodesics on it are non-contractible. Finally, by Theorem 12.1 in \cite{Sab10}, the codisc bundle of this torus with respect to Riemannian metric $D_{g_{\rm std}|_{\T^2}}^*\T^2$ satisfies (\ref{eq-codisc-sym-sys}). Therefore this unit codisc bundle provides another desired example.  
% A codisc bundle $D_F^*\T^2$ that is systolically convex with $F$ non-flat.

\medskip
  
(5) Example in $\mathcal F_{\rm flat} \backslash \mathcal S$.  Consider the product domain $X=\T^2\times A$ where $A$ is a  triangle in $\R^2$ with vertices $(-1,-1), (1,0),(0,1)$ as in Example \ref{ex-perturb-cal} (see Figure \ref{fig:example_(5)}). Since the fiber $A$ is convex, (up to smoothing near vertices) $X=D_F^*\T^2$ with respect to some non-reversible flat Finsler metric $F$ on $\T^2.$   Moreover, by Example \ref{ex-perturb-cal},  $X$  has systolic  ratio $\rhoT(X)=\frac{1}{3}>\frac{1}{4}$. Therefore it is not systolically convex by violating  (\ref{eq-codisc-sym-sys}).
 
\medskip
 
%A codisc bundle $D_F^*\T^2$ that is dynamically convex but not systolically convex, where $F$ is flat. 

(6) Example in $(\mathcal D_{T^*\T^2} \cap \mathcal F) \backslash (\mathcal S \cup \mathcal T_{T^*\T^2})$. Start from the product domain $X=\T^2\times A$ as in (5) right above, where $A$ is a  triangle in $\R^2$ with vertices $(-1,-1), (1,0),(0,1)$ (see Figure \ref{fig:example_(5)}).  Applying perturbation $\Phi_{f_\ep}$  in (\ref{eq-exact-symp}), we know $\Phi_{f_\ep}(X)$ is not a product domain by the second conclusion in Lemma
 \ref{lemma-contact}. Moreover, $\Phi_{f_\ep}(X)$  is not systolically convex, since by (i) in Remark \ref{rmk-lift-lem},
 \[\rhoT(\Phi_{f_\ep}(X))=\rhoT(X)=\frac{1}{3}>\frac{1}{4}.\]
Again, by (i) in Remark \ref{rmk-lift-lem}, $\lambda_{\rm can}|_{\partial \Phi_{f_\ep}(X)}$ is dynamically convex (since all the closed Reeb orbits are non-contractible). Finally, by (ii) in Remark \ref{rmk-lift-lem} it  can be represented as a codisc bundle of torus $D_{F}\T^2$ with respect to some Finsler metric $F$.
 % A codisc bundle $D_F^*\T^2$ that is dynamically convex but not systolically convex, where $F$ is non-flat.

 \medskip

(7) Example in $\mathcal T_{T^*\T^2} \backslash (\mathcal S \cup \mathcal F)$. Again, let us start from the product domain $X=\T^2\times A$ as in (5) above, where $A$ is a  triangle in $\R^2$ with vertices $(-1,-1), (1,0),(0,1)$ (see Figure \ref{fig:example_(5)}). We  perturb   $A\subset \R^2$ into a  non-convex domain $A_\delta\subset \R^2$   illustrated in Figure \ref{fig:example_(7)} in the following way, where $0< \delta < < \frac{1}{2}$ (sufficiently small). 
First, remove a polygon in $A$ with vertices
\[(1,0),(0,1),\,\,(-1+2\delta,-1+\delta),\,\,(-1+ \delta,-1+2\delta),\,\,(-1+2\delta,-1+2\delta).\]
Then, for the remaining domain, adopt the similar smoothing process as in Example \ref{ex-perturb-cal} at each of these vertices, and denote the resulting domain by $A_\delta$. In particular, $A_{\delta}$ is not convex (so $X \times A_{\delta}$ is not in $\mathcal F$), but $\partial A_\delta$ is smooth. 
\begin{figure}[h]
    \centering \includegraphics[width=0.65\linewidth]{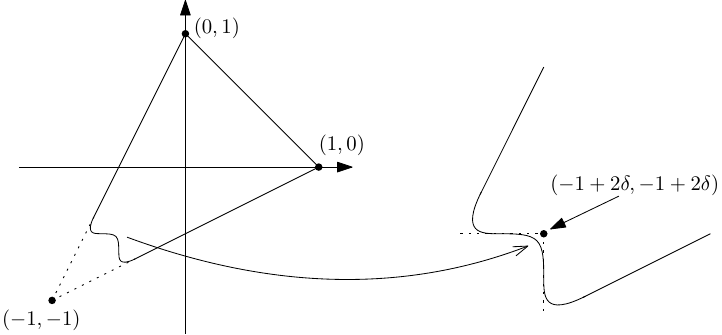}
    \caption{Example of (7).}
     \label{fig:example_(7)}
\end{figure}
By choosing perturbations near vertices small enough, we assume that line segment connecting  $(-1+\frac{4}{3}\delta, -1+2\delta)$ and $(-1+\frac{5}{3}\delta, -1+2\delta)$ is part of the boundary $\partial A_\delta$ (so, the unit normal vector along this line segment is $(0, -1)$); similarly, the line segment connecting  $(-1+2\delta,-1+\frac{4}{3}\delta)$ and $(-1+2\delta,-1+\frac{5}{3}\delta)$ is also a part of the boundary $\partial A_\delta$ (so, the unit normal vector along this line segment is $(-1, 0)$). Moreover, for any point $p=(p_1,p_2)\in \partial A_\delta$ near the corner $(-1+2\delta,-1+2\delta)$, say, $-1+\frac{5}{3}\delta\leq p_1\leq -1+2\delta$ and $-1+\frac{5}{3}\delta\leq p_2\leq -1+2\delta$, the outer unit normal vector $n(p)=(n_1(p),n_2(p))$ at $p$ satisfies $n_1(p),n_2(p)\leq 0$  by continuity. 

Consider the product domain $X = \T^2 \times A_{\delta} \in \mathcal T_{T^*\T^2}$. Suppose there exists a simple closed Reeb orbit in $\T^2\times \{p\}$ where $p = (p_1, p_2)$ satisfies $-1+\frac{4}{3}\delta \leq p_1 \leq -1 + 2\delta$ and $-1+\frac{4}{3}\delta \leq p_2 \leq -1 + 2\delta$. By  (\ref{eq-action}), let  $c_p$  be the smallest positive  rescaling  of $n(p)$    such that $c_p n(p)=(m_1,m_2)\in \Z^2 \backslash \{(0,0)\}$, then $m_1,m_2\leq 0$ and  the action of the Reeb orbit is estimated as follows, 
\[T=p_1m_1+p_2m_2\geq  \min\{-p_1,-p_2\} \geq 1-2\delta.\]
The same argument works for point $p\in \partial A_\delta$ near the corner  $(-1+\delta,-1+2\delta)$ and the corner $(-1+2\delta,-1+\delta)$. Therefore, we conclude that 
\[\sys(\T^2\times \partial A_\delta)\geq 1-\mathcal O(\delta),\quad \vol(\T^2\times \partial A_\delta,\lambda_{\rm can}|_{\T^2\times \partial A_\delta}) =  3 - \mathcal O(\delta^2)\]
 where $\mathcal O (\delta)$ and $\mathcal O (\delta^2)$ represent two scalars, depending on $\delta$ that are in order of $\delta$ and $\delta^2$ respectively. Therefore, when $\delta$ is sufficiently small, we have 
 \[\rhoT(\T^2\times A_\delta) \geq \frac{1}{3} - \mathcal O(\delta^2)> \frac{1}{4}\]
which shows that $\T^2 \times A_{\delta}$ is not systolically convex, not lying in $\mathcal S$. 

     \medskip

(8) Example in $\mathcal D_{T^*\T^2} \backslash (\mathcal T_{T^*\T^2} \cup \mathcal S \cup \mathcal F )$.  Consider the product domain  $ \T^2\times A_\delta $ in (7) right above. For any $\ep>0$, apply $\Phi_{f_\ep}$ in (\ref{eq-exact-symp}) and then the image $\Phi_{f_\ep}(\T^2\times A_\delta)$ is a desired example. Indeed, it is not a product domain by the second conclusion of Lemma \ref{lemma-contact}, so not in $\mathcal  T_{T^*\T^2} $. By (i) in Remark \ref{rmk-lift-lem} and the discussion in (7) right above, it is in $\mathcal D_{T^*\T^2}$ but not in $\mathcal S$. Finally, by (2) in Remark \ref{rmk-lift-lem}, it is not in $\mathcal F$. 

 \medskip

(9) Example in $\mathcal F \backslash \mathcal D_{T^*\T^2}$. By definition of dynamical convexity at the beginning of  Section \ref{sec-back-conx}, it suffices to construct a cosphere bundle of torus $(S_{\widetilde{g}}^*\T^2=\partial D_{\widetilde{g}}^*\T^2,\lambda_{\rm can}|_{S_{\widetilde{g}}^*\T^2})$  where there exists a contractible closed Reeb orbit with Conley-Zehnder index strictly less than $3.$ 

Let $(X,g)$ be a hyperbolic surface equipped with a complete 
Riemannian metric $g$ of constant Gaussian curvature $-1$. 
A (unit-speed) closed geodesic $\gamma\subset X$ is called a 
\emph{figure-eight closed geodesic} if it has exactly one transverse 
self-intersection point. By Theorem 4.2.4 (and the paragraph above it) in \cite{Bus92}, any such geodesic on $X$ is contained in a 
{\it pair of pants} $P\subset X$, that is, a compact surface 
diffeomorphic to a sphere with three open discs removed. 
The three boundary components of $P$ inside $X$  are determined by the two lobes of $\gamma$ and the 
complementary region (see Figure~\ref{fig:torus_figure_eight}).  
For further  details and discussion on pairs of pants, 
see Chapters~3 and~4 of~\cite{Bus92}.
\begin{figure}
    \centering
\includegraphics[width=0.55\linewidth]{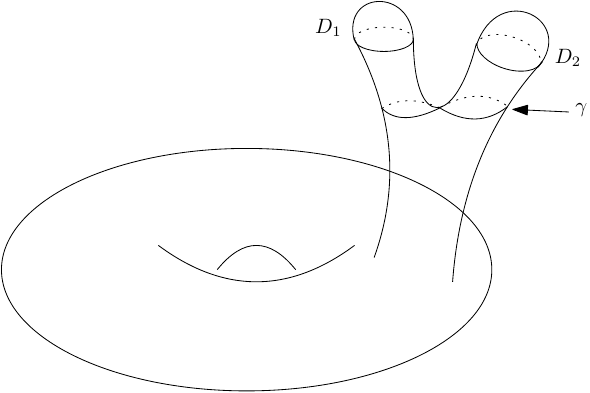}
    \caption{Surface $\Sigma$ with figure-eight closed geodesic $\gamma$.}
    \label{fig:torus_figure_eight}
\end{figure}

Starting from this  pair of pants $P$, we carry out the 
following topological construction.  Denote $\partial P=\gamma_1\sqcup \gamma_2\sqcup\gamma_3$ where each $\gamma_i$ is an embedded circle in  $X$, $1\leq i\leq 3$. 
We attach smoothly two embedded discs $D_1,D_2$ to  the boundary components $\gamma_1,\gamma_2$,
  respectively, and attach smoothly a  torus $\mathbb{T}^2 \setminus \mathring{D}$ with an open disc $\mathring{D}$ removed  along the remaining 
boundary component $\gamma_3$.  
The resulting closed smooth  surface, denoted by $\Sigma$, is  diffeomorphic to the 2-torus. Moreover, in this way the figure-eight curve $\gamma$ becomes contractible (through discs $D_1$ and $D_2$).

Then consider open tubular neighborhoods $U_i$ containing $\gamma_i$ such that locally $U_i$  is diffeomorphic to $\gamma_i\times (-\ep,\ep)$ and $U_i\cap \gamma_j=\emptyset$ for $j \neq i$. We define a smooth  cut-off function $\chi\colon \Sigma\to [0,1]$ such that
\[ \chi \equiv 1  \text{ on }  P\backslash{(U_1\cup U_2\cup U_3)}, \quad \chi \equiv 0  \text{ on }  \Sigma\backslash{(P\cup U_1\cup U_2\cup U_3)}. \]
Let $h$ be any complete Riemannian metric  on $\Sigma$. We  define a smooth metric on $\Sigma$ as follows, 
\[\widetilde{g}\coloneqq \chi g+(1-\chi)h\]
where for well-definedness we extend the metric $g$  on $P$ to the entire surface  $\Sigma$ trivially such that $g\equiv0$ on $\Sigma\backslash P$.
 By construction, $\widetilde{g}$ agrees with $g$ in a neighborhood of $\gamma$ in $P$, and in particular  $\widetilde{g}|_\gamma=g$.  
Consequently, the original figure-eight geodesic  remains a closed geodesic of $(\Sigma,\widetilde{g})$. Moreover, since $\gamma$ now bounds the attached discs, it is contractible in $\Sigma$. Thus we obtain a smooth Riemannian metric on the closed surface 
 $\Sigma$ that is diffeomorphic to $\mathbb{T}^2$, 
for which $\gamma$ remains  a closed geodesic under $\widetilde{g}$ and it is contractible.
 
Since $\gamma$ is an immersed $\tilde g$-geodesic with one
transverse self-intersection in $\Sigma$, it bounds two embedded discs, still denoted by $D_1,D_2\subset\Sigma$.
For the orientations on $D_1$ and $D_2$  induced by $\Sigma$, we have  
$\partial(D_1 - D_2) =\gamma$
as oriented $1$-chains  since the two loops of the figure-eight curve $\gamma$ carry opposite orientations when viewed as boundary components of  of $D_1$ and $D_2$. 

Consider the unit cosphere bundle $(S^*_{\tilde g}\Sigma,\lambda_{\rm can}|_{S^*_{\tilde g}\Sigma})$, which is  homeomorphic to $\T^2\times S^1$.  
By Theorem 1.5.2 (b) in \cite{Gei08}, the Reeb lift of $\gamma$ is the curve
$\{\widehat\gamma(t)\}_{t\in[0,T]}= \{\bigl(\gamma(t),\xi(t)\bigr) \}_{ t\in[0,T]} $ in $S^*_{\tilde g}\Sigma$,
where $\xi(t)=\tilde g_{\gamma(t)}(\dot{\gamma}(t),\cdot)$ is the unit covector dual to $\dot\gamma(t)$ with respect to
$\tilde g$.  The homotopy class of $\widehat\gamma$ in
$\pi_1(S^*_{\tilde g}\Sigma)\cong\pi_1(\Sigma)\oplus \pi_1(S^1)$ has two components:
a base component $[\gamma]\in\pi_1(\Sigma)$ and a fiber component
$k \in \pi_1(S^1)=\Z$ as
the winding number of the map $\xi:[0,T]\to S^1$ into the fiber circle.

Since $\gamma$ is contractible in $\Sigma$, we have $[\gamma]=0$. At each point $x\in \Sigma$, the Riemannian metric $\widetilde{g}_x$ induces a  linear isomorphism given by duality:
\[ \flat_x\colon T_x\Sigma\to T_x^*\Sigma, \quad v\to \widetilde{g}_x(v,\cdot).\]
Restricting $\flat_x$ to unit vectors gives an orientation-preserving  diffeomorphism 
\[\flat_x\colon S_x\Sigma \simeq S^1\to S_x^*\Sigma \simeq S^1\] between the unit tangent circle and the unit cotangent circles of $\Sigma$ at $x$. 
The winding number of   $\dot\gamma$  is exactly the rotation number of the figure-eight geodesic $\gamma$. Note that $\gamma$ is bounded by two embedded discs $D_1,D_2\subset \Sigma$ such that $\partial (D_1-D_2)=\gamma$.  Then we calculate:  
\[ \mathrm{rot}(\gamma=\partial (D_1-D_2) )=\mathrm{rot }(\partial D_1)-\mathrm{rot }(\partial D_1)=1-1=0.\]  
Therefore the winding number of $\xi$ is $0$ and  $\widehat\gamma$ is contractible.
 
Moreover, in our case, $\widetilde{g}|_\gamma$ is  hyperbolic therefore the Gaussian curvature $K\equiv -1$ on $\gamma$. Then we apply a similar argument as in Section 6.1 of \cite{SZ21}  to calculate the Morse index of $\gamma.$  Let $J=J(t)$ be a  Jacobi field along $\gamma$ that is normal to $\dot{\gamma}$, then the Jacobian equation is given by
\begin{equation}\label{eq-Jacobi}
\nabla_{\dot{\gamma}}\nabla_{\dot{\gamma}} J +R(J,\dot{\gamma})\dot{\gamma}=0.
\end{equation}
For surface case, for any $t\in[0,T],$ $R(J(t),\dot{\gamma}(t))\dot{\gamma}(t)$ is proportional to $J(t)$ so that we calculate 
\begin{equation}
    \begin{split}
       \langle R(J(t),\dot{\gamma}(t))\dot{\gamma}(t),J(t) \rangle&=|\dot\gamma(t)|^2|J(t)|^2\langle R(e_2(t),e_1(t))e_1(t),e_2(t)\rangle\\&=K\cdot |J(t)|^2=-|J(t)^2| .
    \end{split}
\end{equation}  
Here $\{e_1(t),e_2(t)\}$ is an orthonormal frame along $\gamma$ such that $e_1(t)=\dot\gamma(t)$ and  $e_2(t)$ is the unit normal vector field along $\gamma$. Since we choose $J(t)$  normal to $\dot \gamma(t)$ so that $J(t)=j(t)\cdot e_2(t)$ for some function $j\colon [0,T]\to \R$. We write  (\ref{eq-Jacobi})   as 
 \begin{equation}\label{eq-simp-Jacobi}
 \begin{split}
     \nabla_{\dot{\gamma}(t)}\nabla_{\dot{\gamma}(t)}J(t)- J (t)&=\nabla_{\dot{\gamma}(t)}\nabla_{\dot{\gamma}(t)}(j(t)\cdot e_2(t))- j(t)\cdot e_2(t)\\&=\nabla_{\dot{\gamma}(t)} (j'(t)\cdot e_2(t))- j(t)\cdot e_2(t)\\&=(j''(t)-j(t))\cdot e_2(t)=0
 \end{split}
 \end{equation}
 where we use the fact that $\nabla_{\dot{\gamma}(t)}  e_2(t)=0$ in the second and the third equality. 
 The two linearly independent solutions to  (\ref{eq-simp-Jacobi}) are $j_-(t)= e^t$ and  $j_+(t)=e^{-t}$. Let 
 \[ E(t)\coloneqq T\gamma(t)^{\perp }\oplus T\gamma(t)^{\perp }=\langle e_2(t)\rangle\oplus\langle e_2(t)\rangle.\]
  Then under the basis $e_2(t)$, the linearized Poincare map 
 \[P(T)\colon E(0)\to E(T), \quad P(t)(J(0),\nabla_{\dot \gamma(0)}J(0))=(J(T),\nabla_{\dot \gamma(T)}J(T))\]
 is given by $P(T)(j(0),j'(0))=(j(T),j'(T))$. Using the linearly independent solution $j_-$ and $j_+$,  we calculate that the eigenvalues of $P(T)$ are $e^{\pm T}$ where $T$ is the length of $\gamma.$ Therefore  neither one of these eigenvalues has norm one and $\gamma$ is a hyperbolic closed geodesic (cf.~Definition 6.2 in \cite{SZ21}). Then by Proposition 5 in \cite{Kli74}, since $J(t)\neq 0$ for any $t\in [0,T]$, the Morse index of $\gamma$ is   $\operatorname{ind}_{\mathrm{Morse}}(\gamma)=0$.

 Let $\gamma^*T\Sigma\to S^1$ be the rank-$2$ real vector bundle over  $S^1 $ and it is trivial sice $\Sigma$ is oriented. Therefore by  Theorem 1.2  in \cite{Web02},  the
Conley-Zehnder index $\mu_{\rm CZ}(\widehat\gamma)$ of the  Reeb lift $\widehat\gamma$ in $(S_{\widetilde{g}}^*\Sigma,\lambda_{\rm can}|_{S_{\widetilde{g}}^*\Sigma})$  satisfies $
\mu_{\rm CZ}(\widehat\gamma)
= -\operatorname{ind}_{\mathrm{Morse}}(\gamma)$.
 Since we arranged $\operatorname{ind}_{\mathrm{Morse}}(\gamma)=0$, we have
$\mu_{\rm CZ}(\widehat\gamma)=0 < 3.$
This proves the existence of a  closed  contractible Reeb orbit of
$(S_{\widetilde{g}}^*\Sigma,\lambda_{\rm can}|_{S_{\widetilde{g}}^*\Sigma})$ with Conley-Zehnder index strictly less than $3$, which implies that $\lambda_{\rm can}|_{S_{\widetilde{g}}^*\Sigma}$ is {\it not} dynamically convex.
\end{proof}
\subsection{Proof of Theorem \ref{thm-non-flat}}\label{ssec-proof-thm-B}
 Recall that given a contact manifold $(Y,\lambda=\ker\xi)$ with contact form $\lambda$ and a closed Reeb orbit $\gamma\colon [0,T]\to (Y,\lambda)$, the {\it linearized return map} $P_\gamma\colon \xi_{\gamma(0)}\to \xi_{\gamma(T)}$ is the restriction of the differential of the time $T$ Reeb flow. We say that a closed Reeb orbit $\gamma$ of  $(Y,\lambda)$ is {\it nondegenerate} if the linearized return map  $P_\gamma$  does not have $1$ as an eigenvalue. We say that the contact form $\lambda$ is nondegenerate if all closed Reeb orbits of $(Y,\lambda)$ are nondegenerate. As in Example 1.7 of \cite{Hut24}, for $\dim(Y)=3$, assume that the contact form $\lambda$ is nondegenerate, a simple closed Reeb orbit $\gamma$ is  {\it positive hyperbolic} if the eigenvalues of  $P_\gamma$ are all positive, {\it negative hyperbolic} if the
eigenvalues of $P_\gamma$ are all negative, and {\it elliptic} if the eigenvalues of $P_\gamma$ are on the unit circle in complex plane $\C$. 
In our case, let   $X=D_F^*\T^2$ be a codsic bundle of torus under flat Finsler metric $F$ on $\T^2.$ Equivalently, we write $X=\T^2\times A\subset T^*\T^2$ where $A$ is a convex   domain in $\R^2$.

By (\ref{eq-trivialization}),  the linearized Reeb flow $d\phi_R^t|_{\xi_y}$ on contact manifold $ (\T^2 \times \partial A,\lambda_{\rm can}|_{\T^2 \times \partial A})$ is given by a shear matrix for any $y=(q,p)\in\T^2\times A$ and $t>0.$ Note that by definition the linearized Reeb flow preserves the Reeb vector field so that $d\phi_R^t(R(y))=R(\phi^t(y))$. Then we can write the linearized Reeb flow $d\phi_R^t$ under the ordered basis  $(e_1(y),e_2(y),R(y))$  as
  \begin{equation}\label{eq-linearized}
          d\phi^t_R(y)=
 \begin{pmatrix}
  1  & -tf_{ {e_2(y)}}(p) &0\\
0    &1  &0 \\
0    &0  &1
\end{pmatrix}
  \end{equation}
  where $f_{{e_2(y)}}(p)$ is defined in (\ref{eq-function-f}).
Therefore, the linearized return map of Reeb flow is a unipotent shear and all the  eigenvalues are one so that  all closed Reeb  orbits of  $(\T^2 \times \partial A,\lambda_{\rm can}|_{\T^2 \times \partial A})$ are degenerate. 

Recall the definition of zeta function in Definition 1.5, Definition 1.11  and Definition 6.11 of \cite{Hut24}, we have the following proposition for $X$.
\begin{prop}\label{prop-zeta-func}
For  $X=\T^2\times A$, the associated  zeta function  satisfies
$\zeta(\T^2\times A,\lambda_{\rm can})=1.$
\end{prop}
\begin{proof}[Proof of Proposition \ref{prop-zeta-func}]
Denote the action spectrum of  $ (\T^2 \times \partial A,\lambda_{\rm can}|_{\T^2 \times \partial A})$ as
  \begin{equation}\label{eq-dfn-spec}
      \mathrm{Spec}(\T^2\times \partial A):= \{\mathcal{A}(\gamma)\mid  \text{$\gamma$ is a closed Reeb orbit of $(\T^2\times \partial A,\lambda_{\rm can}|_{\T^2\times \partial A}$)} \}.
  \end{equation} 
Given  contact manifold $(\T^2\times \partial A,\lambda_{\rm can}|_{\T^2\times \partial A})$ and any constant $e^L\notin \R\backslash \mathrm{Spec}(\T^2\times \partial A)$, similar to the proof of Lemma 7.1 in \cite{Hut24}, we construct an {\it $L$-approximation} (defined in Definition 6.1 of \cite{Hut24}) of $\T^2\times A$ as follows.

First we  rescale $A\subset \R^2$ with respect to the origin by $A_\ep\coloneqq (1-\ep)\cdot A$ and $A_\ep\subset \mathrm{int}(A)$ for any  $1>\ep>0$. Then we slightly rotate $A_\ep$ to  $\widehat{A}\subset \R^2$ such that all closed Reeb orbits of $(\T^2\times \partial\widehat{A},\lambda_{\rm can}|_{\T^2\times \partial\widehat{A}})$ are Morse-Bott. To achieve this, observe that by  (\ref{eq-linearized}),  
$\ker(d\phi^t_R-I)_{(q,p)} $ is of dimension $2$ if and only if $ f_{{e_2(y)}}(p)\neq 0$. By (\ref{eq-function-f}), $f_{{e_2(y)}}(p)=0$ if and only if  
\begin{equation}\label{eq-f-equal-0}
    n_1(p)n_2(p(s))'|_{s=0}-n_1(p(s))'|_{s=0}n_2(p) =0.
\end{equation} 
 Then the equation (\ref{eq-f-equal-0}) implies $ \mathrm{det}(n(p),n'(p(0))=0$ so that $n(p)$ is parallel to the vector $n'(p(0)).$ On the other hand, since $n(p(s))$ is a unit normal vector, by differentiating $\langle n(p(s)),n(p(s))\rangle=1$ we have $\langle n'(p(0)),n(p(0))\rangle =0$ therefore $n'(p(0))=0$.
Now we  view the normal vector $n(p)$ as the image of $p$ under the map $n\colon \partial A_\ep\to  S^1\subset \C $.
  For vector $v=p'(0)$, $\langle v\rangle$ spans $T_{p(0)} \partial A\simeq\R$ and the differential of $n$ is computed as:
\[d_{p(0)}n(v)=  \left.\frac{d}{ds}\right|_{s=0}n(p(s))=n'(p(0)).\]
 In other words, $\ker (f_{{e_2(y)}})=\mathrm{Crit}(n)$ is exactly  the set of critical points of $n.$    By Sard's theorem, the set of critical values $n(\mathrm{Crit}(n))$  has Lebesgue measure zero in  $S^1 $.  

Suppose there exists a closed Reeb orbit $\gamma$  of action $T>0$ with
$\gamma(0)=y=(q,p)\in \T^2\times \partial A$.
By  \eqref{eq-cond-action},  $T\,V(p)\in \Z^2\backslash\{0\}$,
 where $V(p)=(V_1(p),V_2(p))$ is defined in (\ref{eq-Reeb-eqs-3}).  
Since $V(p)$ is proportional to the outer unit normal vector $n(p)$ of $\partial A_\ep$, we have
\[n(p)\in S_{\Q}\coloneqq \left\{\left.\frac{(m,n)}{\sqrt{m^2+n^2}}\,\right|\, (m,n)\in \Z^2\backslash\{0\}\right\}\subset S^1 \]
as a countable set in $S^1$.   Let $\psi_\alpha\colon \R^2\to \R^2$ be the rotation map of $\R^2$  with respect to the origin by an angle $\mathrm{arg}(\alpha)\in [0,2\pi)$ where $\alpha $ is on the unit circle.  We consider  $n\circ \psi_\alpha(\mathrm{Crit}(n\circ \psi_\alpha))$ as the  set of critical values of $n\circ \psi_\alpha$ on $\psi_\alpha(A_\ep)$. Note that $n\circ \psi_\alpha(\mathrm{Crit}(n\circ \psi_\alpha))= \psi_\alpha(n(\mathrm{Crit}(n)))$  
  since rotation by $\alpha$ on $p\in A_\ep$ also rotates the normal vector $n(p)$ by  $\alpha.$  Then we consider the set  
\begin{equation}\label{eq-set-E-union}
   E\coloneqq \{\alpha\in S^1\mid \psi_\alpha(n(\mathrm{Crit}(n)))\cap S_\Q\neq \emptyset\} =   \bigcup_{u\in S_\Q} \{\alpha\in S^1\mid  \psi_{-\alpha}(u)\in n(\mathrm{Crit}(n))\}.
\end{equation} 
Note that given $u\in S_\Q$, the elements in the set $\{\alpha\in S^1\mid  \psi_{-\alpha}(u)\in n(\mathrm{Crit}(n))\}$   one-to-one correspond  to  $n(\mathrm{Crit}(n))$ therefore the set is of measure zero. Then by (\ref{eq-set-E-union}), as a countable union of measure-zero set, $E$ is of measure zero. Consequently, $S^1\backslash E$ is dense in $S^1$.  

Thus, for any $\ep>0,$ we can choose sufficiently small $\delta>0$ such that  $\psi_\delta(A_\ep)\subset \mathrm{int}(A)$ and $n\circ \psi_\delta(\mathrm{Crit}(n\circ\psi_\delta ))\cap S_\Q=\emptyset$. For any $p\in \psi_\delta(A_\ep)$ such that $n(p)\in S_\Q$, $n(p)\notin n\circ \psi_\delta(\mathrm{Crit}(n\circ\psi_\delta ))$. Therefore, on $\psi_\delta(A_\ep),$ $ f_{e_2(y)}(p) \neq 0$ for any   $p$ satisfying $n(p)\in S_\Q $, and $\ker (d\phi_R^t-I)_{(q,p)} $ is of dimension $2.$ For simplicity, we  denote $\psi_\delta(A_\ep)$ as $\widehat{A}\subset A$.

Next we consider  the submanifold 
\[N_T := \{(q,p)\in \T^2\times \partial \widehat{A}\mid \phi^T_R(q,p)=(q,p) \}.\]
By (\ref{eq-Reeb}), if there exists a closed Reeb orbit of action $T$ passing through $(q,p)\in \T^2\times \partial \widehat{A}$, then every point in the torus fiber $\T^2\times \{p\}$ lies on a closed Reeb orbit with the same action $T$. Therefore $N_T\subset \T^2\times \widehat{A}$ is a disjoint union of tori in the form of  $\T^2\times \{p\}$ and $T_{(q,p)}N_T\simeq T_q\T^2$ is of dimension $2$. By definition,   any vector $v\in T_{(q,p)}N_T$ satisfies 
\[(d\phi_R^t -I)(v)=d( \phi_R^t -\mathrm{id})(v)=0.\]
Therefore $\ker(d\phi^t_R-I)_{(q,p)}\supset T_{(q,p)}N_T $.  Then by dimension reason $\ker(d\phi^t_R-I))_{(q,p)}= T_{(q,p)}N_T$. By
Definition 1.7 in \cite{Bou02},  the contact form $\lambda_{\rm can}|_{\T^2\times \partial \widehat{A}}$ is  Morse-Bott.

By the discussion above, if there exists a closed Reeb orbit passing through $(q,p)\in \T^2\times \partial \widehat{A}$, then  $\T^2\times \{p\}$ is an $S^1$-family of   Morse-Bott Reeb orbits.  
Following the same Morse-Bott formalism of~\cite{Bou02},  
 we   approximate $\T^2\times A $ by perturbing 
 $\T^2\times \widehat{A}$ to replace  each of the  $S^1$-families of closed Reeb orbits  with action  less than or equal to $e^L$ by two simple Reeb orbits with   action less than $e^L$  such that one is elliptic and another one is positive hyperbolic and they are nondegenerate. 
Moreover, the perturbed domain, still denoted by $ \T^2\times \widehat{A}$, is an {\it L-approximation} of $\T^2\times  A$ by Definition 6.1 in \cite{Hut24}. 
 
 Then the Euler characteristic (defined in Definition 6.7 of \cite{Hut24}) of $(\T^2\times \widehat{A},\lambda_{\rm can})$ is given by $\chi^L( \T^2\times \widehat{A},\lambda_{\rm can})=0$ since elliptic orbits have positive Lefschetz sign while positive hyperbolic orbits have negative Lefschetz sign and they cancel out in $\chi^L( \T^2\times \widehat{A},\lambda_{\rm can})$.  Therefore by Proposition 6.3  and Definition 6.7 in \cite{Hut24},    the Euler characteristic of $(\T^2\times A,\lambda_{\rm can})$ is the same as the Euler characteristic of $(\T^2\times \widehat{A},\lambda_{\rm can})$: 
\begin{equation}\label{eq-Euler-zero}
    \chi^L(\T^2\times  A,\lambda_{\rm can})=\chi^L( \T^2\times \widehat{A},\lambda_{\rm can})=0.
\end{equation}  
Therefore, by  Definition  6.8  in \cite{Hut24} and the equation (\ref{eq-Euler-zero}) above, we derive that  $(\T^2 \times   A,\lambda_{\rm can})$ is  dynamically tame. By Definition 6.11 in \cite{Hut24}  and equation (\ref{eq-Euler-zero}),  we have $\zeta_{\mathring{CH}}(\T^2\times  A,\lambda_{\rm can})=0$. Combined with  Definition 5.2 in \cite{Hut24}, we conclude that the zeta function   $\zeta(\T^2\times A,\lambda_{\rm can})=1$.\end{proof}
 
\begin{remark}\label{rmk-zeta-toric}
By Proposition 1.18 of \cite{Hut24}, for any  toric domain  $X_\Omega\subset (\R^4,\omega_{\rm std})$  with $(a,0),(0,b)\in \partial \Omega$, the associated zeta function is  
    \[\zeta\left(X_{\Omega},\lambda_{\rm std}\right)=\frac{1}{\left(1-t^{a}\right)\left(1-t^{b}\right)}.\]
 The factors  $(1-t^a)^{-1}, (1-t^b)^{-1}$ in $\zeta\left(X_{\Omega},\lambda_{\rm std}\right)$ arise from  the elliptic Reeb orbits of $(\partial X_{\Omega},\lambda_{\rm std}|_{\partial X_\Omega})$ corresponding to the  points $(a,0), (0,b) \in \partial \Omega$, respectively.  In contrast, in the product case  $(\T^2\times \partial A,\lambda_{\rm can}|_{\T^2 \times \partial A})$ admits no elliptic Reeb orbits. Consequently, no such  contributions appear in the corresponding zeta function.  
\end{remark}
 To prove   Theorem \ref{thm-non-flat},  by Proposition \ref{prop-zeta-func}, we only need to find a Finsler metric $\widetilde{F}$ such that the codisc bundle $D_{\widetilde{F}}^*\T^2$ admits a different zeta function.   Then Theorem \ref{thm-non-flat} follows from   Theorem 1.16 in \cite{Hut24}.
\begin{proof}[Proof of Theorem \ref{thm-non-flat}]
The construction of Finsler metric $\widetilde{F}$ follows from the following three steps. 

\medskip

 Step one: Construct a torus of revolution and show that every closed geodesic on it is non-contractible.
Consider a torus $\T^2\simeq \R^2/\Z^2$ in coordinates $(\varphi,\theta)$ with metric
\[g_{(\theta,\varphi)}=a(\theta)^2d\varphi \otimes d\varphi+a(\theta)^2d\theta \otimes d\theta, \quad \text{for any $( \varphi,\theta)\in(\R/\Z)^2$}\]
where $a(\theta)>0$ is a smooth profile function on $\R/\Z$ with a unique minimal $a(\theta_0)=\min_{\theta\in \R/\Z}a(\theta)$ achieved by $\theta_0\in \R/\Z$  and $a'(\theta_0)=0.$ 

 Let $\{\sigma(t)=(\varphi(t),\theta(t)) \}_{t\in [0,T]}\subset \T^2$ be  any  (unit-speed) closed curve with period $T$. The Lagrangian of $\sigma$ is given by
\[\mathcal{L}(\sigma,\dot \sigma)\coloneqq\frac{1}{2}(a(\theta)^2\dot\varphi^2+a(\theta)^2\dot\theta^2).\]
 Then $\sigma$ is a closed geodesic if and only if it satisfies the Euler-Lagrangian equation given by
\begin{equation}\label{eq-geodesic-EL}
\left\{
\begin{aligned}
&\frac{d}{dt}\!\left(a(\theta)^2\,\dot{\varphi}\right)=0,\\[4pt]
&\frac{d}{dt}\!\left(a(\theta)^2\,\dot{\theta}\right)
- a(\theta) \frac{da(\theta)}{d\theta}\,\dot{\varphi}^2
- a(\theta)\frac{da (\theta)}{d\theta}\,\dot{\theta}^2
=0.
\end{aligned}
\right.
\end{equation}

Let  $[\sigma]=(m,n)\in H_1(\T^2;\Z)\cong \Z^2$ be the homology class of $\sigma$ with respect to the coordinate basis $(\varphi,\theta)$. Then we lift $g=a(\theta)^2d\varphi^2+a(\theta)^2d\theta^2 $ to the plane $\R^2$ and lift $\sigma$ to  a  geodesic segment  $\widetilde{\sigma}=(\widetilde{\varphi},\widetilde{\theta})\subset \R^2$ with $\widetilde{\varphi}(T)-\widetilde{\varphi}(0)=m$  as the winding number of $\varphi$ to $\R/\Z$ and $\widetilde{\theta}(T)-\widetilde{\theta}(0)=n$ as the winding number of $\theta $ to $\R/\Z.$
 By the first equation in (\ref{eq-geodesic-EL}), $a(\theta)^2\dot\varphi$ is a constant. Since $a(\theta)>0$ for any $\theta$, we conclude that the sign of $\dot\varphi$ is fixed.  If $\dot\varphi\neq 0$, then
we have
\[m=\widetilde{\varphi}(T)-\widetilde{\varphi}(0)=\int_{0}^T\dot{\widetilde{\varphi}}(t)dt=\int_{0}^T\dot{ \varphi }(t)dt\neq 0\]
since the sign of $\dot\varphi$ is fixed. 
If $\dot\varphi= 0$, then $\sigma$ is unit-speed implies 
\[1=a(\theta)^2(\dot\varphi^2+  \dot\theta^2)=a(\theta)^2\dot\theta^2.\]
 Therefore  $\dot\theta\neq 0$  so that the sign of $\dot\theta$ is fixed.   Similarly we have
\[n=\widetilde{\theta}(T)-\widetilde{\theta}(0)=\int_{0}^T\dot{\widetilde{\theta}}(t)dt=\int_{0}^T\dot{ \theta }(t)dt\neq 0.\]
Then the above discussion implies that $[\sigma]=(m,n)\neq 0\in H^1(\T^2;\Z)$ so that $\sigma$ is  always non-contractible.

\medskip

 Step two: Characterize the shortest closed geodesic and show that it is unique up to reparametrization.   Let $\{\gamma(t)=((t,\theta_0)\}_{t\in [0,1]}$ be a closed curve on $\T^2$ with $\varphi(t)=t$, $\theta(t)=\theta_0$. Then for the first equation  in (\ref{eq-geodesic-EL}), we have $\ddot{\varphi}=0$ and $\frac{d}{dt}\!\left(a(\theta)^2\,\dot{\varphi}\right)=0$. Moreover, $a'(\theta_0)( =\frac{d}{d\theta}|_{\theta=\theta_0}a(\theta))=0$ and $\dot\theta(t)=0$ implies the second equation in (\ref{eq-geodesic-EL}) vanishes. Therefore $\{\gamma(t)\}_{t\in [0,1]}$ is a closed geodesic on $(\T^2,g)$ with length 
\[l_g(\gamma)=\int_{0}^1\sqrt{g(\dot{\gamma}(t),\dot{\gamma}(t))}dt=a(\theta_0).\]
We  show that $\gamma(t)=(t,\theta_0)$ is the unique (up to reparametrization) length-minimizing geodesic among all closed  geodesics.

For any closed geodesic $\{\sigma\}_{t\in [0,T]}\subset \T^2$, using $\sqrt{x^2+y^2}\ge |x|$ and the fact that $a(\theta)\ge a(\theta_0)$, we obtain
\begin{equation}\label{eq:length-lower-bound}
\begin{aligned}
l_g(\sigma)
&=\int_0^T\sqrt{a(\theta(t))^2\,\dot{\varphi}(t)^2+a(\theta(t))^2\,\dot{\theta}(t)^2} dt \\
&\ge \int_0^T a(\theta(t))\,|\dot{\varphi}(t)| dt
\ge a(\theta_0)\int_0^T|\dot{\varphi}(t)| dt.
\end{aligned}
\end{equation}
Meanwhile we  have
\[
\int_0^T|\dot{\varphi}(t)| dt =\int_0^T \left|\dot{\widetilde{\varphi}}(t)\right| dt\ \ge\ \left|\int_0^T \dot{\widetilde{\varphi}}(t)  dt\right|
=| {\widetilde{\varphi}}(T)- {\widetilde{\varphi}}(0)|
=|m|.
\]
Combining the inequality above with \eqref{eq:length-lower-bound} yields the  estimate
\begin{equation}\label{eq:mn-bound}
l_g(\sigma)\ \ge\ |m|\,a(\theta_0).
\end{equation}
Similarly, we have  
\begin{equation}\label{eq-n-bound}
    l_g(\sigma)\ \ge\ |n|\,a(\theta_0).
\end{equation} 
In particular, since $(m,n)\neq 0$,  $l_g(\sigma)\geq a(\theta_0)$ with equality holds if and only if  $|(m,n)|=1$ and  $a(\theta(t))\equiv a(\theta_0)$  by (\ref{eq:length-lower-bound}). While $a(\theta)$ has a unique minimal at $a(\theta_0),$ this implies $\theta(t)\equiv \theta_0$.  Hence the lower bound of $l_g(\sigma)$ is sharp and is attained precisely by the parallel $\gamma(t)=(t,\theta_0)$ up to reparametrization.

\medskip
 
 Step three:  Perturb the metric 
 $g$  to obtain a non-reversible, nondegenerate Finsler metric  $\widetilde{F}$ that admits a unique shortest closed geodesic.
For simplicity we take $a(\theta_0) =1$ as the minimal of $a(\theta)$ discussed above.
For small  $ \ep>0$, we  define a non-reversible Randers metric on $\T^2$ as 
\[F (q,v)\coloneqq\sqrt{g_q(v,v)}-\ep d\varphi(v), \quad \text{for any $q=(\varphi,\theta)\in \T^2$ and $v\in T_q\T^2$}.\]
Here we require $\ep\|d\varphi\|_g=\ep a(\theta_0)=\ep<1$ so that $F(q,v)$ is everywhere positive.
Then by a standard fact of Randers metric (cf.~Exercise 11.3.4 in \cite{BCS00}),  the geodesic equation of $F$ is the same as the geodesic equation of Riemannian metric $g$ since $d\varphi$ is exact. The length of $\gamma$ now becomes 
\begin{equation}\label{eq-length-F}
    l_F(\gamma)=\int_{0}^1\left(\sqrt{g (\dot\gamma(t),\dot\gamma(t))}-\ep d\varphi(\dot\gamma(t))\right)dt=l_g(\gamma)-\int_{0}^1\ep d\varphi(\partial_{\varphi})dt=1-\ep.
\end{equation} 
For any other closed geodesic  $\{\sigma(t)\}_{t\in [0,T]}$ (including the inverse of $\gamma$), 
\begin{equation*}
    \begin{split}
        l_F(\sigma)&=\int_{0}^T\left(\sqrt{g(\dot{\sigma}(t),\dot{\sigma}(t))}-\ep d\varphi(\dot{\sigma}(t))\right)dt \\&= l_g(\sigma)-\ep\int_{\sigma}d\varphi= l_g(\sigma)-\ep \langle [d\varphi],[\sigma]\rangle= l_g(\sigma)-m\ep.
    \end{split}
\end{equation*} 
where  $[\sigma]=(m,n)\in H_1(\T^2;\Z)$.  Note that by (\ref{eq:mn-bound}), for $m\neq 0$ we have
\[l_g(\sigma)-m\ep\geq |m|-m\ep\geq |m|(1-\ep)\geq  1-\ep =l_F(\gamma) \]
 with equality if and only if $\sigma$ is a reparametrization of $\gamma.$  For $m=0$, $n\neq 0$ by the discussion in Step (i). By (\ref{eq-n-bound}) we have 
 \[l_F(\sigma)=l_g(\sigma)\geq |n| a(\theta_0)\geq 1\] 
 Therefore, we conclude that up to reparametrization the shortest closed geodesic  of $(\T^2,F)$ is $\gamma$ with length $1-\ep$ while the length of all other closed  geodesics of $(\T^2,F)$ has length $> 1-\ep.$

 We further apply  a standard bumpy metric perturbation (see Theorem 4 in \cite{RT22}) to this Finsler metric   so that after a $C^\infty$-small  perturbation, the metric is still non-reversible and 
all closed  geodesics are nondegenerate. Denote this metric by $\widetilde{F}$. By choosing such a perturbation small enough, up to $o(\ep)$-difference, $\gamma$    has length 
  $l_{\widetilde{F}}(\gamma)=  1-\ep $. Moreover, the length of all other closed  geodesics of $(\T^2,\widetilde{F})$ has length $> 1-\ep.$  

  \medskip
 
 By the argument above, $\widetilde{F}$ is bumpy so that the contact form
  $ \lambda_{\rm can}|_{S_{\widetilde{F}}^*\T^2}$ is nondegenerate and its Reeb flow corresponds to the
co-geodesic flow of $\widetilde{F}$.  
 Then the shortest closed geodesic $\gamma$ corresponds to  a closed simple Reeb orbit $\widehat{\gamma}$
with action $\mathcal A(\widehat{\gamma})=l_{\widetilde{F}}(\gamma)= 1-\ep$, and every other closed Reeb orbit (including the iterations of $\widehat{\gamma}$) has
action $> 1-\ep$. Then denote $\zeta(D_{\widetilde{F}}^*\T^2) =\zeta(S_{\widetilde{F}}^*\T^2,\lambda_{\rm can}|_{S_{\widetilde{F}}^*\T^2})$ and by Example 1.7 in \cite{Hut24}, the zeta function of $D_{\widetilde{F}}^*\T^2$  can be calculated as follows, 
\[
\zeta(D_{\widetilde{F}}^*\T^2) 
=
\prod_{\alpha \text{ simple Reeb orbit}}
\left\{
\begin{array}{ll}
\left(1 - t^{\mathcal{A}(\alpha)}\right)^{-1}, & \alpha \text{ elliptic},\\[0.3em]
1 - t^{\mathcal{A}(\alpha)},        & \alpha \text{ positive hyperbolic},\\[0.3em]
1 + t^{\mathcal{A}(\alpha)},        & \alpha \text{ negative hyperbolic}.
\end{array}
\right.
\]
In particular,   the factor corresponding to $\widehat{\gamma}$ in zeta function appears as:
\begin{itemize}
    \item [(i)] if $\widehat{\gamma}$ is elliptic, $\zeta(D_{\widetilde{F}}^*\T^2)$ has a factor $(1 - t^{1-\ep})^{-1}$;
    \item [(ii)] if $\widehat{\gamma}$ is positive (negative) hyperbolic, $\zeta(D_{\widetilde{F}}^*\T^2)$ has a factor $(1 \pm t^{1-\ep})$.
\end{itemize}
 In either way,  the lowest term $t^{1-\ep}$ in $\zeta(D_{\widetilde{F}}^*\T^2)$  can not be canceled. Indeed, all other factors involve powers $t^{\mathcal{A}(\alpha)}$ with $\mathcal{A}(\alpha)>\mathcal{A}(\widehat{\gamma})=1-\ep$, so they contribute only higher-order terms. Expanding the product $\zeta(D_{\widetilde{F}}^*\T^2)$ formally as a power series in $t$, the coefficient of $t^{1-\ep}$ is nonzero (equal to $\pm1$, depending on the type of $\widehat{\gamma}$).
 Therefore,
$\zeta(D_{\widetilde{F}}^*\T^2) \neq 1.$ 

Then  by  Theorem 1.16 in \cite{Hut24}, zeta function is invariant under exact symplectomorphisms. By Proposition \ref{prop-zeta-func}, $\zeta\equiv 1$ on codisc bundle of flat torus $\F_{\rm flat}$. Then the computation above shows that $D_{\widetilde{F}}^*\T^2$ is not exact symplectomorphic to  any codsic bundle of flat torus.
\end{proof}
%\begin{proof}[An alternative example]
%By Proposition 5.8 in \cite{SZ21}, there exists a Riemannian metric $g$
%  on  $1$-bulked torus $\T^2\subset \R^3$ such that:
%\begin{itemize}
%  \item [(i)] there is a unique (up to reparameterization) non-contractible simple closed $g$-geodesic $\{\gamma(t)\}_{t\in [0,1]}$ of
 %       length $l_g(\gamma)=\int_0^1{  \sqrt{g(\dot \gamma(t),\dot \gamma(t))} dt}=1$. In terms of coordinates $(q_1,q_2)$ of $\T^2\simeq \R^2/\Z^2$, $\gamma $ is represented by  $\{q_2=0\}$. 
 % \item [(ii)] Any closed geodesic different from $\gamma$ has length bounded below by a fixed constant $\gg1$. 
%\end{itemize}
%\end{proof}

\subsection{Proof of Theorem \ref{thm-large-scale-one} and Theorem \ref{thm-large-scale-two}}
 
We adopt the same strategy in \cite{Tra95} to prove the following proposition which builds the connections between toric domains $\mathcal{T}_{\R^{2n}}$ and product domains   $\mathcal{T}_{T^*\T^n} $  with $n\geq 2.$
\begin{prop}\label{prop-toric}
  Let  $X_\Omega\subset( \R^{2n},\omega_{\rm std})$ be a monotone toric domain. Then for any $0<\ep <1$,  there exist   symplectic embeddings:
    \begin{equation*}
        ((1-\ep)X_\Omega,\omega_{\rm std})\hookrightarrow (\T^n\times \mathrm{int}( \Omega),\omega_{\rm can})\hookrightarrow  (X_\Omega,\omega_{\rm std}).
    \end{equation*}
\end{prop}
\begin{proof}
    As in Proposition 4, Section 4.4 in \cite{HZ95} (also see Lemma 2.1 in \cite{Jia93}), given any point $(q_1,\ldots,q_n,p_1,\ldots,p_n)\in  \T^n\times \R^n_{> 0} \subset T^*\T^n  $, define a  symplectic embedding 
    $\Phi\colon (\T^n\times \R^n_{> 0} \subset T^*\T^n,\omega_{\rm can}) \to (\C^n\simeq \R^{2n},\omega_{\rm std})$ by
    \begin{equation}\label{eq-emb-one}
         \Phi(q_1, \ldots, q_n, p_1,\ldots, p_n)= \left( \sqrt{\frac{p_1}{\pi}}e^{2\pi i q_1},\ldots,\sqrt{\frac{p_n}{\pi} }e^{2\pi i q_n}\right).  
    \end{equation}
Moreover, when restricted to $\T^n\times \mathrm{int}(\Omega)$, the moment  image
\[\mu(\Phi(q_1, \ldots, q_n, p_1,\ldots, p_n))=(p_1,\ldots,p_n)\in \mathrm{int}(\Omega)\]
  where $\mu$ is the moment map defined in (\ref{eq-toric-domain}). Therefore, the embedding from $\T^2\times \mathrm{int}(\Omega)$ to $X_{\mathrm{int}(\Omega)}$ gives a symplectomorphism: 
\begin{equation}\label{eq-symplectomorphism}
    (\T^n\times \mathrm{int}( \Omega),\omega_{\rm can}) \simeq  (X_{\mathrm{int}(\Omega)},\omega_{\rm std}).
\end{equation} 
Conversely, similar to Proposition 5.2 in \cite{Tra95},
consider $SD(s)\subset B^2(s)$ to be the “slit” disc:
\begin{equation}\label{eq-slit-disc}
   SD(s)\coloneqq B^2 (s)\backslash\{ (x,y)\in B^2 (s) \mid x\geq 0 \text{ and } y=0\}. 
\end{equation} 
Then choose any large enough constant $R>0$ such that $X_\Omega\subset B^{2n}(R)$.  Let $\ep>0$ be small enough so that $(R-\ep)^2+\ep^3<R^2$ and
\begin{equation}\label{eq-shift-Omega}
    (1-\ep)^2\Omega+\pi  \ep^3 \coloneqq \{(p_1+\pi \ep^3,\ldots,p_n+\pi\ep^3)\mid (p_1,\ldots p_n)\in (1-\ep)^2\Omega\}\subset \mathrm{int}(\Omega).
\end{equation} 
 Such $\ep$ always exists since $\ep^3$ is of   higher order than $\ep^2$. Define 
 \begin{equation}\label{eq-area-pres}
     \sigma  \colon B^2( R-\ep)\to SD(R) 
 \end{equation}
 to be an area-preserving (with respect to the standard area form $dx \wedge dy$) embedding such that for any $(x,y)\in B^2(R-\ep),$
 \[|\sigma(x,y)|^2\leq x^2+y^2 +\ep^3. \] 
For an illustrative picture of $\sigma$, see Figure \ref{fig:slit-disc}.
\begin{figure}[h]
    \centering
    \includegraphics[width=0.7\linewidth]{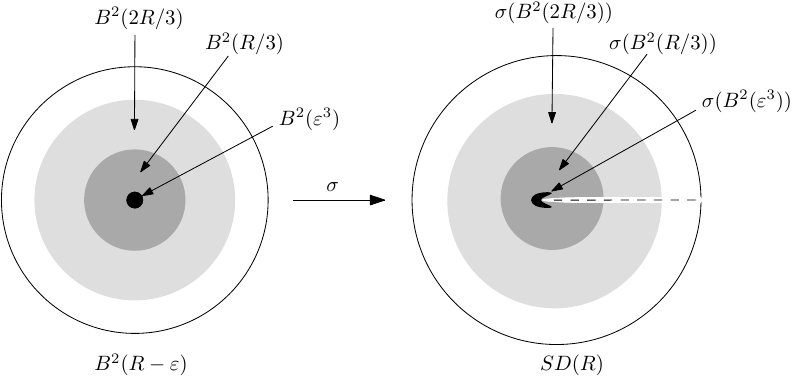}
    \caption{Area-preserving map $\sigma$.}
    \label{fig:slit-disc}
\end{figure}
An explicit construction of such $\sigma$ can be found in Appendix \ref{sec-app-B}. This induces a symplectic embedding (see an illustrative picture in Figure \ref{fig:embedding-Psi}) 
\begin{equation}\label{eq-emb-toric}
   \Psi\colon   (1-\ep)\cdot X_\Omega\hookrightarrow  X_{\mathrm{int}(\Omega)}, \quad \Psi (x_1,y_1, \ldots, x_n, y_n)=(\sigma(x_1,y_1),\ldots, \sigma(x_n,y_n)). 
\end{equation} 
\begin{figure}[h]
    \centering
    \includegraphics[width=0.65\linewidth]{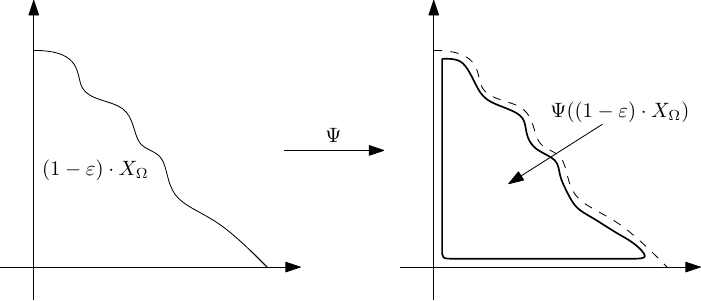}
    \caption{Symplectic embedding map $\Psi$ when dimension $2n=4.$}
    \label{fig:embedding-Psi}
\end{figure}
The map is symplectic since 
 \[\Psi^*\omega_{\rm std}=\Psi^* \left(\sum_{i=1}^n dx_i\wedge dy_i \right)=\sum_{i=1}^n\sigma^*(dx_i\wedge dy_i)=\sum_{i=1}^ndx_i\wedge dy_i=\omega_{\rm std}.\]
 
By our defining property  on $\sigma$, for any $(x_1,y_1,\ldots,x_n,y_n)\in (1-\ep)X_\Omega$,  we have
 \[0< \pi|\sigma(x_i,y_i)|^2 \leq    \pi (x_i^2+y_i^2+\ep^3) \quad \text{for any $1\leq i\leq n$.} \]
 For any $(p_1,\ldots,p_n)\in \Omega$ and $(p_1',\ldots,p_n')\in \R^n_{\geq 0}$ satisfying $p_i'\leq p_i $ for all $i=1,\ldots,n$, we have $(p_1',\ldots,p_n')\in \Omega$, since  $X_\Omega$  is a monotone toric domain. 
By definition of $\mu$, $(\pi(x_1^2+y_1^2), \ldots, \pi (x_n^2+y_n^2))\in (1-\ep )^2\Omega$ and by (\ref{eq-shift-Omega}) we have
\[\mu(\Psi (x_1,   y_n,\ldots ,x_n,  y_n))=\left( \pi|\sigma(x_1,y_1)|^2 ,\ldots, \pi|\sigma(x_n,y_n)|^2 \right) ,\]
which falls into $(1-\ep)^2\Omega+\pi \ep^3\subset \mathrm{int}(\Omega)$ by monotone assumption on $X_\Omega$ and our choice of $\ep.$
   Therefore $\Psi$ indeed defines a symplectic embedding from
   \[((1-\ep)X_{\Omega},\omega_{\rm std})\hookrightarrow(X_{\mathrm{int}(\Omega)},\omega_{\rm std})\simeq (\T^n\times \mathrm{int}(\Omega),\omega_{\rm can}) \]
   as desired. 
\end{proof}
\begin{remark}
In the proof above, the monotone assumption on  the toric domain $X_\Omega$ of Proposition \ref{prop-toric} is to guarantee that  for any $(p_1,\ldots,p_n)\in \Omega$ and $(p_1',\ldots,p_n')\in \R^n_{\geq 0}$ satisfying $p_i'\leq p_i $ for all $i=1,\ldots,n$, we still have $(p_1',\ldots,p_n')\in \Omega$.  For a general toric domain that is not necessarily monotone, the conclusion in Proposition \ref{prop-toric} may fail. %Consequently the map $\Psi$ defined in (\ref{eq-emb-toric}) need not send $(1-\ep) X_\Omega$ into $X_{\mathrm{int}(\Omega)}.$
\end{remark} 
The above proposition implies the following result. 
\begin{cor}\label{cor-capacity-BM}
    For any monotone toric domains $X_\Omega,X_{\Omega'}\subset (\R^{2n},\omega_{\rm std})$ and any symplectic capacity $c$,  we have
    \begin{equation}\label{eq-cap-CBM}
        c(X_\Omega)=c(\T^n\times \Omega), \quad d_{\rm BM}(X_\Omega,X_{\Omega'})=d_{\rm BM}(\T^n \times \Omega,\T^n\times \Omega' ) 
    \end{equation} 
    where $d_{\rm BM}$ is the coarse Banach-Mazur distance defined in (\ref{eq-CBM}).
\end{cor}
\begin{proof}[Proof of Corollary \ref{cor-capacity-BM}]
     By Proposition \ref{prop-toric} and the monotonicity property of symplectic capacity, for any $1>\ep>0$   we have
    \[c((1-\ep)X_\Omega)\leq c(\T^n \times \mathrm{int}(\Omega))\leq c(X_\Omega).\]
By letting $\ep\to0$ we conclude that $c(\T^n \times \mathrm{int}(\Omega))=c(X_\Omega). $ Moreover, for any $\Omega\subset \R^n_{\geq 0}$, $\mathrm{int}(\Omega)\subset \Omega\subset (1+\ep)\mathrm{int}(\Omega)+v$ where $v$ is a small translation vector   with $\|v\|\ll \ep$, chosen so that  $(1+\ep)\mathrm{int}(\Omega)+v$ contains the origin. Applying monotonicity and conformality of capacities   and letting $\ep\to0$ (so that $v\to0\in \R^n$), it follows that  $c(\T^n \times \mathrm{int}(\Omega))=c(\T^n\times \Omega).$

 For the second equality in (\ref{eq-cap-CBM}), given   monotone toric domains $X_\Omega,X_{\Omega'}$, suppose there exists constant $C\geq 1$ and symplectic embeddings such that
\begin{equation}\label{eq-emb-CBM}
     (X_\Omega,\omega_{\rm std})\hookrightarrow (C\cdot X_{\Omega'},\omega_{\rm std}),\quad (X_{\Omega'},\omega_{\rm std})\hookrightarrow (C\cdot X_\Omega,\omega_{\rm std}) .
\end{equation} 
Then by Proposition  \ref{prop-toric}, there exist symplectic embeddings:
\begin{equation*}
    \begin{split}
       & (\T^n\times \mathrm{int}(\Omega),\omega_{\rm can})\hookrightarrow (X_\Omega,\omega_{\rm std})\hookrightarrow  (C\cdot X_{\Omega'},\omega_{\rm std})\hookrightarrow\left(\frac{C}{(1-\ep)}\cdot \left(\T^n\times \mathrm{int}(\Omega')\right),\omega_{\rm can}\right),\\
     &  (\T^n\times \mathrm{int}(\Omega'),\omega_{\rm can})\hookrightarrow (X_{\Omega'},\omega_{\rm std})\hookrightarrow  (C\cdot X_{\Omega },\omega_{\rm std})\hookrightarrow\left(\frac{C}{(1-\ep)}\cdot \left(\T^n\times \mathrm{int}(\Omega)\right),\omega_{\rm can}\right).
    \end{split}
\end{equation*}
By definition of $d_{\rm BM}$ in (\ref{eq-CBM}),  this implies
 \[d_{\rm BM}(\T^n\times \mathrm{int}(\Omega),\T^n\times \mathrm{int}(\Omega'))\leq \ln \left(\frac{C}{1-\ep}\right).\]
By letting $\ep\to 0$ we have
 \[d_{\rm BM}(\T^n\times \mathrm{int}(\Omega),\T^n\times \mathrm{int}(\Omega'))\leq \ln \left(C\right).\]
Since  $C$ was arbitrary among constants admitting embeddings as in~\eqref{eq-emb-CBM},  by definition of $d_{\rm BM}$, we have 
\begin{equation}\label{eq-dbm-up-bound}
    d_{\rm BM}(\T^n\times \mathrm{int}(\Omega),\T^n\times \mathrm{int}(\Omega'))\leq d_{\rm BM}(X_\Omega,X_{\Omega'}).
\end{equation} 
By a symmetric argument using Proposition \ref{prop-toric}, we obtain the other direction of (\ref{eq-dbm-up-bound}):
\begin{equation}\label{eq-CBM-toric-product}
    d_{\rm BM}(\T^n\times \mathrm{int}(\Omega),\T^n\times \mathrm{int}(\Omega'))= d_{\rm BM}(X_\Omega,X_{\Omega'}).
\end{equation} 
  By an argument analogous to the proof of the first equality in (\ref{eq-cap-CBM}), based on the inclusions  $\mathrm{int}(\Omega)\subset \Omega\subset (1+\ep)\mathrm{int}(\Omega)+v$,  we obtain
  \[d_{\rm BM}(\T^n\times \mathrm{int}(\Omega),\T^n\times \mathrm{int}(\Omega'))= d_{\rm BM}(\T^n\times  \Omega ,\T^n\times  \Omega' ).\]
  Together with (\ref{eq-CBM-toric-product}), we  prove the second equality in (\ref{eq-cap-CBM}).
\end{proof}

By applying Proposition \ref{prop-toric} and Corollary \ref{cor-capacity-BM} to $d_{\rm BM}$, we can prove Theorem \ref{thm-large-scale-one} as follows.
 \begin{proof}[Proof of Theorem \ref{thm-large-scale-one}]
   Given any $N \in \N_{>0}$,  by  Theorem 1.2  in \cite{DR23}, there exists a quasi-isometric embedding from Euclidean space $(\R^N,|\cdot|_\infty)$ to the set  of concave toric domains in $\R^4.$ For any $v\in \R^N$, denote the corresponding concave toric domain by $X_{\Omega_v}\subset (\R^4,\omega_{\rm std})$. By definition it is also a monotone toric domain.  Then consider the product domain $\mathbb{T}^2\times \Omega_v \in \mathcal{T}_{T^*\T^2}$.  By  Theorem 1.2  in \cite{DR23},  there exist some uniform constants $C,D>0$  such that
      for any $v,w\in \R^N$, the  coarse Banach-Mazur distance $d_{\rm BM}$ satisfies:
     \[  d_{\rm BM}(X_{\Omega_v},X_{\Omega_w})  \geq \frac{1}{C}\cdot|v-w|_{\infty}-D. \]
Then applying the second conclusion in   Corollary \ref{cor-capacity-BM},    
     \[d_{\rm BM}(\mathbb{T}^2\times  \Omega_v , \mathbb{T}^2\times  \Omega_w )= d_{\rm BM}(X_{\Omega_v},X_{\Omega_w})  \geq \frac{1}{C}\cdot|v-w|_{\infty}-D.\]
       Define $d_{I}$ as the {\it inclusion distance} between star-shaped domains in $\R^{2}$ (more generally, in $\R^n$ for later use)   by
     \begin{equation}\label{eq-dI}
        d_{I} (\Omega_v ,\Omega_w )  \coloneqq \inf \{\ln C \mid\Omega_v \to C \cdot \Omega_w, \Omega_w\to C \cdot \Omega_v\} 
     \end{equation}
     where $\to$ refers to inclusion and  the scaled domain
     \[C\cdot\Omega_v\coloneqq \{C\cdot x \mid x\in \Omega_v\}\]
     is obtained by uniform radial dilation from the origin.
     The upper bound comes from
     \[d_{\rm BM}(\mathbb{T}^2\times \Omega_v , \mathbb{T}^2\times  \Omega_w )\leq d_{I}( \Omega_v , \Omega_w ) \leq C\cdot|v-w|_{\infty}+D\] with the last inequality from  the defining property of $X_{\Omega_v},X_{\Omega_w}$ as shown in \cite{DR23}. In conclusion, the map
     \[v\in (\R^N,|\cdot|_{\infty})\to \T^2\times \Omega_v\in (\mathcal{T}_{T^*\T^2},d_{\rm BM})\] 
     gives our desired quasi-isometric embedding.
\end{proof}
 
Recall the shape invariant for product domains $\mathcal{T}_{T^*\T^n}$, introduced in \cite{Yak91,RZ21}, is defined as follows:
\begin{dfn}[see Section 2.2 in \cite{RZ21}]\label{dfn-shape}
  For any $n\geq 2$  and any domain $A\in \R^n$, the shape  invariant of $(\T^n\times A,\omega_{\rm can}=d\lambda_{\rm can})$ is defined as
    \begin{equation}\label{eq-dfn-shape}
        \operatorname{Sh}(\T^n\times A) :=
\left\{
a \in H^1(\T^n ; \mathbb{R})
\ \middle|\ 
\begin{array}{l}
\text{there exists a Lagrangian embedding } 
 \\[4pt]
f \colon \T^n  \to\T^n\times A,\text{ such that }  \\[4pt]
f^*|_{  H^1 (\T^n; \mathbb{R})} = \mathrm{id},\text{ and   } f^*([\lambda_{\rm can}]  ) =a
\end{array}
\right\}.
    \end{equation}
\end{dfn}
If we choose a basis $\{e_1,\ldots,e_n\} $ of $H^1(\T^n;\R)\simeq \R^n$ so that $a\in H^1(\T^n;\R)$ can be identified as a vector in $\R^n$.  Then the shape invariant defined above admits the following property:
\begin{lemma}[Theorem 2.4.1 in \cite{Yak91}]\label{lemma-shape}
    Let  $A\subset \R^n$ be a connected open subset and    $\T^n\times A\subset T^*\T^n$. Then $\operatorname{Sh}(\T^n\times A)=A.$
\end{lemma}
\begin{proof}[Proof of Theorem \ref{thm-large-scale-two}]
    Given any  open domains
    $A,B\subset  \R^n$,  denote $d_{\rm HBM}(\T^n \times A,\T^n\times B)=:\ln M\geq 0$. Then by Definition \ref{dfn-HBM}, for any $\ep>0$, there exist  exact $H^1$-trivial  embeddings:
    \[\varphi\colon \T^n \times A\xhookrightarrow{H^1-\text{trivial}  }\T^n\times ((M+\ep)\cdot B),\quad \psi\colon\T^n\times B \xhookrightarrow{H^1-\text{trivial}} \T^n\times((M+\ep)\cdot A).\]
For any $a\in \mathrm{Sh}(\T^n\times A)$, by (\ref{eq-dfn-shape}) there exists  a Lagrangian embedding $f$ with $f^*([\lambda_{\rm can}])=a$. Then the composition
\begin{equation}\label{eq-Lag-emb}
    \varphi\circ f\colon \T^n\to \T^n\times ((M+\ep)\cdot B)
\end{equation} 
  gives a Lagrangian embedding of $\T^n$, and
  \[(\varphi\circ f)^*([\lambda_{\rm can}])= f^* ( \varphi^*([\lambda_{\rm can}]))=f^*([\lambda_{\rm can}])=a\]
  where the second equality is due to the  exactness of $\varphi$. Moreover, $(\varphi\circ f)^*|_{H^1(\T^n;\R)}=\mathrm{id}$ since $\varphi$ is $H^1$-trivial. Therefore, by definition  
 we have  $a\in \mathrm{Sh}(\T^n\times ((M+\ep)\cdot B))$. By Lemma \ref{lemma-shape}, this implies that 
 $A\subset  (M+\ep)\cdot B .$
By a symmetric argument we have $B\subset  (M+\ep)\cdot A .$  In other words, by definition of the inclusion distance  in (\ref{eq-dI}), 
\[d_I(A,B)\leq \ln (M+\ep) \]
where $d_I$ is the inclusion distance defined in (\ref{eq-dI}). By letting $\ep\to0$ we have
  \begin{equation}\label{eq-HBM-inc-1}
        d_I(A,B)\leq \ln M= d_{\rm HBM}(\T^n \times A,\T^n\times B).
  \end{equation} 
  
On the other hand, any inclusion between constant  multiples of $A$ and $B$ can be regarded as inclusion between constant multiple of $\T^n\times A$ and $\T^n\times B$, therefore by definition of $d_{\rm HBM}$, we have the following estimation:
\begin{equation}\label{eq-HBM-inc-2}
      d_{\rm HBM}(\T^n \times A,\T^n\times B)\leq d_I(A,B).
\end{equation}
 Together   we conclude that $ d_{\rm HBM}(\T^n \times A,\T^n\times B)= d_I(A,B)$. Note that $A,B\subset \R^n$  can be chosen freely. Through a similar argument of Theorem A in  \cite{RZ21}, we prove (i) in Theorem \ref{thm-large-scale-two} as follows.

 \medskip
 
Since $n\ge 2$, we may choose two linearly independent cohomology classes in
$H^1(\T^n;\R)$, which spans a $2$-dimensional subspace, identified with $\C$. Within this subspace, for any fixed integer $N\geq 2$, consider $N$-pairwise non-collinear directions given by $\{e^{\sqrt{-1}\frac{\pi(j-1)}{N}}\}_{j=1}^N$. Then, given any vector $v=(v_1,\ldots,v_N)\in \R^N$, define
$2N$ points as follows, 
\[
p_j \coloneqq e^{v_j} e^{\sqrt{-1}\frac{\pi(j-1)}{N}} \in \C \,\,\,\,\mbox{and}\,\,\,\,
p_{j+N} \coloneqq -p_j\in \C,
 \,\,\,\,\mbox{for $j=1,\ldots,N$}.
\]
Let $P_v\subset \C$ be the polygon whose vertices are
$p_1,p_2,\ldots,p_{2N}$ in this cyclic order, with edges given by the line
segments $[p_j,p_{j+1}]$ (with indices taken modulo  $2N$). Then $P_v$ is a  star-shaped domain in $\R^2 $ (see an illustrative picture in Figure \ref{fig:piecewise-linear-domain} where we take $N=4$). 
\begin{figure}[h]
    \centering
    \includegraphics[width=0.7\linewidth]{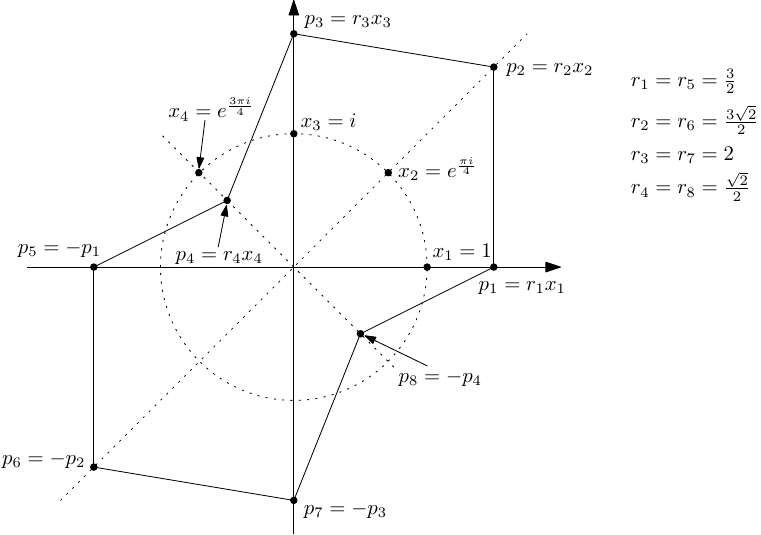}
    \caption{Star-shaped domain $P_v$.}
    \label{fig:piecewise-linear-domain}
\end{figure}
Moreover, for the complementary $(n-2)$-dimensional Euclidean subspace 
$\mathbb{R}^{n-2} \subset H^1(\mathbb{T}^n; \mathbb{R})$, 
consider the   closed higher-dimensional cube $[0,1]^{n-2} \subset \mathbb{R}^{n-2} $  and   define $A_v\coloneqq P_v \times[0,1]^{n-2} \subset \mathbb{R}^{n}$ to be the product domain, which is also star-shaped in $\R^{n}$ (after smoothing the corner). In this way, we associate to each $v \in \R^N$ a fiberwise star-shaped domain $\T^n \times A_v \subset T^*\T^n$. 

By the discussion above, we have 
 \begin{equation*}
     \begin{split}
        d_{\rm HBM}(\T^n \times A_v,\T^n\times A_w) & =d_{I}(A_v,A_w)\\&= \inf \{\ln C \mid e^{v_i}\leq C\cdot e^{w_i},  e^{w_i}\leq  C\cdot e^{v_i} \text{ for any $1\leq i\leq N $}  \}  \\
        & =\left|v-w\right|_\infty
     \end{split}
 \end{equation*}
 for any $v,w\in \R^N$. Therefore, such an association serves as the requested isometric embedding. 

 \medskip
 
Next, let us prove (ii) of Theorem \ref{thm-large-scale-two}. Given an embedding 
 $X_{\mathrm{int}(\Omega)}\xhookrightarrow{H^1-\text{trivial}}C\cdot X_{\mathrm{int}(\Omega')} $ 
 for some constant $C>0$,  by Definition \ref{dfn-HBM-toric} and symplectomorphism (\ref{eq-symplectomorphism}), it gives rise to an exact $H^1$-trivial embedding
 \[(X_{\mathrm{int}(\Omega)},\omega_{\rm std})\simeq (\T^n\times \mathrm{int}(\Omega),\omega_{\rm can})\xhookrightarrow{H^1-\text{trivial}} (\T^n\times \mathrm{int}(C\cdot \Omega'),\omega_{\rm can})\simeq (C\cdot X_{\mathrm{int}(\Omega')},\omega_{\rm std}).\]
 Therefore for any small $\ep>0$,  we have
 \begin{equation}\label{eq-comp}
     \T^n\times  (1-\ep)\Omega   \to  (\T^n\times \mathrm{int}(\Omega),\omega_{\rm can})\xhookrightarrow{H^1-\text{trivial}} (\T^n\times \mathrm{int}(C\cdot \Omega'),\omega_{\rm can})\to \T^n\times  C\cdot \Omega' 
 \end{equation}
 where $\to$ refers to inclusion. By composing the  maps in (\ref{eq-comp}) above, we derive an exact  $H^1$-trivial embedding from $(\T^n\times  (1-\ep)\Omega,\omega_{\rm can})$ to $ (\T^n\times  C\cdot \Omega',\omega_{\rm can})$. By a symmetric argument and let $\ep\to 0$  we have:
 \begin{equation}\label{eq_HBM_d_I}
     d_{\rm HBM}(X_{\Omega},X_{\Omega'})= d_{\rm HBM}(\T^n \times \Omega,\T^n\times{\Omega'})=d_I(\Omega,\Omega') 
 \end{equation} 
 where the second equality is due to (\ref{eq-HBM-inc-1}) and (\ref{eq-HBM-inc-2}). 

 Then one can consider a similar construction as in the proof of (i) above, with the only difference that the fiber now is in $\R^n_{\geq0}$ rather than $\R^n.$ Explicitly, for any $v=(v_1,\ldots,v_N)\in\R^{N}_{\geq 0}$, consider domain $\Omega_v\subset\R_{\geq 0} ^{n}$ in the form of $Q_v \times [0,1]^{n-2}$, where $Q_v$ is a 2-dimensional polygon with vertices encoding the rescaling of components $v_i$ in the given vector $v$. Then $v \mapsto X_{\Omega_v}$ gives the desired isometric embedding.
\end{proof}
\begin{remark}\label{rmk-disagree-distance}
   Note that  in Definition \ref{dfn-HBM} and Definition \ref{dfn-HBM-toric} of $d_{\rm HBM}$, compared with $d_{\rm BM}$, we put extra constraints on the symplectic embeddings between domains, so $d_{\rm HBM}\geq d_{\rm BM}$. In certain cases the inequality could be strict. Consider the following example shown in Proposition 2.8  of  \cite{DR23}.  Let $X_s= B^4(1)\cup E(\frac{1}{s},s)$  where $E(\frac{1}{s},s)$ is a symplectic ellipsoid in $(\R^4,\omega_{\rm std})$ with $s>1.$
    Then we have 
    \[d_{\rm BM}(X_s,X_t)\leq \ln2,\text{ but } \,d_{\rm HBM}(X_s,X_t)= d_{ I}(X_s,X_t)\geq \left|\ln\left(\frac{s}{t}\right) \right|.  \]
    By fixing $t$ and letting $s\to \infty$,  $d_{\rm BM}(X_s,X_t) \leq \ln 2$ while   $d_{\rm HBM }(X_s,X_t)\to \infty$.  
\end{remark}
%As noted in  Remark \ref{rmk-disagree-distance} above, the distance $d_{\rm HBM}$ does not in general coincide with the coarse Banach–Mazur distance $d_{\rm BM}.$ This naturally raises the question:
%\begin{question} 
%  Under what additional assumptions do these two distances agree? An affirmative answer would allow Theorem~\ref{thm-large-scale-two} to be applied directly to  $d_{\rm BM}$. 
%\end{question}

 \subsection{Proof of Theorem \ref{thm-normalized} and Theorem \ref{thm-ball-normalized}}\label{section-normalized}
 First   recall that for any $v=(v_1,\ldots,v_n)\in \R^n$, the $\ell^p$-norm on $\mathbb{R}^n$ is defined by
\begin{equation}\label{eq-l-p-norm}
    \|v\|_p := 
\begin{cases}
\left( \displaystyle\sum_{i=1}^n |v_i|^p \right)^{1/p}, & \text{if } 1 \le p < \infty, \\[1em]
\max\limits_{1 \le i \le n} |v_i|, & \text{if } p = \infty.
\end{cases}
\end{equation}
With respect to this norm, the $\ell^p$-ball of radius $r>0$ in $\mathbb{R}^n$ is denoted as follows, 
\begin{equation}\label{eq-p-ball}
   B_p(r) \coloneqq \{  v = (v_1,\dots,v_n) \in \mathbb{R}^n \mid \|v\|_p \le r  \}. 
\end{equation}
In this section, we will consider a Finsler metric $F$ on $T \T^n\simeq \T^n\times \R^n$ defined by $\ell^p$-norm  $F(q,v):=\|v \|_p$, which is flat and reversible.
Then we can prove the following proposition:
\begin{prop}\label{prop-ell-norm}
     Let $F$ be a Finsler metric on $\mathbb{T}^2$ induced by an $\ell^p$-norm with $p \in [1,\infty]$. 
Then for any ball-normalized symplectic capacity $c$, we have $c\big(D_F^*\mathbb{T}^2\big) = 2.$ 
\end{prop}
Before proving the proposition,  we introduce a more general family of  toric domains in $\R^{2n}$, imitating the definition at the beginning of  Section 1 of \cite{Mcd25}.
      A {\it generalized toric domain} is a toric domain whose moment image $\Omega\subset \R_{\geq 0}^n$ is not required to contain the origin.  A generalized toric domain $X_\Omega$ is called  {\it generalized convex} if $\Omega\subset \R_{\geq 0}^n$ is geometrically convex and compact.  For open toric domain  $X_{\Omega}$, it is  called generalized convex if the closure $\overline{\Omega}$ is convex and compact.
      
Note that the map $\Phi$   in (\ref{eq-emb-one})  establishes a symplectomorphism   between $(\T^n\times \mathrm{int}(\Omega),\omega_{\rm can})$ and $(X_{\mathrm{int}(\Omega)},\omega_{\rm std})$ for any  generalized toric domain $X_\Omega$.  Combined with the embedding criterion on concave-to-convex toric domains established in  \cite{Mcd25}, we argue as follows:

\begin{proof}[Proof of  Proposition  \ref{prop-ell-norm}]\footnote{In this proof, the proof of condition (ii) in Theorem \ref{thm-normalized}, and  Example \ref{ex-cap-non-coin} below, we always denote $B^4(r)$ as an open ball in $(\R^4,\omega_{\rm std})$ with radius $r>0$. Note that in terms of calculation of capacity, taking the closure of an open ball does not make a difference.}
    For any codsic bundle $(D_{F }^* \mathbb{T}^2,\omega_{\rm can})$  where  $F $ is induced by an  $\ell^p$-norm with $p\in [1,\infty]$, it can be identified with product domain $\mathbb{T}^2\times B_q(1)\subset (T^*\T^2,\omega_{\rm can})$ for $\ell^q$-ball $B_q(1)\subset \R^{2}$ (defined in (\ref{eq-p-ball})) where  $(p,q)$ satisfies $\frac{1}{p}+\frac{1}{q}=1. $
     Observe that  $B_1(1)\subset  B_q (1)\subset B_\infty(1)$ for any $q\in [1,\infty]$. 
     Then by monotonicity of capacity, it suffices prove the following conclusion:
     \[c (\T^2\times B_1(1)) =c (\T^2\times B_\infty(1))=2  \]
     for any ball-normalized capacity  $c$.
   By definition, domain $B_{1}(1)$ is a square in $\R^2$  with vertices
     $(1,0),  (-1,0), (0, 1),  (0,- 1)$. 
     Then  we shift $B_{1}(1)$ by a constant vector $( 1,1)$  so that $B_{1}(1)+(1,1)$ stays in $\R_{\geq 0}^2$,   denoted by $Q.$  The shifting induces a  symplectomorphism between $(\T^2\times B_1(1),\omega_{\rm can})$ and $(\T^2\times Q,\omega_{\rm can})$ so that $c(\T^2\times B_1(1))=c(\T^2\times Q).$
  By considering  the map $\Phi$ in (\ref{eq-emb-one}),  there exists a symplectomorphism  $(\T^2\times {\mathrm{int}(Q)},\omega_{\rm can}) \simeq  (X_{\mathrm{int}(Q)},\omega_{\rm std})$
 and $X_{\mathrm{int}(Q)}\subset \R^4$ is a generalized convex toric domain since $Q\subset \R^2_{\geq 0}$ is convex.
  
 Then we consider the largest symplectic ball embedded into the generalized convex toric domain $X_{ \mathrm{int}(Q)}$. By Section 2.1 in \cite{Mcd25}, the weight decomposition of $Q\subset \R_{\geq 0}^2$ is  given by $(3;1,1,1,1,1)$.
 Then by Theorem  1.2.1 in \cite{Mcd25}, a ball $(B^4(a),\omega_{\rm std})$ symplectically embeds into the generalized  convex toric domain  $X_{\mathrm{int}(Q)}$ if and only if there exists an symplectic embedding:
 \[  \left(\bigsqcup_{i=1}^5 B_i^4\left(\sqrt{\frac{1 }{\pi}}\right)\bigsqcup B^4(a),\omega_{\rm std}\right)  \hookrightarrow \left(B^4\left(\sqrt{\frac{3  }{\pi}}\right),\omega_{\rm std}\right).\]
 By Proposition 9.5 (i) in \cite{Sch17} on the very full packing of $6$-balls, such an embedding exists if and only if  $a= \sqrt{\frac{2  }{\pi}}$. 
In particular, there exists a symplectic embedding
 \[\left(B^4\left( \sqrt{\frac{2  }{\pi}}\right),\omega_{\rm std}\right)\hookrightarrow (X_{\mathrm{int}(Q)},\omega_{\rm std})\simeq  (\T^2\times {\mathrm{int}(Q)},\omega_{\rm can})\subset  (\T^2\times Q,\omega_{\rm can}).  \]
 Therefore, by definition and the embedding above, any ball-normalized capacity $c$ satisfies
  \begin{equation}\label{eq-emb}
      c(\T^2\times B_{1}(1))=c(\T^2\times Q)\geq c\left(B^4\left( \sqrt{\frac{2   }{\pi}}\right)\right)=2  .
 \end{equation}
 
 For the upper bound,  note that  
  $B_\infty(1)=   [-1,1]\times [-1,1] $. By considering the symplectic embedding $\Phi$ in (\ref{eq-emb-one}) and the shifting  vector $v=(1,1)\in \R^2$, we have
 \begin{equation}\label{eq-inclusion-cylinder}
     \Phi(\T^2\times (\mathrm{int}(B_\infty(1)) +v) )\subset Z^4\left(  \sqrt{\frac{2  }{\pi}} \right).
 \end{equation} 
  For any $\ep>0$, we have $\T^2\times  B_\infty(1) \subset \T^2\times \mathrm{int}( B_\infty(1+\ep))$. Then by (\ref{eq-inclusion-cylinder}) we have
 \[c (\T^2\times  B_\infty(1))\leq c (\T^2\times  \mathrm{int}(B_\infty(1+\ep)))\leq  c\left(Z^4\left(  \sqrt{\frac{2 +2\ep }{\pi}} \right)\right)=2+2\ep.\]
 Together with (\ref{eq-emb}), we conclude that
 \[2  \leq c(\T^2\times B_{1}(1))\leq c (\T^2\times  B_\infty(1))\leq 2+2\ep\]
 for any $\ep>0.$ By letting $\ep\to 0,$ we  complete the proof.
\end{proof}
Based on the proposition above, we prove Theorem \ref{thm-normalized} as follows.
 \begin{proof}[Proof of Theorem \ref{thm-normalized}, under condition (i)] Let us identify $D^*_F\T^2$ with $\T^2 \times A$ for some domain $A \subset \R^2$. Let  $ (a,0)\in \partial A $  with $a>0$ denote the intersection point of
$\partial A$ with the positive $x$-axis. By the condition (i), the domain
$A\subset \R^2$ is symmetric with respect to the $x$-axis, so is the boundary
$\partial A$. Therefore, the tangent line of $\partial A$ at $(a,0)$ is vertical, i.e., $x=a$.

By the same symmetry argument, the tangent line to $\partial A$ at
$(-a,0)$ is $x=-a$. Moreover, since $F$ is reversible, condition (i) implies that $F(q_1,q_2,-v_1,v_2) =F(q_1,q_2,v_1,-v_2) = F(q_1, q_2, v_1, v_2)$ for any $(v_1, v_2) \in T_{(q_1, q_2)}\T^2$. Then the same argument as above yields that  the tangent lines at the intersection points
$(0,b)$ and $(0,-b)$ with the $y$-axis are horizontal, that is, 
$y=b$ and $y=-b$, respectively.
Since $\partial A$ is tangent to the vertical lines $x=\pm a$ and the
horizontal lines $y=\pm b$ at these extremal points, and $A$ is convex, it follows that
\[
A \subset [-a,a]\times[-b,b].
\]
Without loss of generality, we assume $a\leq b$ so that $(0,\pm a)\in A$ and then $\ell^1$-ball $B_1(a)\subset A$. 

Let $c$ be any ball-normalized capacity. By the monotonicity of symplectic capacity, we have
\begin{equation}\label{eq-cap-thmE-1}
    2a=c(\T^2 \times B_1(a))\leq c(\T^2\times A)\leq c(\T^2\times [-a,a]\times [-b,b])
  \end{equation}
    where the first equality comes from Proposition \ref{prop-ell-norm}.
Meanwhile, by  considering the symplectic embedding $\Phi$ defined  in (\ref{eq-emb-one}), we have 
\begin{equation}\label{eq-map-polydisk}
    \Phi(\T^2\times ((0,2a)\times (0,2b)))\subset Z^4\left(\sqrt{\frac{2a}{\pi}}\right).
\end{equation} 
Therefore, by (\ref{eq-map-polydisk}) and translation on $(-a,a)\times (-b,b)\subset \R^2$, we have
\begin{equation}\label{eq-cap-thmE-2}
    c(\T^2\times  ((-a,a)\times (-b,b)))=c(\T^2\times ((0,2a)\times (0,2b)))\leq c\left(Z^4\left(\sqrt{\frac{2a}{\pi}}\right)\right) =2a.
\end{equation} 
By considering the closure of $(-a,a)\times (-b,b)\subset \R^2$  we have 
\[c(\T^2\times [-a,a]\times [-b,b])=c(\T^2\times  (-a,a)\times (-b,b))\leq 2a.\]
Combining the inequalities (\ref{eq-cap-thmE-1}) and (\ref{eq-cap-thmE-2}), we conclude that $c(\T^2\times A)=2a.$
 
Finally, since $\T^2\times B_1(a)\subset \T^2\times A\subset \T^2\times([-a,a]\times [-b,b])$ and  $A$ is convex, by  Proposition \ref{prop-convex-hull}  we have
\begin{equation}\label{eq-sys-ineq}
   \sys( \T^2\times \partial B_1(a))\leq\sys(  \T^2\times \partial A)\leq \sys(\T^2\times\partial([-a,a]\times [-b,b])). 
\end{equation} 
By  \cite{Sab10} and the paragraph under Theorem 4.6 in \cite{ABT16}, $\rhoT(\T^2\times  B_1(a))=\frac{1}{4}$, therefore $\sys( \T^2\times \partial B_1(a))=a$. On the other hand, for the point $(a,0)\in  [-a,a]\times [-b,b] \subset \R^2$, the outer unit normal vector at $(a,0)$ is $(1,0)$. By (\ref{eq-action}), there exists a closed Reeb orbit in $\T^2\times \{(a,0)\}$ with action $a$. Therefore,
\[ \sys (\T^2\times \partial([-a,a]\times [-b,b]))\leq a.\]
 Combined with (\ref{eq-sys-ineq}) we deduce that $\sys(  \T^2\times \partial A)=a$. Together with the result that $c(\T^2\times A)=2a$ right above, we conclude that $c(\T^2\times A)=2\sys(\T^2\times \partial A).$
\end{proof}
 Given a convex domain $A\subset \R^2$, one can argue similarly to the proof of Proposition~\ref{prop-ell-norm}. By identifying 
 $\T^2\times A$ (possibly with a shifting on $A$ to guarantee $A\subset \R_{>0}^2$) with a generalized convex toric domain via the symplectic embedding in (\ref{eq-emb-one}). Under this identification, the problem of symplectically embedding a ball into  $(\T^2\times A,\omega_{\rm can})$ is reduced to a ball packing problem for the corresponding toric domain. By Theorem~1.2.1 of \cite{Mcd25}, this reduction allows one to characterize the existence of such embeddings in terms of a comparison of embedded contact homology (ECH) capacities, first invented in \cite{Hut11}. In particular, as shown by Theorem 1.11 in \cite{Hut11}, the  ECH capacities for codisc bundles of torus $D_F^*\T^2$ under a flat reversible Finsler metric $F$ can be computed explicitly, which leads to the proof of Theorem \ref{thm-normalized} under condition (ii).  To facilitate our discussion, let us conclude the following result first.
\begin{prop}\label{prop-first-ech}
    For any codisc bundle of  torus $D_F^*\T^2\in \mathcal{F}_{\rm rev}$ induced by a flat reversible Finsler metric $F$ on $\T^2,$ we have $c_1^{\rm ECH}(D_F^*\T^2)=2\sys(\partial D_F^*\T^2)$ where $c_1^{\rm ECH}$  denotes the first ECH capacity.
\end{prop}
\begin{proof}
Let $F$ be a flat (i.e.\ translation-invariant) reversible Finsler metric on $ \T^2$.
Such a metric determines a norm $\|\cdot\|$ on $\R^2$ by
\[
\|v\| := F(q,v), \qquad \text{for any } v\in T_q\T^2\simeq \R^2,
\]
which is independent of the base point  $q\in \T^2$. For simplicity, we write
$F(v)$ instead of  $F(q,v)$.
By Theorem~1.11 in \cite{Hut11}, the $k$-th ECH capacity of  $ (D_F^*\T^2,\omega_{\rm can})$  is given by
\begin{equation}\label{eq-ECH-cal-2}
c_k^{\rm ECH}(D_F^*\T^2)
=
\min \left\{
  \ell_{\|\cdot\|}(\Lambda)
  \;\middle|\;
  \lvert P_\Lambda \cap \Z^2 \rvert = k+1
\right\},
\end{equation}
where the minimum is taken over all convex polygons  $\Lambda\subset \R^2$ with
vertices in $\Z^2$, $P_{\Lambda}$ denotes the closed region enclosed by $\Lambda$ (hence, $|P_\Lambda\cap \Z^2|$ denotes the number of the lattice points  in $P_{\Lambda}$), 
and $\ell_{\|\cdot\|}(\Lambda)$ denotes the perimeter of $\partial \Lambda$ measured with
respect to the norm  $\|\cdot\|$.

We now specialize to the case $k=1$. The condition
$|P_\Lambda\cap \Z^2|=2$ implies that $\Lambda$ has at most two vertices, and
hence must be a line segment connecting two lattice points. By translation
invariance, we may assume that $\Lambda$ is the segment joining the origin
$(0,0)$ and a nonzero lattice vector $v=(v_1,v_2)\in \Z^2$. In this case,
\begin{equation}\label{eq-identification-norm}
    \ell_{\|\cdot\|}(\Lambda)=2\|v\|=2F(v).
\end{equation}
Moreover, the condition $|P_\Lambda\cap \Z^2|=2$ is equivalent to saying that the segment
contains no lattice points other than its endpoints, which holds precisely when
$v$ is primitive, i.e.   $\gcd(v_1,v_2)=1$, or when
 $v\in\{(\pm1,0),(0,\pm1)\}$. Since every nonzero lattice vector is an integer
multiple of a primitive one, minimizing over all primitive vectors is equivalent
to minimizing over all nonzero vectors in  $\Z^2$. Consequently,
\eqref{eq-ECH-cal-2} reduces to
\begin{equation}\label{eq-ECH-sys}
c_1^{\rm ECH}(D_F^*\T^2)
=
\min \left\{
  2F(v)
  \;\middle|\;
  v\in \Z^2,\ v\neq 0
\right\}.
\end{equation}

Finally, for a flat Finsler metric $F$ on $\T^2$, the minimal action  is given by 
\[\sys(\partial D_F^*\T^2)=\min \{F(v)\mid v\in \Z^2,\ v\neq 0\},\]
as explained at the beginning of Section~2 in \cite{Ben24}. Combining this with
\eqref{eq-ECH-sys}, we conclude that $c_1^{\rm ECH}(D_F^*\T^2)=2 \sys(\partial D_F^*\T^2).$ 
\end{proof}

Then we prove the second part of Theorem \ref{thm-normalized} as follows.

\begin{proof}[Proof of Theorem \ref{thm-normalized}, under condition (ii)]
For any codisc bundle of torus $D_F^*\T^2$ with respect to a flat reversible metric $F$ with $\rhoT(D_F^*\T^2)\leq \frac{1}{8}$, we identify it with product domain $\T^2\times A$. By rescaling on $A$ we may assume $\sys(\partial D_F^*\T^2)=1$, then it suffices to prove $c(D_F^*\T^2)=2$ for any ball-normalized capacity $c.$ The upper bound comes from Theorem 1.1 in \cite{Ben24}. The lower bound can be argued as follows.

Similar to  the proof of Proposition \ref{prop-ell-norm}, we translate  the convex domain $A\subset \R^2$ to $\R^2_{\geq 0}$ by a vector $v$. Then by  (\ref{eq-emb-one}),  $(\T^2\times  \mathrm{int}( A+v),\omega_{\rm can})$ is symplectomorphic to a generalized convex toric domain $(X_{\mathrm{int}( A+v)},\omega_{\rm std})$. Then by  Theorem 1.2.1 in \cite{Mcd25}, the symplectic ball $(B^4 (\sqrt{2/\pi}) ,\omega_{\rm std})$ (as a concave toric domain)  symplectically embeds  into $(X_{\mathrm{int}( A+v)},\omega_{\rm std})$ if and only if the following relation holds:
\begin{equation}\label{eq-ech-relatioin-1}
    c_k^{\rm ECH}\left(B^4\left(\sqrt{\frac{2}{\pi}}\right)\right)\leq c_k^{\rm ECH}(X_{\mathrm{int}( A+v)}),\quad \text{ for any $k\geq 1$}
\end{equation} 
   where  $c_k^{\rm ECH}$ denotes the $k$-th ECH capacity. By equation (1.3) in \cite{KDDMV14}, the  ECH capacities of  ball $(B^4 (\sqrt{2/\pi}) ,\omega_{\rm std})$ can be calculated by 
 \begin{equation}\label{eq-ECH-ball-1}
     c_k^{\rm ECH}\left(B^4\left(\sqrt{\frac{2}{\pi}}\right)\right)=2d,
 \end{equation}
 where $d$ is the unique non-negative integer that satisfies $ \frac{d^2+d}{2}\leq k\leq \frac{d^2+3d}{2}$.   In particular, we have $\frac{d^2+d}{2}\leq k$, therefore $d\leq \sqrt{2k+\frac{1}{4}}-\frac{1}{2}$. Thus
 \begin{equation}\label{eq-ECH-ball-2}
     c_k^{\rm ECH}\left(B^4\left(\sqrt{\frac{2}{\pi}}\right)\right)=2d\leq  \sqrt{8k+1}-1.
 \end{equation} 
 Meanwhile, by the symplectomorphism invariant property of any capacity, for any $k\geq 1$,
 \begin{equation}\label{eq-ech-equal}
     c_k^{\rm ECH}(X_{\mathrm{int}( A+v)})=c_k^{\rm ECH}(\T^2\times \mathrm{int}(A+v))=c_k^{\rm ECH}(\T^2\times \mathrm{int}(A)).
 \end{equation} 
Then the inclusion $\mathrm{int}(A)\subset A\subset 
 (1+\ep)\mathrm{int}(A)$ implies that for any $\ep>0$ we have
 \begin{equation}\label{eq-ech-equal-2}
     c_k^{\rm ECH}(\T^2\times \mathrm{int}(A))=c_k^{\rm ECH}(\T^2\times A).
 \end{equation} 
 Combine  (\ref{eq-ECH-ball-1}), (\ref{eq-ECH-ball-2}), (\ref{eq-ech-equal}) and (\ref{eq-ech-equal-2}), relation in (\ref{eq-ech-relatioin-1}) holds if the following relation holds:  
\begin{equation}\label{eq-lowerbound-product}
        c_k^{\rm ECH}(\T^2\times A) \geq   \sqrt{8k+1}-1,\quad \text{ for any $k\geq 1$.}
\end{equation}

Recall the  classical Wulff inequality  as in equation (20.14) in \cite{Mag12} with a clean statement in Section 1 of \cite{FZ22}. In particular, we take convex, positive, 1-homogeneous function $f $ appearing at the beginning of Section 1 of \cite{FZ22}  as the norm $\|\cdot\|.$  Then for any domain $\Lambda\in \R^2$ with finite perimeter $\ell_{\|\cdot\|}(\Lambda)$, by the Wulff inequality\footnote{This inequality was also used in the proof of Lemma 3.11 of \cite{Hut22}, where only the equality case was needed to estimate the upper bound of the ECH capacities of a convex toric domain. In contrast, here we use the  inequality to estimate the lower bound of the ECH capacities of $\T^2\times A$.} we have
\begin{equation}\label{eq-perimeter}
    \ell_{\|\cdot\|}(\Lambda)\geq 2\cdot \mathrm{area}(B^{*} )^{\frac{1}{2}}\cdot \mathrm{area}(\Lambda)^{\frac{1}{2}} 
\end{equation} 
where $B^{*}\subset \R^2$ is the dual of $\{x\in \R^2\mid \|x\|\leq 1\}.$ In our case, since $F(\cdot)=\|\cdot\|$, by definition in (\ref{dfn-F*}) we have $B^{*}=A$.  Then, under condition (ii), we have 
\[ \begin{aligned}
\frac{1}{8}&\;\ge\; \rhoT(\T^2\times A)\\
&=\frac{\sys(\T^2\times \partial A)^2}
{\mathrm{vol}\!\left(\T^2\times \partial A,
\lambda_{\rm can}\big|_{\T^2\times \partial A}\right)} \\[0.3em]
&=\frac{1}{\mathrm{vol}\!\left(\T^2\times \partial A,
\lambda_{\rm can}\big|_{\T^2\times \partial A}\right)}
=\frac{1}{2\,\mathrm{area}(A)} 
\end{aligned}
\]
where we use our assumption that $\sys(\T^2\times \partial A)=1$ in the second inequality. The above inequality implies $\mathrm{area}(A)\geq 4.$ 

Moreover, for any convex polygon $\Lambda$ with vertices in $\Z^2$, denote by $I(\Lambda)$ the number of lattice points in the interior of $\Lambda$ and $B(\Lambda)$ as the   number of lattice points on the boundary of $\Lambda.$ By the well-known Pick's theorem, we have
\[\mathrm{area}(\Lambda)=I(\Lambda)+\frac{B(\Lambda)}{2}-1.\]
By definition, $|P_{\Lambda}\cap \Z^2|=I(\Lambda)+B(\Lambda)$ where recall $|\Lambda\cap \Z^2|$ (appearing in (\ref{eq-ECH-cal-2})) denotes the number of lattice points in the closed region $P_{\Lambda}$ bounded by $\Lambda$. Therefore, for any convex polygon $\Lambda$ with vertices in $\Z^2$ such that   $|P_\Lambda\cap \Z^2|=k+1$, we have
\[\mathrm{area}(\Lambda)=I(\Lambda)+\frac{B(\Lambda)}{2}-1\geq  \frac{I(\Lambda)+B(\Lambda)}{2}-1=\frac{|P_\Lambda\cap \Z^2|}{2}-1=\frac{k-1}{2}. \]
Together with (\ref{eq-perimeter}), for any  convex polygon $\Lambda$ under the assumption that $|P_\Lambda\cap \Z^2|=k+1$, we have
 \begin{equation}\label{eq-perimeter-low-bound}
     \begin{split}
         \ell_{\|\cdot\|}(\Lambda)&\geq 2\cdot\mathrm{area}(B^{*} )^{\frac{1}{2}}\cdot \mathrm{area}(\Lambda)^{\frac{1}{2}}\\&=2\cdot\mathrm{area}(A)^{\frac{1}{2}}\cdot\mathrm{area}(\Lambda)^{\frac{1}{2}}\\&\geq  4\cdot \mathrm{area}(\Lambda)^{\frac{1}{2}}\geq 4\sqrt{\frac{k-1}{2}}\\&= \sqrt{8 k- 8}
     \end{split}
 \end{equation}
 where we use the fact that $B^*=A$ in the first equality and $\mathrm{area}(A)\geq 4$ for the second inequality. 
 
 Therefore by (\ref{eq-perimeter-low-bound}) we have
 \[\ell_{\|\cdot\|}(\Lambda)\geq  \sqrt{8k-8}\geq \sqrt{8k+1}-1 \quad \text{ for any $k\geq 3.$}\]
Then by (\ref{eq-ECH-cal-2}) and the inequality above, we have
\begin{equation}\label{eq-lowerbound-final}
    \begin{split}
        c_k^{\rm ECH}(\T^2\times A) &=\min \left\{
  \ell_{\|\cdot\|}(\Lambda)
  \;\middle|\;
  \lvert P_\Lambda \cap \Z^2 \rvert = k+1
\right\} \\&\geq\sqrt{8k+1}-1\geq   c_k^{\rm ECH}\left(B^4\left(\sqrt{\frac{2}{\pi}}\right)\right)
    \end{split}
\end{equation} 
for any $k\geq 3$. For $k= 1$, the inequality (\ref{eq-lowerbound-final}) is verified by Proposition \ref{prop-first-ech}. For $k=2$, 
\[c_2^{\rm ECH}(\T^2\times A)\geq c_1^{\rm ECH}(\T^2\times A)=2=  c_2^{\rm ECH}\left(B^4\left(\sqrt{\frac{2}{\pi}}\right)\right)\]
where the first inequality comes from the fact that the ECH capacities of a symplectic manifold is non-decreasing in $k$ (cf.~Section 1.1 in \cite{Hut11}) and the last equality comes from computation in (\ref{eq-ECH-ball-1}). Consequently, (\ref{eq-ech-relatioin-1}) yields a symplectic embedding from $(B^4 (\sqrt{2/\pi}) ,\omega_{\rm std})$  into $(X_{\mathrm{int}( A+v)},\omega_{\rm std})$. Thus for any ball-normalized capacity $c$, we have
 \[c(X_{\mathrm{int}( A+v)})\geq c\left(B^4\left(\sqrt{\frac{2}{\pi}}\right)\right)=2.\]
Finally, similarly to (\ref{eq-ech-equal}) and (\ref{eq-ech-equal-2}), we conclude that
 \[c(\T^2\times A)= c(\T^2\times \mathrm{int}(A))=c(\T^2\times \mathrm{int}(A+v))=c(X_{\mathrm{int}( A+v)})\geq 2,\]
 which gives the lower bound of $c(\T^2\times A)$ as desired.
\end{proof}

Next, we give the proof of Corollary \ref{cor-cap-cylinder}. 
\begin{proof}[Proof of Corollary \ref{cor-cap-cylinder}]
By rescaling we may assume $r=1.$  
    For any unit vector $v$, if it is a scalar multiple of a prime integer vector $\alpha=(m,n)$, we can write 
    \[v=\left(\frac{m}{\sqrt{m^2+n^2}},\frac{n}{\sqrt{m^2+n^2}}\right).\]
   By symmetry to $x$ and $y$ axes, we may assume $m\leq  0,n>0.$ If $m=0$ so that $n=1,$  
   \[Y^4(1,(0,1))=\T^2\times (\R\times   [-1,1]).\]
   Then by translating $ \R\times   [-1,1]$ to  $ \R\times [\ep,2+\ep]$ and applying embedding $\Phi$ in (\ref{eq-emb-one}), we have $\Phi(\T^2\times (\R\times [\ep,2+\ep]))\subset Z^4\left(\sqrt{\frac{2+\ep}{\pi}}\right)$ and 
   \begin{equation}\label{eq-width-up-bound}
       c (\T^2\times (\R\times   [-1,1]))=c (\T^2\times(\R\times   [\ep,2+\ep]))\leq c \left(Z^4\left(\sqrt{\frac{2+\ep}{\pi}}\right)\right)=2+\ep
   \end{equation} 
   for any $\ep>0$.  On the other hand,  by Proposition  \ref{prop-ell-norm}, $c (\T^2\times B_1(1))=2$ and $\T^2\times B_1(1)\subset \T^2\times (\R\times [-1,1])$. Therefore
   \[c (\T^2\times (\R\times   [-1,1]))\geq c (\T^2\times B_1(1))=2.\]
   Combined with (\ref{eq-width-up-bound}), by letting $\ep\to0$ we have
   \[c (Y^4(1,(0,1))=c (\T^2\times(\R\times   [-1,1]))=2.\]
   
Next, consider the case where $m<0 $ and $n>0$.  By definition, the boundary of $Y^4(1,v)$ consists of two  components:
    \[\T^2\times \left\{(x,y)\mid mx+ny=\sqrt{m^2+n^2}\right\},\quad \T^2\times \left\{(x,y)\mid mx+ny=-\sqrt{m^2+n^2}\right\}.\]
In particular, for any $a>0$, the following  convex polygon is contained in $(-1,1)v\times v^{\perp}\subset \R^2$ with vertices:
    \[\pm \left(a,\frac{\sqrt{m^2+n^2}}{n}-\frac{m}{n}a\right),\quad \pm \left(-a,\frac{\sqrt{m^2+n^2}}{n}+\frac{m}{n}a\right) ,\]
    and denote this polygon as $P(a,v)$.
    
We will prove that for large enough $a>0,$  $(\T^2\times \partial P(a,v),\lambda_{\rm can}|_{\T^2\times \partial P(a,v)})$ admits minimal action $\sys(\T^2\times \partial  P(a,v))=\sqrt{m^2+n^2}$. 
    First, by (\ref{eq-action}), the (rescaled) outer normal vector at 
    $\left(\frac{m}{\sqrt{m^2+n^2}},\frac{n}{\sqrt{m^2+n^2}}\right)$ is $\alpha=(m,n)$ and there exists a closed Reeb orbit  in $\T^2\times \left\{\left(\frac{m}{\sqrt{m^2+n^2}},\frac{n}{\sqrt{m^2+n^2}}\right)\right\}$ with action 
    \begin{equation}\label{eq-action-m-n}
        m\cdot \frac{m}{\sqrt{m^2+n^2}}+n\cdot \frac{n}{\sqrt{m^2+n^2}}=\sqrt{m^2+n^2}.
    \end{equation} 

Around each vertex, we apply similar smoothing process as in Example \ref{ex-perturb-cal}.  Suppose there exists a closed Reeb orbit in  $\T^2\times \{p\}$ for some $p$ that is in $\ep$-small neighborhood of the vertex $\left(a,\frac{\sqrt{m^2+n^2}}{n}-\frac{m}{n}a\right)$ with (rescaled) outer normal vector $(c,d)$. By (\ref{eq-action}), we can write the action of this orbit as 
\[T=p\cdot (c,d)= a  c+\left(\frac{\sqrt{m^2+n^2}}{n}-\frac{m}{n}a\right)  d+\mathcal{O}(\ep)=a\left(c-\frac{dm}{n}\right)+d\frac{\sqrt{m^2+n^2}}{n}+\mathcal{O}(\ep)\]
where the term $\mathcal{O}(\ep)$ comes from the perturbation of $\T^2\times \partial P(a,v)$ near the vertex and it converges to $0$ as $\ep\to0.$ 

Observe  that around the vertex $\left(a,\frac{\sqrt{m^2+n^2}}{n}-\frac{m}{n}a\right)$, the normal vector satisfies $d\geq 0$ and $cn\geq dm$ 
with inequality holds if  and only if  $c=m$ and $d=n.$ When  $c=m$ and $d=n$, $T=\sqrt{m^2+n^2}+\mathcal{O}(\ep)$ converges to $\sqrt{m^2+n^2}$ as $\ep\to0.$   Otherwise, $cn-dm\geq  1.$ By the fact that $n>0,$ we have
    \[T=a\left(c-\frac{dm}{n}\right)+d\frac{\sqrt{m^2+n^2}}{n}+\mathcal{O}(\ep)\geq \frac{a}{n}+d\frac{\sqrt{m^2+n^2}}{n}+\mathcal{O}(\ep)\geq \frac{a}{n}+ \mathcal{O}(\ep) .\]
For a fixed pair $(m,n)$, by letting $a$ big enough and $\ep$ small enough we conclude that
\[T\geq \sqrt{m^2+n^2}\]
for any closed Reeb orbit that is in   $\T^2\times \{p\}$. Similarly we obtain the same conclusion around other three vertices. Together with calculation in (\ref{eq-action-m-n}), we have
\[\mathrm{sys}(\T^2\times \partial P(a,v))=\sqrt{m^2+n^2}\]
for large enough $a>0.$ Since the volume of $\T^2\times \partial P(a,v)$ goes to infinity as $a\to +\infty$,  we have
\[\rhoT( \T^2\times \partial P(a,v))<\frac{1}{8}\]
for large-enough $a>0.$ Then by condition (ii) of Theorem \ref{thm-normalized}, we have
\begin{equation}\label{eq-cap-cylinder}
    c(\T^2\times P(a,v))=2\sys(\T^2\times \partial P(a,v))=2\sqrt{m^2+n^2}.
\end{equation} 
Since $\T^2\times P(a,v)\subset Y^4(1,v)$, we have the following lower bound:
\[c( Y^4(1,v))\geq c(\T^2\times P(a,v)) =2\sqrt{m^2+n^2}.\]
The upper bound $c( Y^4(1,v))\leq 2\sqrt{m^2+n^2}$ follows  from Theorem 1.1 in \cite{Ben24}.

Finally, In the case where $v$ is {\it not} a scalar multiple of an integer vector, Theorem A in \cite{FZ25} implies that for every integer $a>0$ there exists a symplectic embedding
\[
(\T^2\times \{(x,y)\in \R^2_{\ge 0}\mid x+y\le a\},\omega_{\rm can})
\hookrightarrow (Y^4(1,v),\omega_{\rm can}).
\]
By monotonicity of the symplectic capacity $c$, this yields
\[
c \left(Y^4(1,v)\right)
\ge
c \left(\T^2\times \{(x,y)\in \R^2_{\ge 0}\mid x+y\le a\}\right).
\]
By Corollary~\ref{cor-capacity-BM},
\[
c \left(\T^2\times \{(x,y)\in \R^2_{\ge 0}\mid x+y\le a\}\right)
=
c \left(B^4\!\left(\sqrt{\tfrac{a}{\pi}}\right)\right)
=
a.
\]
Therefore $c(Y^4(1,v))\ge a$ for every $a>0$, which implies $c(Y^4(1,v))=+\infty$. \end{proof}

Here, we emphasize that the proof of Theorem \ref{thm-normalized} (under condition (ii)) is based on Proposition \ref{prop-first-ech}, in particular, the formula (\ref{eq-ECH-cal-2}). This formula applies to the codisc bundle of a  flat reversible Finsler metric on $\T^2$, according to Theorem 1.11 in \cite{Hut11}. The following example shows that there exists a flat but non-reversible metric $F$ on $\T^2$ where the conclusion in Theorem \ref{thm-normalized} may fail. 
 
\begin{ex}\label{ex-cap-non-coin}
For the product domain $(\T^2\times A,\omega_{\rm can})$ appearing in Remark \ref{rmk-1-3} where   $A$ is a triangle with vertices $(0,1),(1,0),(-1,-1)$, ball-normalized capacities do {\it not} coincide.  In particular,  $c_{\rm Gr}(\T^2\times A)=\frac{3}{2}$ and $c_1^{\rm ECH}(\T^2\times A)=2$.
 
By translation on $A\subset \R^2$, we obtain a symplectomorphism 
\[(\T^2\times A,\omega_{\rm can})\simeq (\T^2\times(A+(1,1)),\omega_{\rm can})\]
where $A+(1,1)$ is a triangle in $\R^2_{\geq 0}$ with vertices $(0,0),(2,1),(1,2),$ denoted by $A'.$ Then the symplectomorphism in (\ref{eq-symplectomorphism}) applies and
\[(\T^2\times \mathrm{int}(A'),\omega_{\rm can})\simeq  (X_{\mathrm{int}(A')},\omega_{\rm std})\]
where $X_{\mathrm{int}(A')}$ is a generalized convex toric domain.

Following the  cutting algorithm in  Section 2.1 of \cite{Mcd25}, $X_{\mathrm{int}(A')}$ admits  weight decomposition $(3;1,1,1,1,1,1)$. By  Theorem 1.2.1 of \cite{Mcd25}, the existence of a symplectic embedding of the ball $(B^4(a),\omega_{\rm std}) \hookrightarrow(X_{\mathrm{int}(A)},\omega_{\rm std})$ is equivalent to the existence of the following symplectic embedding: 
   \[ \left(\bigsqcup_{i=1}^6 B_i^4\left(\sqrt{\frac{1 }{\pi}}\right)\bigsqcup B^4(a),\omega_{\rm std}\right)\hookrightarrow \left( B^4\left(\sqrt{\frac{3  }{\pi}}\right),\omega_{\rm std}\right).\]
  First we prove that $a$ can be taken as $ \sqrt{\frac{3  }{2\pi}}$, i.e., there exists a symplectic embedding:
  \[ \left(\bigsqcup_{i=1}^6 B_i^4\left(\sqrt{\frac{1 }{\pi}}\right)\bigsqcup B^4 \left(\sqrt{\frac{3  }{2\pi}}\right),\omega_{\rm std}\right) \hookrightarrow \left(B^4\left(\sqrt{\frac{3  }{\pi}}\right),\omega_{\rm std}\right).\]
  By  Theorem 1.2 in \cite{Dan19}, this is equivalent to prove the following relation:  
  \begin{equation}\label{eq-ECH-relation}
      c_k^{\rm ECH}\left(\bigsqcup_{i=1}^6 B_i^4\left(\sqrt{\frac{1 }{\pi}}\right)\bigsqcup B^4 \left(\sqrt{\frac{3  }{2\pi}}\right) \right)\leq c_k^{\rm ECH}\left( B^4\left(\sqrt{\frac{3  }{\pi}}\right)\right), \quad \text{for any  $k\in \N$ }
  \end{equation}
  where $c_k^{\rm ECH}$ is the $k$-th ECH capacity.
 On the one hand, by the combinatorial formula of the ECH capacities (see (1.5) in \cite{KDDMV14}), given a disjoint union of balls $ \coprod_{i=1}^{n} B^4\left(\sqrt{a_i/\pi}\right) $, we have 
 \begin{equation}\label{eq-ECH-union-ball}
     c_k^{\rm ECH}\left(\coprod_{i=1}^{n} B^4\left(\sqrt{\frac{a_{i}}{\pi}}\right)\right)=\max \left\{\sum_{i=1}^{n} a_{i} d_{i} \,\left\lvert\, d_i\in \N,\, \sum_{i=1}^{n} \frac{d_{i}^{2}+d_{i}}{2} \leq k\right.\right\}.
 \end{equation}
By the Cauchy–Schwarz inequality
 we have
 \[\sum_{i=1}^{n} a_{i} d_{i} \leq \sqrt{\left(\sum_{i=1}^n a_i^2\right)\cdot \left(\sum_{i=1}^n d_i^2\right) }\leq \sqrt{2k\left(\sum_{i=1}^n a_i^2\right)} \] 
 for any sequence $\{d_i\}_{i=1}^n$ that satisfies the relation in (\ref{eq-ECH-union-ball}). Consequently by (\ref{eq-ECH-union-ball})  we have
 \[  c_k^{\rm ECH}\left(\coprod_{i=1}^{n} B^4\left(\sqrt{\frac{a_{i}}{\pi}}\right)\right)\leq  \sqrt{2k\left(\sum_{i=1}^n a_i^2\right)}.\]
 Applying the inequality above to our case where $n=7$ and $(a_i)_{i=1}^7=(1,1,1,1,1,1,3/2)$, we derive:
 \begin{equation}\label{eq-ECH-up-bound}
     c_k^{\rm ECH}\left(\bigsqcup_{i=1}^6 B_i^4\left(\sqrt{\frac{1 }{\pi}}\right)\bigsqcup B^4 \left(\sqrt{\frac{3  }{2\pi}}\right) \right)\leq  \sqrt{\frac{33}{2}k}.
 \end{equation} 
 On the other hand, by equation (1.3) in \cite{KDDMV14}, the  capacities of a ball can be calculated by 
 \begin{equation}\label{eq-ECH-ball}
    c_k^{\rm ECH}\left(B^4\left(\sqrt{\frac{a}{\pi}}\right)\right)=ad,
 \end{equation}
 where $d$ is the unique non-negative integer that satisfies $ \frac{d^2+d}{2}\leq k\leq \frac{d^2+3d}{2}$. Then  we have a lower bound of $d$:
 \[ d\geq  \sqrt{2k+\frac{9}{4}}-\frac{3}{2}.\]
 Therefore by (\ref{eq-ECH-ball}) and the lower bound of $d:$
 \[c_k^{\rm ECH}\left(B^4\left(\sqrt{\frac{a}{\pi}}\right)\right)=ad\geq a\left(\sqrt{2k+\frac{9}{4}}-\frac{3}{2}\right).\]
In our case, we have
 \begin{equation}\label{eq-ECH-low-bound}
     c_k^{\rm ECH}\left( B^4\left(\sqrt{\frac{3  }{\pi}}\right)\right) \geq 3\left(\sqrt{2k+\frac{9}{4}}-\frac{3}{2}\right).
 \end{equation} 
 Combining (\ref{eq-ECH-up-bound}) and (\ref{eq-ECH-low-bound}),  we conclude that
 \[c_k^{\rm ECH}\left(\bigsqcup_{i=1}^6 B_i^4\left(\sqrt{\frac{1 }{\pi}}\right)\bigsqcup B^4 \left(\sqrt{\frac{3  }{2\pi}}\right) \right)\leq c_k^{\rm ECH}\left( B^4\left(\sqrt{\frac{3  }{\pi}}\right)\right) \quad \text{for any  $k\geq 594$ }.\]
 The rest of relations in  (\ref{eq-ECH-relation}) for $k\leq 593$ can be verified directly by a computer program (for instance, by exhausting all possible choices of $d_i$ in (\ref{eq-ECH-union-ball}) for each fixed $k$). Therefore, the relation (\ref{eq-ECH-relation}) always holds and $a$ can be taken as $ \sqrt{\frac{3  }{2\pi}}$.
 
Since the weight sequence of the  toric domain $X_{\mathrm{int}(A')}$ is given by  $(3;1,1,1,1,1,1)$, we calculate the ECH capacity of $X_{\mathrm{int}(A')}$ using Lemma 3.3.1 in \cite{Mcd25} as 
\begin{equation}\label{eq-first-ech}
c_k^{\rm ECH}(X_{\mathrm{int}(A')})  \coloneqq 
\min_{\,k = l - k_1 - \cdots - k_6}   
\left(  c_l^{\rm ECH}\!\left(B^{4}\!\left(\sqrt{\tfrac{3}{\pi}}\right)\right)  
 -\sum_{i=1}^6 c_{k_i}^{\rm ECH}\!\left(B^{4}\!\left(\sqrt{\tfrac{1}{\pi}}\right)\right)\right).
\end{equation}
Especially, by taking $k=3$ and $(l,k_1\ldots,k_6)=(9,1,1,1,1,1,1)$, we conclude that
$c_3^{\rm ECH}(X_{\mathrm{int}(A')})\leq 3$. If there exists a ball $B^4(b)$ embedded into $X_{\mathrm{int}(A')}$ with $b> \sqrt{\frac{3  }{2\pi}} $,  by (\ref{eq-ECH-ball}), we have $c_3^{\rm ECH}(B^4(b))>3$. However,  by the monotonicity of ECH capacities (cf.~Theorem 1.1 in \cite{Hut11}), we have
\[c_3^{\rm ECH}(B^4(b))\leq c_3^{\rm ECH}(X_{\mathrm{int}(A')})\leq 3,\]
which is a contradiction. Therefore  $c_{\rm Gr}(X_{\mathrm{int}(A')})\leq \frac{3}{2}$. In conclusion, $B^4\left(\sqrt{\frac{3  }{2\pi}}\right)$ is the largest ball embedded into $X_{\mathrm{int}(A')}$, so $c_{\rm Gr}(X_{\mathrm{int}(A')})=c_{\rm Gr}(\T^2\times {\mathrm{int}(A')})=\frac{3}{2}.$

 Next we compute the first ECH capacity of $ X_{\mathrm{int}(A')}$.   By taking $(k,l,k_1,\ldots,k_6)=(1,9,3,1,1,1,1,1)$ in (\ref{eq-first-ech}),   we conclude that
   $c_1^{\rm ECH}(X_{\mathrm{int}(A')})\leq 2$. For any ball-normalized capacity $c$, by definition of the Gromov width $c_{\rm Gr}\leq c$. In particular, we have (since $c_1^{\rm ECH}$ is a ball-normalized capacity)
   \begin{equation}\label{eq-first-ECH-relation}
       c_1^{\rm ECH}(X_{\mathrm{int}(A')})=c_1^{\rm ECH}(\T^2\times \mathrm{int}(A'))\geq c_{\rm Gr}(\T^2\times \mathrm{int}(A'))\geq \frac{3}{2}.
   \end{equation} 
Note that by (\ref{eq-ECH-ball}), if $a\in \N$, then
\[c_k^{\rm ECH}\left(B^4\left(\sqrt{\frac{a}{\pi}}\right)\right)=ad\]
is always an positive integer for any $k\geq 1.$
For  the right-hand side of   (\ref{eq-first-ech}), the minimum is taken from a linear combination of integers, which is again an integer.   Thus   $c_1^{\rm ECH}(X_{\mathrm{int}(A')})$ is   an integer. Combine  $c_1^{\rm ECH}(X_{\mathrm{int}(A')})\leq 2$ with  (\ref{eq-first-ECH-relation}), we deduce that 
\[c_1^{\rm ECH}(X_{\mathrm{int}(A')})=c_1^{\rm ECH}(\T^2\times \mathrm{int}(A'))=2.\]
Since $ \T^2\times \mathrm{int}(A )\subset \T^2\times A  \subset \T^2\times  (1+\ep)\cdot \mathrm{int}(A ) $ for any $\ep>0$, we conclude that 
\[c_1^{\rm ECH}(\T^2\times A )=c_1^{\rm ECH}(\T^2\times \mathrm{int}(A ))=c_1^{\rm ECH}(\T^2\times \mathrm{int}(A'))=2. \]
Therefore, $c_{\rm Gr}(\T^2\times A) \neq c_1^{\rm ECH}(\T^2\times A)$. 
\end{ex}

 As in the Euclidean case, coincidence of ball-normalized symplectic capacities are verified  not only on certain classes of convex domains, but also on some subfamily of dynamically convex domains, including monotone toric domains. Similarly, in cotangent bundle of torus we can also establish coincidence of normalized symplectic capacities on product domains that are not fiberwise convex.
By    directly applying   Proposition \ref{prop-toric} and  Corollary \ref{cor-capacity-BM}, we argue as follows.
\begin{proof}[Proof of Theorem \ref{thm-ball-normalized}]
By Theorem 1.2  in \cite{DR232}, all  ball-normalized capacities coincide on monotone toric domains. Then by Corollary \ref{cor-capacity-BM}, for any  ball-normalized symplectic capacity $c$ and  toric domain $X_\Omega$, $c(X_\Omega)=c(\T^n\times \Omega)$ therefore  ball-normalized capacities coincide on $(\T^n\times \Omega,\omega_{\rm can})$.
 By Theorem 2 in \cite{GMV22}, all  cube-normalized capacities agree on monotone toric domains. Similarly, by Corollary \ref{cor-capacity-BM}, we conclude that all cube-normalized capacities coincide on $(\T^n\times \Omega,\omega_{\rm can})$.
\end{proof}
 Finally, let us end this section with an interesting point. The coincidence of ball-normalized capacities on $(\T^n \times A,\omega_{\rm can})$ where $A$ is convex and centrally symmetric is in fact a direct consequence of the  coincidence of ball-normalized capacities on Lagrangian product $A\times A^*\subset (T^*\R^n,\omega_{\rm can})$ where $A^*$ is the dual of $A$. Equivalently, one may phrase this result in contrapositive form as follows.
\begin{prop}\label{prop-counterexample} For $\T^n \times A \in \mathcal T_{T^*\T^n}$ where $A$ is a centrally symmetric convex domain in $\R^n$, if 
\begin{equation} \label{Gr-HZ}
c_{\rm Gr}(\T^n\times A)< 2 \mathrm{sys}(\T^n\times \partial A)
\end{equation}
then, under the identification $\Psi: (T^*\R^n,\omega_{\rm can})\simeq (\R^{2n},\omega_{\rm std})$ defined by 
\[\Psi (q_1,\ldots, q_n, p_1,\ldots ,p_n) =  (q_1,p_1,\ldots,q_n,p_n)  \quad \text{for any $(q_1,\ldots, q_n, p_1,\ldots p_n) \in T^*\R^n$,}\]
the image $\Psi(A \times A^*)$ is a centrally symmetric convex domain in $(\R^{2n}, \omega_{\rm std})$ where the strong Viterbo conjecture fails, explicitly, $c_{\rm Gr}(\Psi(A \times A^*)) \neq c_{\rm HZ}(\Psi(A \times A^*))$. Here, $c_{\rm HZ}$ denotes the Hofer-Zehnder capacity (for more details, see \cite{HZ90}). \end{prop}
\begin{proof}[Proof of Proposition \ref{prop-counterexample}]
    By equation (3) in \cite{Ben24}, the Gromov width $c_{\rm Gr}$ satisfies
    \begin{equation}\label{eq-gromov-width-lower-bound}
        c_{\rm Gr}(\T^n\times A)\geq \frac{\sys(\T^n\times \partial A)\cdot c_{\rm Gr}(A\times A^*)}{2}.
    \end{equation} 
 By  Theorem 1.7 in \cite{AKO14}, the Hofer-Zehnder capacity satisfies $ c_{\rm HZ}(A\times A^* )=4.$ 
  Therefore, if $c_{\rm Gr}(\T^n\times A)< 2\sys(\T^n\times \partial A) ,$ then the inequality (\ref{eq-gromov-width-lower-bound}) above would imply 
  \[c_{\rm Gr}(A\times A^*)< 4=c_{\rm HZ}(A\times A^*).\]
    Consequently we have
  \[c_{\rm Gr}(\Psi(A \times A^*))=c_{\rm Gr}(A\times A^*) < c_{\rm HZ}(A\times A^*)= c_{\rm HZ}(\Psi(A \times A^*)),\]
  which completes the proof.
\end{proof}

\begin{remark} Via a similar argument in the proof of Theorem \ref{thm-normalized}, one could systematically estimate the Gromov width of $\T^2\times A$ for many centrally symmetric convex domain $A$. So far, we haven't found any example where $c_{\rm Gr}(\T^2\times A)< 2 \mathrm{sys}(\T^2\times \partial A)$ holds. \end{remark}

\appendix

\section{Estimation of systolic ratio} \label{sec-app} 

%First we show the relation between inclusion and minimal action for certain star-shaped domains.
 
\begin{prop}\label{prop-convex-hull}
  Given star-shaped  domains $U\subset V\subset \R^2$ where $V$ is convex, the minimal action   satisfies the following relation:
  \[\sys(\T^2\times \partial U)\leq \sys (\T^2\times \partial V).\]
\end{prop}
\begin{proof}
    Consider the convex hull of $U$, namely the smallest convex domain that contains  $U$, denoted by $\hat{U}.$ First we prove that $\sys(\T^2\times \partial U)\leq \sys (\T^2\times \partial \hat{U})$.   For any  point $p\in \partial U$ but $p\notin \partial \hat{U}$,
   as shown in Figure \ref{fig-convex-hull}, by taking the convex hull of $U$, there exists $q_1,q_2\in \partial U$ such that

\noindent (i) $q_1,q_2\in \partial \hat{U}$ and $p$ lies strictly in the segment in $U$ connecting $p_1$ and $p_2;$

\noindent (ii) the line passing through $q_1,q_2$ is a supporting line for $\hat{U}$ such that  the line is tangent to $\partial U$ at $q_1$, $q_2.$ 
   
  By (\ref{eq-action}), the closed Reeb orbits on $(\T^2\times \partial U,\lambda_{\rm can}|_{\T^2\times \partial U})$ correspond  to points on $\partial U$ with (rescaled) outer normal vector $c_pn(p)\in \Z^2\backslash\{0\}$. By construction, the line segment $[q_1,q_2]\subset \partial \hat U$
replaces the concave arc of $\partial U$ between $q_1$ and $q_2$.
Consequently, the outward normal vectors of $\partial \hat U$ along the  
$[q_1,q_2]$ are constant and coincide with the normal direction of the
supporting line. In particular, any (rescaled) outer normal vector $c_pn(p)\in \mathbb{Z}^2$
arising from points $p\in \partial U$ lying strictly between $q_1$ and
$q_2$ does not occur as a normal vector of $\partial \hat U$. Thus, the
corresponding closed Reeb orbits on $\T^2\times \partial U$ do not persist on
$\T^2\times \partial \hat U$. As a result, the action spectra (defined in (\ref{eq-dfn-spec})) satisfy:
\[\mathrm{Spec}(\T^2\times \partial \hat{U})\subset \mathrm{Spec}(\T^2\times \partial U).\]
Therefore by definition of the minimal action we have
\begin{equation}\label{eq-sys-convex-hull}
    \sys(\T^2\times \partial U)\leq \sys(\T^2\times \partial \hat{U}).
\end{equation}

    \begin{figure}
        \centering
        \includegraphics[width=0.4\linewidth]{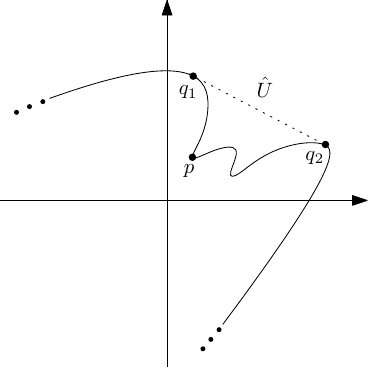}
        \caption{Domain $U$ and its convex hull $\hat{U}$ in local picture.}
        \label{fig-convex-hull}
    \end{figure}

  For any star-shaped  domains $U\subset V\subset \R^2$ where $V$ is convex, by definition $\hat{U}\subset V$. By (\ref{eq-sys-convex-hull}), it suffices to prove  that
  \[\sys (\T^2\times \partial\hat{U})\leq \sys (\T^2\times \partial V).  \]
  For any closed Reeb orbit of $(\T^2\times \partial V,\lambda_{\rm can}|_{\T^2\times \partial V})$ with action $T$, by (\ref{eq-Reeb}), it is contained in torus fiber $\T^2\times \{p\}$ for some $p\in \partial V.$
  Then by (\ref{eq-action}), the action of the orbit  is 
  \[T=c_p (p_1n_1(p)+p_2n_2(p))\]
  where $n(p)=(n_1(p),n_2(p))$ is the outer unit normal vector at $p=(p_1,p_2)\in \partial A$.  Denote   $c_pn(p)=(m_1,m_2)\in \Z^2$. The line $m_1x+m_2y-T=0$ is tangent to $\partial V$ at $p\in \partial V$ since $m_1p_1+m_2p_2-T=0$ and  the normal vector at $p$ is  perpendicular to the line.

  Define the support function of $U$ in the direction $(m_1,m_2)$ by 
  \[T': = \max_{(x,y)\in \hat{U}}(m_1x+m_2y).\]
  Since $\hat{U}$ is compact and convex, this maximum is attained at some point $p'=(p_1',p_2')\in \partial \hat{U} $ and the line $m_1x+m_2y-T'=0$ is a tangent line of $\hat{U}.$ Now, since $\hat{U}\subset V$ and $m_1x+m_2y\leq T$ on $V$, we have
  \[T'\coloneqq \max_{(x,y)\in \hat{U}}(m_1x+m_2y)\leq\max_{(x,y)\in V}(m_1x+m_2y)=T\]
   Moreover, the line $m_1x+m_2y-T'=0$ is   tangent to  $\hat{U} $ with $T'\leq T.$
  Then the  (rescaled) outer normal vector at $p'$ is still $(m_1,m_2)$ with $m_1p_1'+m_2p_2'=T'\leq T$. By (\ref{eq-action}), there exist a closed Reeb orbit in $\T^2\times \{p'\}$ with action $T'.$
  Thus $\sys(\T^2\times \partial V)\leq T'\leq T$. Since we can choose  any $T\in \mathrm{Spec}(\T^2\times \partial V)$, we conclude that $\mathrm{sys}(\T^2\times \partial \hat{U})\leq \sys(\T^2\times \partial V).$
\end{proof}
\begin{remark}
    Note that Proposition \ref{prop-convex-hull} does not always hold  for  star-shaped domains $U\subset V\subset \R^2$ if $V$ is  non-convex. For example, let $U$ be a polygon with vertices $(1,1),(1,-1),(-1,-1),(-1,1)$. For any $1>\ep>0$, let $V_\ep$ be  a non-convex polygon  with vertices:
    \[(1,1),(\ep,1),(\ep,1+\ep),(-\ep,1+\ep),(-\ep,1),(-1,1),(-1,-1),(1,-1)\]
    so that $U\subset V_\ep$. While $U$ is convex and it contains the polygon $B_1(1)$, by Proposition \ref{prop-convex-hull}  we have
    \[\sys(\T^2\times \partial U) \geq \sys(\T^2\times \partial B_1(1)) =1 \]
 where the last inequality comes from the  paragraph under (\ref{eq-sys-ineq}). However, by considering the outer unit normal vector $(1,0)$ at $\left(\ep,1+\frac{\ep}{2}\right)$ and the calculation in (\ref{eq-action}), there exists a closed Reeb orbit with action $\ep$ in $\T^2\times \{\left(\ep,1+\frac{\ep}{2}\right)\}.$
    Therefore   $\sys(\T^2\times \partial V_\ep)\leq \ep<1$. Moreover, $V_\ep$ can be chosen arbitrarily  close to $U$ in the $C^0$-sense when $\ep>0$ is sufficiently small.
\end{remark}
Based on the proposition above, we obtain the following corollary controlling the systolic ratio of domains in terms of the area ratio between the  convex hull of the domain and the domain itself.
\begin{cor}\label{cor-bound-sys}
Let $U\subset \R^2$ be a compact star-shaped domain and denote by $\hat{U}$ its convex hull.
Then
\[
\rhoT(\T^2\times U)
\le
\frac{\mathrm{area}(\hat{U})}{3 \mathrm{area}(U)}.
\]
\end{cor}
\begin{proof}
    By definition of $\rhoT$ in (\ref{eq-codisc-sym-sys})  we have
    \begin{equation*}
        \begin{split}
            \rhoT(\T^2\times U)&=\frac{\sys(\T^2\times \partial U)^2}{\mathrm{vol}(\T^2\times \partial U, \lambda_{\rm std}|_{  \T^2\times \partial U})}\\&\leq \frac{\sys(\T^2\times \partial \hat{U})^2}{\mathrm{vol}(\T^2\times \partial U, \lambda_{\rm std}|_{  \T^2\times \partial U})}\\&=\frac{2 \mathrm{area}(\hat{U})}{ 2\mathrm{area}(U)}\cdot\frac{\sys(\T^2\times \partial \hat{U})^2}{\mathrm{vol}(\T^2\times \partial \hat{U}, \lambda_{\rm std}|_{  \T^2\times \partial U})}\leq \frac{\mathrm{area}(\hat{U})}{3 \mathrm{area}(U)}
        \end{split}
    \end{equation*} 
    where the first inequality is due to Proposition \ref{prop-convex-hull}; the last inequality is due to the upper bound of systolic ratio for codsic bundle of flat torus as illustrated in  Theorem IV of \cite{ABT16} (see also (i) in Remark \ref{rmk-1-3}).
\end{proof}

Moreover, for certain cases, taking convex hull does {\it not} increase the systolic ratio $\rhoT$, therefore, the systolic ratio upper bounded persist.

\begin{ex}\label{ex-convex-hull}
    For any star-shaped domain $A\subset \R^2$ such that the convex hull of $A$ is the polygon with vertices $(1,0),(0,1), (-1,0),(0,-1), $ we have
    \[\rhoT(\T^2\times A)\leq \frac{1}{4}.\]
 First, observe that $A$  contains the  points $(\pm 1,0),(0,\pm 1)$. In particular, 
 \[(1,0),(0,1)\in A\cap \R^2_{\geq 0}.\]
 Then restricting to the first quadrant, we parametrize the arc $\partial A\cap \R^2_{\geq 0}$ by polar angle $\theta\in [0,\frac{\pi}{2}]$, expressed as $\gamma(\theta)=(x(\theta), y(\theta))$  with $\gamma(0)=(1,0)$ and $\gamma(\frac{\pi}{2})=(0,1).$ Then consider the function $x(\theta)+y(\theta)$, by requirement of $A$, we have $x(\theta)+y(\theta)\leq 1$. By compactness there exists $\theta_1\in [0,\frac{\pi}{2}]$ such that
 \[a_1: = x(\theta_1)+y(\theta_1)=\min_{\theta\in [0,\frac{\pi}{2}]} \{x(\theta)+y(\theta)\}.\]
 Since $(1,0)\in \gamma $, we have $a_1\leq 1$ with equality holds if and only if $x(\theta)+y(\theta)\equiv 1$ for any $\theta\in [0,\frac{\pi}{2}]$. If $a_1=1$, then the line segment joining $(1,0),(0,1)$ is exactly the boundary of $\partial A$ in the first quadrant therefore tangent to $(x(\theta_1),y(\theta_1)).$ If $a_1<1$, since $\theta_1$  takes the minimum, the line $x+y=a_1$ is tangent to $\partial A$ at $(x(\theta_1),y(\theta_1)).$    Moreover, in either case,  $x(\theta)+y(\theta)\geq a_1$ for any $\theta\in [0,\frac{\pi}{2}]$, that is, the line segment joining $(a_1,0)$ and  $(0,a_1)$  is contained inside $A$.  
 Therefore the triangle $\Delta(a_1)$ with vertices $(0,a_1),(a_1,0),(0,0)$ is  contained in $A$. 

 Due to the tangency to  the line $x+y=a_1$, the (rescaled) outer normal vector at  $(x(\theta_1),y(\theta_1))\in \partial A$ is $(1,1)\in \Z^2.$
 By (\ref{eq-action}) there exists a closed Reeb orbit with action $a_1$ in $\T^2\times \{(x(\theta_1),y(\theta_1))\} $.
 Therefore $\sys(\T^2\times \partial A)\leq  a_1$.
 
 Similarly, we get other three triangles $\Delta(a_i)$ for other three quadrants with $i=2,3,4$. Without loss of generality assume  $a_1=\min_{1\leq i\leq 4} \{a_i\} $.  Then we have
 \[\rhoT(\T^2\times \partial A)= \frac{\sys(\T^2\times \partial A)^2}{\mathrm{vol}(\T^2\times \partial A, \lambda_{\rm std}|_{\partial \T^2\times \partial A})}\leq \frac{ a_1^2}{2\sum_{i=1}^4\mathrm{area}(\Delta(a_i))}\leq \frac{a_1^2}{4a_1^2}=\frac{1}{4}\]
 where the second inequality is due to the fact  that $A$ contains the triangles $\Delta_i$, $1\leq i\leq 4.$
\end{ex}

\section{Construction of embedding} \label{sec-app-B} 
In this section we construct the symplectic embedding $\sigma$ appearing in (\ref{eq-area-pres}). For simplicity, we assume $R=1$ and $\sigma\colon B^2(  1-\ep )\hookrightarrow SD(1)$ where $SD(1)$ is the slit disc defined in (\ref{eq-slit-disc}). For $\ep>0$ small enough we consider the following   two functions:
\begin{equation}\label{eq-alpha}
\alpha_\ep\colon \left(-\infty,1-\frac{\ep}{2} \right) \to [0,1],\quad \alpha_\ep(x)\coloneqq \left\{
\begin{aligned}
& 1 \quad \text{if }  x \in \left(0,1-\frac{\ep}{2} \right) \\
&  0  \quad  \text{if }  x \in \left(-\infty, -\frac{\ep^3}{2} \right)   \\  & \text{linear} \quad \text{if otherwise.}
\end{aligned}
\right.
\end{equation}
 and 
 \begin{equation}\label{eq-beta}
 \beta_\ep\colon \R\to  [-\ep^9,\ep^9],\quad  \beta_\ep(y)\coloneqq \left\{
\begin{aligned}
& -y \quad \text{if }  y \in \left(-\ep ^9,\ep ^9\right) \\
&  0  \quad  \text{if }  y \in \left(-\infty, -\frac{\ep^6}{2}  \right) \cup\left( \frac{\ep^6}{2}, +\infty\right) \\  & \text{linear} \quad \text{if otherwise.}
\end{aligned}
\right.
\end{equation}
\begin{figure}[h]
    \centering
    \includegraphics[width=0.95\linewidth]{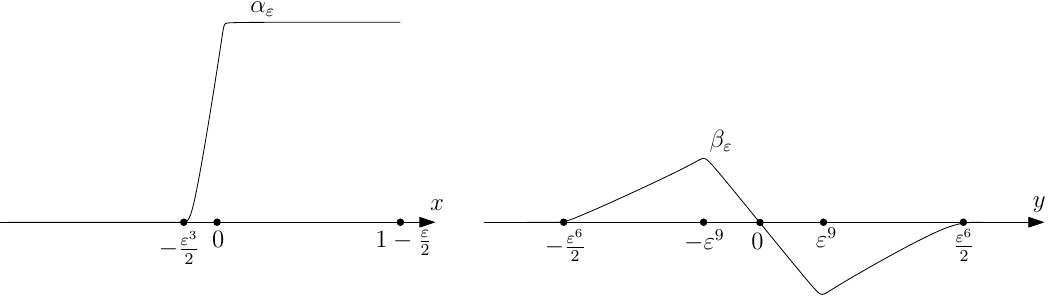}
    \caption{Function $\alpha_\ep$ and $\beta_\ep$}
    \label{fig:alphabeta}
\end{figure}
Then we  apply $o(\ep^9)$-small perturbation to $\alpha_\ep,\beta_\ep$ in neighborhoods of the corner points to smooth them,  still denoted them as $\alpha_\ep,\beta_\ep$, separately.  For an illustrative picture, see Figure \ref{fig:alphabeta}.

Next, we define the Hamiltonian function
\[ H\colon \left(-\infty,1-\frac{\ep}{2} \right)\times \R\to \R,  \quad H\coloneqq \alpha_\ep(x) \cdot \beta_\ep (y) .\]
Note that by  definition of $\alpha_\ep,\beta_\ep$, $H$ is in fact  supported in the strip $(-\ep^3,1-\frac{\ep}{2}]\times (-\ep^6,\ep^6).$
 On $(\R^2,\omega_{\rm std}=dx\wedge dy)$, the Hamiltonian equation of the Hamiltonian vector field $X_H$ is given by 
 \[\iota_{X_H}\omega_{\rm std}=-dH.\]
   In our case, 
   \begin{equation}\label{eq-Ham-vf}
       X_H(x,y)=\left(\frac{\partial H}{\partial y},-\frac{\partial H}{\partial x}\right)=\left(\frac{\partial \beta_\ep(y)}{\partial y}\alpha_\ep(x),-\frac{\partial \alpha_\ep (x)}{\partial x}\beta_\ep(y)\right).
   \end{equation} 
 We denote the Hamiltonian flow of $X_H$  as $\{\varphi^t\}_{t\in [0,1]}$.  
For any given point $(x,y)\in (-\ep^3,1-\frac{\ep}{2})\times (-\ep^6,\ep^6)$, we can discuss.

\medskip

\noindent Case I: If $y\in (-\ep^9,\ep^9)$, then there are two possibilities of $x:$

\noindent (i) if $x\in (0,1- \frac{\ep}{2})$, then $X_H(x,y)=(-1,0)$ so that there exists time $t\in (0,1)$ such that $\varphi^t(x,y)=(0,y)$. Note that for any $x\in (-\ep^3,\ep^3)$, $\alpha'_\ep(x)\leq \frac{2}{\ep^3}$. Therefore we estimate 
\[|X_H(x,y)|=\sqrt{\left(\frac{\partial \beta_\ep(y)}{\partial y}\alpha_\ep(x) \right)^2+\left(\frac{\partial \alpha_\ep (x)}{\partial x}\beta_\ep(y)\right)^2}\leq \sqrt{1\cdot\ep^6+\frac{4}{\ep^6}\cdot \ep^{12}}\leq 3\ep^3.\]
Then denote $z=(0,y)$, we discuss as follows. If $\varphi^{1-t}(z)$ is in  $(-\ep^3,\ep^3)\times (-\ep^6,\ep^6) $, then the estimation above implies
\begin{equation}\label{eq-estimation-norm}
     |\varphi^{1-t}(z)-z|=\left|\int_{0}^{1-t}\frac{d}{ds}\varphi^s(z)ds\right|\leq \int_{0}^{1-t}\left| X_H(\varphi^s(z))\right|ds\leq (1-t)\cdot 3\ep^3 \leq 3\ep^3.
\end{equation}
 If $\varphi^{1-t}(z)$  has $x$-coordinate $\geq \ep^3,$  by (\ref{eq-Ham-vf}), $x$-coordinate increases along the flow only if  $y\notin (-\ep^9,\ep^9)$, therefore $\varphi^{1-t}(z) \notin [\ep^3,1-\frac{\ep}{2})\times (-\ep^9,\ep^9)$. Then  we can
 estimate   
\begin{equation}\label{eq-Ham-vf-norm}
\begin{split}
      |X_H(x,y) |&=\sqrt{\left(\frac{\partial \beta_\ep(y)}{\partial y}\alpha_\ep(x_0) \right)^2+\left(\frac{\partial \alpha_\ep (x_0)}{\partial x}\beta_\ep(y)\right)^2}\\&\leq \sqrt{\frac{\ep^{18}}{(\frac{\ep^6}{2}-\ep^9)^2}\cdot1+ \frac{4}{\ep^6}\cdot\ep^{12}}\leq 3\ep^3.  
\end{split}
\end{equation}
In this case, similarly, (\ref{eq-estimation-norm}) holds for $\varphi^{1-t}(z)$.
As a flow, we have $\varphi^1(x,y)=\varphi^{1-t}\circ\varphi ^t(x,y) $. Therefore by triangle inequality:
\begin{equation}\label{eq-up-bound-sigma}
\begin{split}
     |\varphi^1(x,y)|^2&= |\varphi^{1-t}(z)|^2\leq(|\varphi^{1-t}(z)-z|+|z|)^2  \\& \leq (3\ep^3+|z|)^2< (3\ep^3+\ep^{9})^2<10\ep^6
\end{split}
\end{equation} 
for any $x\in (0,1- \frac{\ep}{2} )$ and $y\in (-\ep^9,\ep^9)$.

\medskip

\noindent (ii) If $x \in \left (-\ep^3,0 \right]$, then by the same argument in (i), we have $|\varphi^1(x,y)|<10\ep^6.$

\medskip 
 
\noindent Case II: If $y\in [\ep^9,\ep^6)$,  note that  for any $x\in (-\ep^3,1-\frac{\ep}{2}),$  $-\frac{\partial \alpha_\ep(x)}{\partial x}\beta_\ep(y)\geq 0$ therefore flowing by $\varphi^1$ does not decrease $y $ and $\varphi^1(x,y)$ admits $y$-coordinate in  $(\ep^9,\ep^6
)$.  Meanwhile, for any $y\in [\ep^9,\ep^6)$  by (\ref{eq-Ham-vf-norm})  we can estimate
\[|\varphi^1(x,y)-(x,y)| \leq  \int_{0}^1|X_H(\varphi^s(x,y))|^2ds \leq 9\ep^6.\] 
By triangle inequality and the inequality above we have
\begin{equation}\label{eq-up-bound-sigma-2}
    |\varphi^1(x,y)|^2\leq (|\varphi^1(x,y)-(x,y)|+|(x,y)|)^2\leq  (9\ep^6+|(x,y)|)^2\leq |(x,y)|^2+ 20\ep^6
\end{equation} 
where we use $|(x,y)|<1 $ and $81\ep^{12}\ll 2\ep^6$ in the last inequality.
 Similar  conclusion holds for $y\in (-\ep^6,-\ep^9].$ We can draw the following illustrative picture of the  Hamiltonian flow $\{\varphi^t\}_{t\in [0,1]}$:
 \begin{figure}[h]
     \centering
     \includegraphics[width=0.6\linewidth]{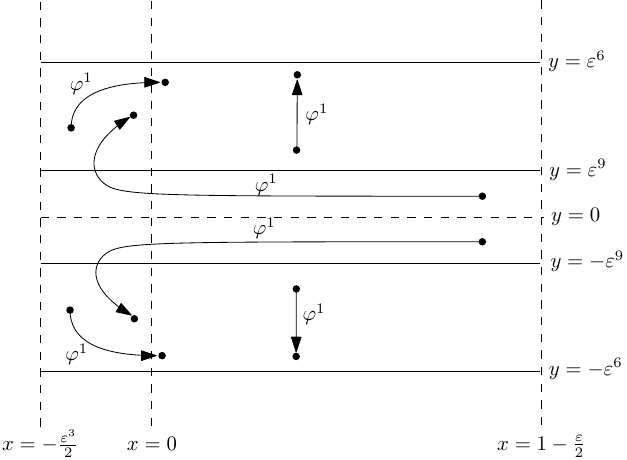}
     \caption{Flow by   $\varphi^1$   at each point.}
     \label{fig:Ham_flow}
 \end{figure}

Then we consider the restriction given  by $\varphi^1|_{B^2(1-\ep)}$.  It is  area-preserving since $\varphi^1$ is a Hamiltonian diffeomorphism. Moreover, for any point in the intersection
\[(x,y)\in S_\ep\coloneqq B^2(1-\ep)\cap  \left(\left(-\ep^3,1-\frac{\ep}{2}\right)\times (-\ep^6,\ep^6)\right),\]
we have $x\leq 1-\ep$. Therefore the above inequalities (\ref{eq-up-bound-sigma}) and  (\ref{eq-up-bound-sigma-2}) imply that for any $(x,y)\in S_\ep$, we have
\begin{equation}\label{eq-up-bound-final}
    |\varphi^1(x,y)|^2\leq |(x,y)|^2+20\ep^6=x^2+y^2+20\ep^6.
\end{equation} 
Since on $B^2(1-\ep)\backslash S_\ep$, $\varphi^1$ is identity we conclude that (\ref{eq-up-bound-final}) holds for every $(x,y)\in B^2(1-\ep).$
 
By  Case I  above, we have $\varphi^1|_{B^2(1-\ep)}(x,0) \subset B^2(4\ep^{3})$ for any $x\in (-\ep^3,1-\frac{\ep}{2}).$
Therefore, by considering the shifted map $\sigma(x,y) : = \varphi^1|_{B^2(1-\ep)}(x,y)-(4\ep^3,0)$,  we obtain an embedding
\[\sigma: B^2(1-\ep)\to SD(1).\]
By triangle inequality and (\ref{eq-up-bound-final}) we have
\[|\sigma (x,y)|^2=|\varphi^1|_{B^2(1-\ep)}(x,y)-(4\ep^3,0)|^2\leq |\varphi^1|_{B^2(1-\ep)}(x,y)|^2+8 \ep^3+16 \ep^6\leq x^2+y^2+9\ep^3.\]
Then by replacing $\ep$ with $\frac{\ep}{3}$ in the   above construction, we obtain  our desired $\sigma.$

\subsection*{Acknowledgements}
We thank useful communications with Alberto Abbondandolo, Simon Vialaret, and  Pedro Salom\~ao around Question \ref{que-sys-ratio} in this paper. The first author is partially supported by
National Key R\&D Program of China No.~2023YFA1010500, NSFC No.~12301081, 
NSFC No.~12361141812, and NSFC No.~12511540054.

\bibliographystyle{amsplain}   % or "abbrv", "unsrt", "alpha", etc.
\bibliography{references}   % your .bib file (without extension)

\end{document}